\documentclass[10pt]{article} 
\usepackage[margin=1in]{geometry}
 
\usepackage{listings}
\usepackage{geometry}

\usepackage{amsmath,amsthm,amsfonts,amssymb,amscd}
\usepackage[colorlinks=true, pdfstartview=FitV, linkcolor=blue,citecolor=blue, urlcolor=blue]{hyperref}
\usepackage[usenames,dvipsnames]{color}
\usepackage{times}
\usepackage{xcolor}

\usepackage{mathtools}
\usepackage{mathabx}
\usepackage{float}
\usepackage{makeidx}
\usepackage{stmaryrd}
\usepackage{mathrsfs}
\usepackage{emptypage}
\usepackage{xfrac}
\usepackage{bbm}

\renewcommand{\epsilon}{\varepsilon}

\usepackage{graphicx}
\newcommand{\p}{\ensuremath{\partial}}

\newcommand{\mc}{\ensuremath{\mathcal}}

\definecolor{labelkey}{rgb}{0,0,1}

\def\les{\lesssim}

\def\eps{\varepsilon}
\renewcommand*{\div}{\ensuremath{\mathrm{div\,}}}

\newcommand{\R}{\mathbb{R}}

\newcommand{\Z}{\mathbb{Z}}

\renewcommand*{\tilde}{\widetilde}
\renewcommand*{\hat}{\widehat}
\renewcommand*{\bar}{\overline}

\newcommand{\T}{{\mathbb T}}

\newtheorem{theorem}{Theorem}[section]
\newtheorem{lemma}[theorem]{Lemma}
\newtheorem{prop}[theorem]{Proposition}
\newtheorem{proposition}[theorem]{Proposition}

\theoremstyle{definition}

\newtheorem{remark}[theorem]{Remark}

\numberwithin{equation}{section}

\def\p{\partial}

\def\f1r{{\frac{1}{r}}  }
\def\p{\partial}

\def\f1r{{\frac{1}{r}}  }

\allowdisplaybreaks[4]

\title{
Non-radial implosion for the defocusing nonlinear Schr\"odinger equation in $\mathbb{T}^d$ and $\mathbb{R}^d$}
 \author{G. Cao-Labora, J. G\'omez-Serrano, J. Shi, G. Staffilani}

\date{} %
\begin{document}

\maketitle
 \begin{abstract}
In this paper we construct smooth, non-radial solutions of the defocusing nonlinear Schr\"odinger equation that develop an imploding finite time singularity, both in the periodic setting and the full space. 
 \end{abstract}
\tableofcontents

\section{Introduction}

\subsection{Historical Background}

We consider the  defocusing nonlinear Schr\"odinger (NLS) equation  given by
\begin{equation} \label{eq:NLS}
i\p_t v + \Delta v - v|v|^{p-1} = 0
\end{equation}
for some odd integer $p$. Here  $v$ is a complex valued function. Other than the  mass, which is the  $L^2$ norm square of $v$, the Hamiltonian given by
\begin{equation*}
E(v) = \frac{1}{2} \int_{\mathbb R^d} | \nabla v |^2 + \frac{1}{p+1} \int_{\mathbb R^d} |v|^{p+1}
\end{equation*}
is also conserved. 
The natural scaling of this equation is given by $v_\lambda (x, t) = \lambda^{2/(p+1)} v(\lambda x, \lambda^2 t)$, and if $v$ is a solution, then $v_\lambda $  is also a solution for any $\lambda > 0$. Note that if we consider the periodic setting, i.e. the space variable $x$ belongs to a certain torus $\T^d$, then  the torus needs to be rescaled appropriately. The critical exponent $s_c$ for which the homogeneous Sobolev norm  $\dot H^{s_c}(\R^d)$  remains invariant under scaling is given by $s_c = \frac{d}{2} - \frac{2}{p-1}$. The goal of this paper is to show that even in the defocusing case,  we can exhibit non-radial or periodic  blowup  for the initial value problem associated to equation \eqref{eq:NLS} in the energy supercritical regime, namely when the initial data is in $H^s$ and   $s>s_c > 1$. 

Before stating precisely the main results, we will touch upon the history of well-posedness and blow up for the initial value problem  \eqref{eq:NLS} both in $\R^d$ and $\T^d$. At this point the literature is vast hence here we will recall only the few works that are more strictly related to the main results we present below.
Let us first review the results in $\R^d$. In the subcritical case  $s>s_c$, local well-posedness was originally proved by Ginibre and Velo \cite{Ginibre-Velo:class-nls-i-cauchy} using Strichartz estimates, see also \cite{Cazenave-Weissler:cauchy-problem-nls-hs, Cazenave:semilinear-nls-book}. Since in this case the proof shows that the time of existence depends on the $H^s$ norm of the initial data, for $s\leq0$,  using  the mass conservation, or for $s\leq 1$, using  conservation of the Hamiltonian, 
by iteration global well-posedness in $L^2$ or $H^1$ follows easily if the data are respectively in those spaces. On the other hand using the $I-$method, which is based on the concept of {\it almost} conserved energies,  in certain cases it is possible to extend the local well-posedness to global also below the conserved integrals, see for example \cite{Colliander-Keel-Staffilani-Takaoka-Tao:almost-conservation-global-rough-solutions-nls}. If $s=s_c$ local well-posedness is still available via Strichartz estimates using the same techniques in the works mentioned above, but in this case the local time of existence depends also on the profile of the initial data and in order to obtain global well-posedness one needs to treat the problem not as a perturbation of the linear one, but as a truly nonlinear system. The first results in this direction were obtained for the energy critical case ($s_c=1$), in the 3D  radial case \cite{Bourgain:gwp-defocusing-critical-nls-radial, Grillakis:nls}. The radial assumption was removed in \cite{Colliander-Keel-Staffilani-Takaoka-Tao:global-wellposedness-energy-critical-nls-R3} and in higher dimensions the problem was solved in \cite{Ryckman-Visan:gwp-scattering-nls-r14, Visan:defocusing-nls-high-dimensions}.
For the mass critical  problem ($s_c=0$),  global well-posedness in the radial case was proved in \cite{Killip-Tao-Visan:cubic-nls-2d-radial,Tao-Visan-Zhang:gwp-scattering-nls-radial-high-dimensions} , and the radial assumption was removed in \cite{Dodson:gwp-scattering-defocusing-nls-d-geq-3, Dodson:gwp-scattering-defocusing-nls-d-eq-1, Dodson:gwp-scattering-defocusing-nls-d-eq-2}. When $s_c\not \in \{0, 1\}$  the problem is more challenging since a priori the $H^{s_c} $ norm of the solution is no longer bounded. First in  \cite{Kenig-Merle:scattering-h12-cubic-defocusing-nls} in 3D and later in \cite{Murphy:defocusing-h12-nls-high-dimensions, Yue:gwp-focusing-nls-t4} in all dimensions, the authors proved that if one assumes the uniform boundedness of the $H^{s_c}$ norm of the solutions then global well-posedness follows. Recently in \cite{Dodson:scattering-defocusing-nls-critical} it was shown that   global well-posedness and scattering  is  available without the a-priori bound in $H^{s_c}$. However the  data belongs not to $H^{s_c}$  but to a certain Besov space that scales like  $H^{s_c} $.  

We now address the global well-posedness in supercritical spaces, namely when $s<s_c$. Before the work of \cite{Merle-Raphael-Rodnianski-Szeftel:implosion-nls} there was strong belief that also in the energy supercritical regime, $s_c>1$,  the a priori boundedness of the $H^{s_c}$ norm should result in global well-posedness. This in fact was conjectured in \cite{Bourgain:problems-hamiltonian-pde}, see also \cite{Kenig-Merle:nondispersive-radial-nlw, Killip-Visan:energy-supercritical-nls}. Numerically such a conjecture was also supported in \cite{Colliander-Simpson-Sulem:numerical-supercritical-nls}. In contrast, in \cite{Tao:blowup-defocusing-nlw,Tao:blowup-supercritical-defocusing-nls} models of equation  very close to NLS or wave equations at supercritical regimes were proved to exhibit blow-up, while other initial value problem of NLS or wave type with only a logarithmic supercriticality were proved to have global large solutions \cite{Struwe:gwp-cauchy-supercritical-nlw-2-dim, Krieger-Schlag:large-global-supercritical-nlw-r3+1, Colombo-Haffter:global-regularity-nlw-slightly-supercritical}.  Finally, results of  norm inflation were obtained for very low supercritical regularity such as in \cite{Christ-Colliander-Tao:asymptotics-modulation-low-regularity-illposedness-defocusing, Thomann:instabilities-supercritical-schrodinger-manifolds}.

We now consider well-posedness in the periodic setting. Before the proof of Strichartz estimates by Bourgain in \cite{Bourgain:fourier-transform-restriction-i-nls} local well-posedness was only known for smooth data, i.e. $s>d/2$, \cite{Kato:cauchy-problem-gkdv,Tsutsumi:weighted-sobolev-spaces-nonlinear-dispersive-wave}. Strichartz estimates, first partially proved for rational tori \cite{Bourgain:fourier-transform-restriction-i-nls} and later for all tori \cite{Bourgain-Demeter:proof-decoupling},  allowed for local well-posedness  to be proved for subcritical cases $s>s_c$. Also in this case global well-posedness follows if $s$ is below the mass or energy regularity by iteration as in $\R^d$. In some cases the $I-$method is still applicable \cite{DeSilva-Pavlovic-Staffilani-Tzirakis:gwp-periodic-nls, Schippa:improved-gwp-mass-critical-nls-tori} and if the torus is irrational some number theoretical observation can be used to attain global well-posedness for the same range of exponents that gives  local well-posedness \cite{Herr-Kwak:strichartz-estimates-gwp-cubic-nls-t2}.  Unfortunately at the critical level $s=s_c$ the Strichartz estimates have a derivative $\log$ loss and even local well-posedness is distinctly harder to obtain. One such example is the energy critical problem in 3D (hence quintic nonlinearity) proved to be first locally well-posed in \cite{Herr-Tataru-Tzvetkov:gwp-critical-nls-small-data-h1-t3} and then globally well-posed in \cite{Ionescu-Pausader:energy-critical-defocusing-nls-t3}. When the critical exponent $s_c\not \in \{0,1\}$ a conditional result similar to that of \cite{Kenig-Merle:nondispersive-radial-nlw, Killip-Visan:energy-supercritical-nls} has been proved in 2D for $s_c=1/2$  in \cite{Yu-Yue:gwp-periodic-quintic-nls}. 

To the best of our knowledge, before our work  there were  no results or conjecture concerning blow up for periodic supercritical defocusing initial  value problems like \eqref{eq:NLS}. In fact even in the focusing case the results are very few. We recall \cite{Ogawa-Tsutsumi:blowup-nls-quartic-potential-periodic} in 1D, see also \cite{Oh:blowup-periodic-nls}, the numerical study in \cite{Sulem-Sulem-Frisch:tracing-complex-singularities}, and \cite{Planchon-Raphael:existence-stability-loglog-critical-nls-domain} where the authors extend the  $\log – \log$ blow-up dynamics in \cite{Merle-Raphael:universality-blowup-l2-nls} to the  $L^2$-critical nonlinear Schr\"odinger equation in a domain.

In this paper we are interested in \textit{imploding} solutions. These are solutions for which the $L^\infty$ norm of the function blows up as $t$ approaches the singularity time. This is in contrast with \textit{shock} solutions, where the $L^\infty$ norm stays bounded and the gradient blows up, which are more ubiquitous in PDE with conservation laws such as compressible Euler or NLS. Specifically, Guderley \cite{Guderley:singularities-radial} was the first to construct radial, self-similar imploding singularities for compressible Euler, though non-smooth. See also further work by Jenssen--Tsikkou 
\cite{Jenssen-Tsikkou:amplitude-blowup-radial-isentropic-euler,Jenssen-Tsikkou:radially-symmetric-non-isentropic-euler}. 

Merle--Raphaël--Rodnianski--Szeftel \cite{Merle-Raphael-Rodnianski-Szeftel:implosion-i} made a breakthrough by constructing the first smooth, radially imploding solutions to the compressible Euler equations. They also extended this work to create imploding singularities for both the compressible Navier-Stokes equations (with decaying density at infinity) \cite{Merle-Raphael-Rodnianski-Szeftel:implosion-ii}, and the energy-supercritical defocusing nonlinear Schrödinger equation \cite{Merle-Raphael-Rodnianski-Szeftel:implosion-nls}. More specifically, they proved that there exists a sequence of  scalings $r_k$ and a sequence of smooth self-similar imploding profiles with such scalings. This was done for almost every value of $\gamma$ (the adiabatic exponent relating the pressure to the density). In \cite{Buckmaster-CaoLabora-GomezSerrano:implosion-compressible} (see also the review paper \cite{Buckmaster-CaoLabora-GomezSerrano:implosion-compressible-review}), the first two authors together with Buckmaster improved the result to cover all cases of $\gamma$. Moreover, they showed the existence of non-decaying imploding singularities for Navier-Stokes. One of the open questions that Merle--Rapha\"el--Rodnianski--Szeftel highlight is the construction of nonradial blow-up since their setting only allows for radial perturbations. In \cite{CaoLabora-GomezSerrano-Shi-Staffilani:nonradial-implosion-compressible-euler-ns-T3R3}, we solved this problem. The difficulty introduced by the nonradiality is deep since it essentially turns every ODE analysis into PDE analysis, which is a lot harder. After our work,
Chen--Cialdea--Shkoller--Vicol \cite{Chen-Cialdea-Shkoller-Vicol:vorticity-blowup-2d-compressible-euler} constructed solutions to the 2D compressible Euler equation for which vorticity exhibits finite time blow-up. Chen \cite{Chen:vorticity-blowup-compressible-euler-rd} extended this analysis to any $d \geq 2$. In the context of other equations, Shao--Wei--Zhang \cite{Shao-Wei-Zhang:blowup-defocusing-nlw,Shao-Wei-Zhang:self-similar-implosion-relativistic-euler} proved the existence of imploding solutions to the nonlinear wave and the relativistic Euler equations. We also point out to the numerical work of Biasi \cite{Biasi:self-similar-compressible-euler}.

\subsection{Main result} \label{sec:12}

In \cite{CaoLabora-GomezSerrano-Shi-Staffilani:nonradial-implosion-compressible-euler-ns-T3R3} we managed to answer a question of \cite{Merle-Raphael-Rodnianski-Szeftel:implosion-ii} concerning the construction of non-radial imploding solutions for compressible Euler and Navier-Stokes equations. Our goal in this paper is to extend the previous construction to the case of the NLS equation \eqref{eq:NLS}. Broadly speaking, we are able to construct finite time imploding singularities with either: 
\begin{enumerate}
    \item smooth periodic initial data (\textbf{Theorem \ref{th:periodic}}).
    \item smooth, non-radially symmetric initial data that does not vanish at infinity (\textbf{Theorem \ref{th:Euclidean}}). 
\end{enumerate}

The proof of Theorems \ref{th:periodic} and \ref{th:Euclidean} starts with the transformation of the defocusing NLS equation into the compressible Euler equation via the Madelung transform \eqref{decompositionofsol}. In order to perform the non-radial stability argument, we are inspired by our previous work \cite{CaoLabora-GomezSerrano-Shi-Staffilani:nonradial-implosion-compressible-euler-ns-T3R3}. We extend the new symmetry-breaking analysis that we developed there from 3D to higher dimensions and use it to perform the linearized non-radial stability estimate.  This also requires a new angular repulsitivity condition on the profile to close the energy estimate. 
As a byproduct, we improve and generalize the techniques used in \cite{Merle-Raphael-Rodnianski-Szeftel:implosion-nls}. The key part of the non-linear stability argument hinges on involving the logarithm of the density, which leads to a new simplified version of the highest-order energy estimate without derivative losses. This energy estimate is new, and is a combination of the compressible Euler energy and the Hamiltonian of the defocusing NLS. 
The result in the periodic setting is obtained by introducing an argument involving a refined cut-off.

\subsection{Setup of the problem}

The equation in hydrodynamical variables is obtained as follows. Let 
\begin{equation}\label{decompositionofsol}
v = \sqrt \rho e^{i\psi}.
\end{equation} From \eqref{eq:NLS}, we have that
\begin{align} 
\nonumber
i \frac{\p_t \rho}{2\sqrt \rho} e^{i\psi} - \p_t \psi \sqrt \rho e^{i\psi} &=
 \left( \sqrt \rho e^{i\psi} \right) \rho^{\frac{p-1}{2}}  - \Delta(\sqrt \rho) e^{i\psi} - \frac{\nabla \rho}{\sqrt \rho} i \nabla \psi e^{i \psi} - \sqrt \rho \div \left( i \nabla \psi e^{i\psi} \right) \\ 
&= 
\sqrt \rho  e^{i \psi} \left(
 \rho^{\frac{p-1}{2}}  - \frac{\Delta(\sqrt \rho)}{\sqrt \rho} - \frac{\nabla \rho}{\rho} i \nabla \psi  - i  \Delta \psi   + |\nabla \psi |^2   
\right).\label{eq:NLS_polar}
\end{align}

Now
\begin{equation*}
\Delta (\sqrt \rho) = \div \left( \frac{\nabla \rho}{2 \sqrt \rho }\right) = \frac{\Delta \rho}{2 \sqrt \rho} - \frac{1}{4} \frac{|\nabla \rho|^2}{\rho^{3/2}}.
\end{equation*}

We first divide \eqref{eq:NLS_polar} by $e^{i\psi}$. After that, since both $\rho$ and $\psi$ are real, we can identify real and imaginary parts, and obtain
\begin{align} \begin{split} \label{eq:NLS_polar2}
\p_t \psi &= 
-\rho^{\frac{p-1}{2}}  + \frac{1}{2\rho} \Delta \rho - \frac{1}{4} \frac{|\nabla \rho|^2}{\rho^2} - |\nabla \psi |^2  \\
\p_t \rho &= 2 \left(
 - \nabla \rho  \cdot \nabla \psi  -   \rho \Delta \psi  \right)
\end{split} \end{align}

In order to see the similarity with the compressible Euler equations, we define $u = \nabla \psi$ and observe that the system \eqref{eq:NLS_polar2} takes the form
\begin{align*}
\p_t u &= -\nabla \left( \rho^{\frac{p-1}{2}} \right) - 2 u \nabla u  + \nabla \left( \frac{ \Delta \rho }{2\rho}- \frac{1}{4} \frac{|\nabla \rho|^2}{\rho^2} \right)\\
\p_t \rho &= -2 u \nabla \rho - 2 \rho \div (u) = -2 \div (u \rho)
\end{align*}

We rescale time via $\p_{t'} = 2\p_t$. We have
\begin{align*}
 \p_{t'} u &= -\frac12 \nabla \left( \rho^{\frac{p-1}{2}} \right) - u \nabla u  + \nabla \left( \frac{ \Delta \rho }{4\rho}- \frac{1}{8} \frac{|\nabla \rho|^2}{\rho^2} \right)\\
\p_{t'} \rho &=  -\div (u \rho)
\end{align*}
which corresponds to the compressible Euler equations for $\gamma = \frac{p+1}{2}$ and forcing $\frac12 \nabla \frac{ \Delta \sqrt \rho }{\sqrt{\rho}}$ (the coefficient in front the pressure term $\nabla \rho^\gamma$ can also be adjusted by rescaling $\rho$ by a constant).

We let $\alpha = \frac{\gamma - 1}{2} = \frac{p-1}{4}$  and postulate the self-similar ansatz
\begin{align}\label{selfsimilarchange}
\psi = \frac{ (T-t)^{\frac{2}{r} - 1} }{r} \Psi \left( \frac{x}{(T-t)^{1/r}}, - \frac{\log (T-t)}{r} \right), \\\nonumber
\rho = \frac{ (T-t)^{\frac{1}{\alpha r} - \frac{1}{\alpha} } }{r} P \left( \frac{x}{(T-t)^{1/r}}, - \frac{\log (T-t)}{r} \right)
\end{align}
in the new variables
\begin{equation*}
s = -\frac{\log(T-t)}{r}, \qquad \mbox{ and } \qquad y = \frac{x}{(T-t)^{1/r}} = x e^s
\end{equation*}
We define accordingly
\begin{equation*}
s_0 = - \frac{\log T}{r},
\end{equation*}
which we will take to be sufficiently large and positive (that is, we will take $T > 0$ to be sufficiently small). Our new time variable will be $s\geq s_0$ and our new space variable is $y\in \mathbb{R}^{d}$ or $y \in e^{s}\mathbb{T}^{d}$, corresponding to the euclidean case or periodic case respectively.

From \eqref{eq:NLS_polar2}, \eqref{selfsimilarchange}, we obtain 
\begin{align*}
 -(2-r)\Psi + y \nabla \Psi + \p_s \Psi   &=
-r^{-2\alpha + 2}  P^{2\alpha}
- |\nabla \Psi|^2 + (T-t)^{2-\frac{4}{r}} \left( \frac{\Delta P}{2 P} -   \frac{|\nabla P|^2}{4P^2} \right) \\ 
 -\frac{1-r}{\alpha} P + y \nabla P + \p_s P  & =
-2  \nabla P \cdot \nabla \Psi
-2  P \Delta \Psi \\ 
\end{align*}

Therefore, we obtain
\begin{align} \begin{split} \label{eq:SSequation}
\p_s \Psi   &= (2-r) \Psi -y \nabla \Psi - |\nabla \Psi|^2 - r^{-2\alpha + 2}  P^{2\alpha}
 + e^{(4-2r)s} \left( \frac{\Delta P}{2 P} -   \frac{|\nabla P|^2}{4P^2} \right) \\ 
\p_s P  & = \frac{1-r}{\alpha} P - y\nabla P
-2  \nabla P \cdot \nabla \Psi
-2  P \Delta \Psi \\ 
\end{split} \end{align}

Defining 
\begin{equation}\label{rescaleden}
S =  \frac{r^{1-\alpha}}{\sqrt \alpha} P^\alpha
\end{equation}
and letting $U = \nabla \Psi$, we have
\begin{align} \begin{split} \label{eq:fulleq_SPsi}
\p_s \Psi &= -(r-2)\Psi - y \nabla \Psi - | \nabla \Psi |^2 - \alpha S^2  + e^{(4-2r)s} \frac{\Delta (S^{1/(2\alpha)})}{S^{1/(2\alpha)}} \\ 
\p_s S &= -(r-1)S - y \nabla S - 2 \nabla S \cdot U - 2 \alpha S \text{div} (U).
\end{split} \end{align}

 One can also write the equation in $U, S$ variables, namely 
\begin{align} \begin{split} \label{eq:fulleq_US}
\p_s U &= -(r-1)U - y \nabla U - 2 U \cdot \nabla U - 2 \alpha S \nabla S  + e^{(4-2r)s} \nabla\frac{\Delta (S^{1/(2\alpha)})}{S^{1/(2\alpha)}} \\ 
\p_s S &= -(r-1)S - y \nabla S - 2 \nabla S \cdot U - 2 \alpha S \text{div} (U).
\end{split} \end{align}
Notice that when $r>2$, \eqref{eq:fulleq_US} is a perturbation
of the compressible Euler equation in dimension $d$.
From \cite{Buckmaster-CaoLabora-GomezSerrano:implosion-compressible,Merle-Raphael-Rodnianski-Szeftel:implosion-i}, we have a self-similar profile neglecting the inhomogeneous part $e^{(4-2r)s} \frac{\Delta (S^{1/(2\alpha)})}{S^{1/(2\alpha)}}$ when $2\leq d \leq 9 $. 

That is, there exist $S_p$, $\Psi_p$, satisfying:
\begin{align} \begin{split} \label{eq:SSprofiles}
0   &= -(r-2) \Psi_p -y \nabla  \Psi_p - |\nabla  \Psi_p|^2 - \alpha S_p^2 \\ 
0  & = -(r-1)  S_p - y\nabla  S_p
-2 \nabla S_p \cdot \nabla \Psi_p
-2 \alpha  S_p \Delta \Psi_p
\end{split} 
\end{align}
Let us mention that here we have an extra coefficient $2$ in \eqref{eq:SSprofiles} compared to the equation in \cite{Buckmaster-CaoLabora-GomezSerrano:implosion-compressible,Merle-Raphael-Rodnianski-Szeftel:implosion-i}, which just corresponds with a rescaling of the profiles by a factor of $2$ that needs to be done in order to pass from compressible Euler to NLS. Thus, rigorously speaking, our profiles correspond to the ones from \cite{Buckmaster-CaoLabora-GomezSerrano:implosion-compressible,Merle-Raphael-Rodnianski-Szeftel:implosion-i} divided by $2$. Letting $\bar U_{p, R} = \p_R \Psi_p$, the constructions from \cite{Buckmaster-CaoLabora-GomezSerrano:implosion-compressible,Merle-Raphael-Rodnianski-Szeftel:implosion-i} guarantee the following properties:
\begin{align} 
 S_{p} &>0, \label{eq:profiles_positive} \\
 |\nabla^{j} \bar U_{p,R}| +  |\nabla^j S_{p} | &\les \langle R \rangle^{-(r-1)-j} , \quad \forall j \geq 0, \quad \mbox{ and } \quad S_{p} \gtrsim \langle R \rangle^{-r+1}. \label{eq:profiles_decay} \\
\partial_{R}S_{p}(0)& =0, \label{eq:regularityatorigin}\\
R + 2 \bar{U}_{p,R} - 2\alpha S_{p} & = 0 \text{ when } R=1. \label{critical point condition} \end{align}
Moreover, the profiles satisfy \textit{repulsivity properties}. That is, there exists $\tilde \eta > 0$ such that:
\begin{align}
1 + 2\p_R \bar U_{p,R} - 2\alpha |\p_R  S_{p} | &> \tilde{\eta}, \label{eq:radial_repulsivity}\\
1 +2\frac{\bar U_{p,R}}{R} - 2\alpha |\p_R  S_{p} | &> \tilde{\eta},\label{eq:angular_repulsivity}
\end{align}
The properties \eqref{eq:profiles_positive}-\eqref{critical point condition} follow from the construction of the profiles easily, and we refer to \cite{Merle-Raphael-Rodnianski-Szeftel:implosion-i, Buckmaster-CaoLabora-GomezSerrano:implosion-compressible} for their justifications. The property \eqref{eq:radial_repulsivity} is essential for the radial stability and also proved in those papers. The property \eqref{eq:angular_repulsivity} is needed for the non-radial stability and requires substantial amount of work. We prove it in Lemma \ref{lemma:angular_repulsivity} in the Appendix, for the choice $(d, p) = (8, 3)$. We remark that the constants $2$ disappear in properties \eqref{critical point condition}--\eqref{eq:angular_repulsivity} in the previous literature \cite{Merle-Raphael-Rodnianski-Szeftel:implosion-i, Buckmaster-CaoLabora-GomezSerrano:implosion-compressible, CaoLabora-GomezSerrano-Shi-Staffilani:nonradial-implosion-compressible-euler-ns-T3R3}, due to the fact that the profiles are multiplied by $2$ with respect to our current notation (in particular, those profiles satisfy a rescaled version of \eqref{eq:fulleq_US} without the factors $2$ in front of the nonlinear terms).

\begin{theorem} \label{th:periodic}  Let $\Psi_{p}, S_{p}$ be self-similar profiles solving \eqref{eq:SSprofiles} for $(d,p) = (8,3)$ and satisfying \eqref{eq:profiles_positive}--\eqref{critical point condition}, for some $r>2$. Let $T > 0$ sufficiently small. 

Then, there exists $C^\infty$ initial data $v_0$ for which equation \eqref{eq:NLS} on $\mathbb T^d$ blows up at time $T$ in a self-similar manner. More concretely, for any fixed $y\in \mathbb T^d$, $v(y,t)=\sqrt{\rho}(y,t)e^{i\Psi(y,t)}$ with $\rho\in \mathbb{R}^{+}$, $\Psi\in \mathbb{R}$, we have:
\begin{align*} 
\lim_{t\rightarrow T^-}r (T-t)^{1-\frac{2}{r}} \Psi\left((T-t)^{\frac1r}y,t\right) &= \Psi_{p} (|y|), \\
\lim_{t\rightarrow T^-} \left( \alpha^{-1 } r (T-t)^{1-\frac{1}{r}} \right)^{1/\alpha} \rho\left((T-t)^{\frac1r}y,t\right) &= S_{p}(|y|)^{1/\alpha} \,.
\end{align*}
Moreover, there exists a finite codimension set of initial data satisfying the above conclusions.
\end{theorem}
\begin{theorem} \label{th:Euclidean}  Let $\Psi_{p}, S_{p}$ be self-similar profiles solving \eqref{eq:SSprofiles} for $(d,p) = (8,3)$ and satisfying \eqref{eq:profiles_positive}--\eqref{critical point condition}, for some $r>2$. Let $T > 0$ sufficiently small. 

Then, there exists $C^\infty$ initial data $v_0$ for which equation \eqref{eq:NLS} on $\mathbb R^d$ blows up at time $T$ in a self-similar manner. More concretely, for any fixed $y\in \mathbb R^d$, $v(y,t)=\sqrt{\rho}(y,t)e^{i\Psi(y,t)}$ with $\rho\in \mathbb{R}^{+}$, $\Psi\in \mathbb{R}$, we have:
\begin{align*} 
\lim_{t\rightarrow T^-}r (T-t)^{1-\frac{2}{r}} \Psi\left((T-t)^{\frac1r}y,t\right) &= \Psi_{p} (|y|), \\
\lim_{t\rightarrow T^-} \left( \alpha^{-1 } r (T-t)^{1-\frac{1}{r}} \right)^{1/\alpha} \rho\left((T-t)^{\frac1r}y,t\right) &= S_{p}(|y|)^{1/\alpha} \,.
\end{align*}
Moreover, there exists a finite codimension set of initial data satisfying the above conclusions.
\end{theorem}

\begin{remark}
For simplicity, we have chosen to fix a particular instance of $(d,p)$. In principle, one could prove a similar result for other values of $(d,p)$ using a similar study of the angular repulsivity property, and even a larger class by means of a computer-assisted proof as in \cite{Buckmaster-CaoLabora-GomezSerrano:implosion-compressible}.
\end{remark}

\begin{remark}
The Sobolev norm $\|\nabla^{s}v\|_{L^2}$ blows up in both $\mathbb{R}^d$ or $\mathbb{T}^d$ when $s> \frac{d}{2(r-1)}-\frac{2}{p-1}.$ In fact we have 
\begin{align*}
&\quad \|\nabla^{s}(\sqrt{\rho}e^{i\psi})\|_{L^2}^{2}\\
&\geq \int _{|x|\leq (T-t)^{\frac{1}{2}}}|\nabla^{s}(\sqrt{\rho}e^{i\psi})^2|dx\\
&\gtrsim \int_{|y|\leq 1}|(T-t)^{\frac{1}{2\alpha r}-\frac{1}{2\alpha}}(T-t)^{\frac{-s}{r}}|^2|\nabla_{y}^{s}(\sqrt{P(y)}e^{i\Psi(y)(T-t)^{\frac{2}{r}-1}})|^2dy (T-t)^{\frac{d}{r}}\\
&\gtrsim (T-t)^{\frac{1}{\alpha r}-\frac{1}{\alpha}-\frac{2s}{r}+\frac{d}{r}+2s(\frac{2}{r}-1)}\\
&=(T-t)^{\frac{1}{\alpha r}-\frac{1}{\alpha}+\frac{d}{r}-2s(1-\frac{1}{r})}.
\end{align*}
Then when $s>\frac{\frac{1}{\alpha r}-\frac{1}{\alpha}+\frac{d}{t}}{2(1-\frac{1}{r})}=\frac{d}{2(r-1)}-\frac{2}{p-1},$
we have blow up of the $\dot{H}^s$ norm.
\end{remark}
\begin{remark} In recent years several works have been put forward in which 
it was proved that when one studies the long time dynamics of a periodic NLS 
the nature of the periodicity, namely the rationality or   irrationality\footnote{We say that a torus $\T^d$ of periods $T_1, T_2, \dots T_d$ is irrational if there is no non-zero vector $\vec w=(w_1,\dots, w_d)  \in \Z^d$ such that  $\sum_{i=1}^dw_iT_i=0$.} of the  domain $\T^d$,  matters, see  \cite{Herr-Kwak:strichartz-estimates-gwp-cubic-nls-t2,Deng-Germain:growth-nls-irrational-tori,
Deng-Germain-Guth-RydinMyerson:strichartz-nls-nonrectangular-tori,
Deng-Germain-Guth:strichartz-nls-irrational-tori,
Staffilani-Wilson:stability-cubic-nls-irrational,
Fan-Staffilani-Wang-Wilson:bilinear-strichartz-irrational,
Giuliani-Guardia:sobolev-explosion-cubic-nls-irrational,
Camps-Staffilani:modified-scattering-cubic-nls-diophantine,
Hrabski-Pan-Staffilani-Wilson:energy-transfer-nls-irrational-tori}.  This physically makes sense since as   time evolves the nonlinear solution starts seeing the different effects of the periodic boundary conditions. In our case, as one can gather from  the conclusion of Theorem \ref{th:periodic}, we do not distinguish  between a rational or irrational torus $\T^d$. This is because the blow up result we prove is localized and not close to the boundary. Of course we do not exclude that there could be other types of blow up mechanism that indeed are sensitive to the nature of the periodicity of the boundary, but the study of this phenomenon is not the scope of the present paper.
\end{remark}

\subsection{Structure of the proof of Theorems \ref{th:periodic} and \ref{th:Euclidean}}

 We define the smooth cut-off function,

\begin{equation}\label{cutofffunction02}
\hat{\mathfrak{X}}_d (x) =1 \text{ when } |x| \leq \frac{1}{2}, \quad \hat{\mathfrak{X}}_d (x) = 0, \text{ when } |x| \geq \frac{2}{3}.
\end{equation}
\begin{equation}\label{cutofffunction03}
\tilde{\mathfrak{X}}_d (x) =0 \text{ when } |x| \leq \frac{1}{8}, \quad \tilde{\mathfrak{X}}_d (x) = 1, \text{ when } |x| \geq \frac{1}{4}.
\end{equation}
With $n_d$ sufficiently large, we define
\begin{equation}\label{cutofffunction}
\mathfrak{X}_d (x) =1 \text{ when } |x| \leq \frac{1}{2}, \quad \mathfrak{X}_d (x) = \langle x \rangle^{-n_d}, \text{ when } |x| \geq \frac{2}{3}.
\end{equation}
For the periodic case, for the sake of simplicity, we set the torus be $[-1,1]^{d}$. Tori of other size/shape can be treated similarly. In order to fit the periodic boundary condition, we define the damped profile
as follows: 
\begin{equation}\label{cutoffprofile02}
\Psi_d(y) = \Psi_p(y) \hat{\mathfrak{X}}_d \left( \frac{y}{ e^s} \right)
\end{equation} 
\begin{equation}\label{cutoffprofile}
S_d(y) = S_p(y) \hat{\mathfrak{X}}_d \left( \frac{y}{ e^s} \right)+e^{-(r-1)s}\tilde{\mathfrak{X}}_d \left( \frac{y}{ e^s} \right)
\end{equation} 
For the whole space $y\in \mathbb{R}^{d}$, notice that from \eqref{eq:profiles_decay} and \eqref{decompositionofsol}, the corresponding solution of \eqref{eq:NLS} has infinite Hamiltonian and mass. We can use a similar damped profile as in \cite{Merle-Raphael-Rodnianski-Szeftel:implosion-nls, Buckmaster-CaoLabora-GomezSerrano:implosion-compressible}:

\begin{equation}\label{cutoffdensity}
S_d(y) = S_p(y) \mathfrak{X}_d \left( \frac{y}{e^s} \right).
\end{equation} 
\begin{equation}\label{cutoffprofile0}
\Psi_d(y) = \Psi_p(y).
\end{equation} 

Let us remark here that the factor $e^s$ is not zero at our initial time $s_0$. Thus, one can think of $\mathfrak{X}_d \left( \frac{y}{e^s} \right)$ as $\mathfrak{X}_d \left( \frac{y}{L e^{s-s_0}} \right)$ for some $L$ sufficiently large. Both expressions are equivalent by taking $L = e^{s_0}$, and $s_0$ will also be sufficiently large in our argument (we will start sufficiently close to the singularity time).
\begin{remark}
From now on our proof will focus on the periodic case and use \eqref{cutoffprofile}, \eqref{cutoffprofile02} as damped profile. Our  proof can be easily adapted to the whole space.
\end{remark}

We then define the perturbations of the damped profile:
\begin{equation}\label{perturb defi}
\tilde S(y, s) = S(y, s) - S_d(y, s), \qquad \mbox{and} \qquad \tilde \Psi(y, s) = \Psi (y, s) - \Psi_d(y).
\end{equation}
We obtain, from \eqref{eq:SSequation} and \eqref{eq:SSprofiles}, that the perturbations $\tilde S, \tilde \Psi$ satisfy the equations:
\begin{align}\label{perturb equ}
\p_s \tilde \Psi &= \underbrace{ -(r-2) \tilde \Psi -y \nabla \tilde \Psi - 2 \nabla \Psi_p \cdot \nabla \tilde \Psi  - 2\alpha  S_d \tilde S }_{\mathcal L_{\Psi}} \underbrace{ -\alpha  \tilde S^2  - |\nabla \tilde \Psi|^2 }_{\mathcal N_{\Psi}}\\\nonumber
&\qquad  \underbrace{ - \p_s \Psi_d - (r-2) \Psi_d - y\nabla \Psi_d - |\nabla \Psi_d|^2 - \alpha S_d^2}_{\mathcal E_\Psi} +
 e^{(4-2r)s}   \underbrace{ \frac{\Delta ( S^{1/(2\alpha)} )}{S^{1/(2\alpha)}} }_{\mathcal D}\\\nonumber
 \p_s \tilde S  & = 
\underbrace{
-(r-1) \tilde S - y\nabla  \tilde S
-2  \nabla S_d   \cdot \nabla \tilde \Psi
-2  \nabla  \tilde S \cdot \nabla \Psi_d
-2 \alpha S_d \Delta  \tilde \Psi
-2\alpha \tilde S \Delta \Psi_p }_{\mathcal L_{P}}
\underbrace{ -2  \nabla  \tilde S \cdot \nabla \tilde \Psi
-2 \alpha \tilde S \Delta  \tilde \Psi }_{\mathcal N_{P}}
\\\nonumber
&\qquad \underbrace{ - \p_s S_d - (r-1) S_d - y\nabla S_d - 2\nabla S_d \cdot \nabla \Psi_d - 2\alpha S_d \Delta \Psi_d}_{\mathcal E_S}  
\end{align}

Let us reexpress $\mathcal E_S$. From \eqref{eq:SSprofiles}, we have that 
\begin{align}\label{errordefi}
\mathcal{E}_{\Psi} & =  \Psi_{p}(y)\nabla\hat{\mathfrak{X}}_{d}(ye^{-s})\cdot y e^{-s} +\hat{\mathfrak{X}}_{d}(ye^{-s})\underbrace{ \left( -(r-2)\Psi_{p}(y)-y\cdot \nabla \Psi_{p}(y)-|\nabla \Psi_{p}(y)|^2-\alpha|S_{p}(y)|^2 \right)}_{0} \\
& -\Psi_{p}(y)\nabla\hat{\mathfrak{X}}_{d}(ye^{-s})\cdot y e^{-s}+|\nabla \Psi_{p}(y)|^2(\hat{\mathfrak{X}}_{d}(ye^{-s})-\hat{\mathfrak{X}}_{d}^2(ye^{-s}))+\alpha|S_{p}|^2(\hat{\mathfrak{X}}_{d}(ye^{-s})-\hat{\mathfrak{X}}_{d}^2(ye^{-s}))
\nonumber \\
& -\Psi_{p}^2(y)|\nabla \hat{\mathfrak{X}}_{d}(ye^{-s})|^2e^{-2s}-\alpha e^{-2(r-1)s}\tilde{\mathfrak{X}}_{d}^2(ye^{-s})-2\Psi_{p}(y)\nabla\Psi_{p}(y)\cdot \nabla \hat{\mathfrak{X}}_d(ye^{-s})\hat{\mathfrak{X}}_d(ye^{-s})e^{-s} \nonumber \\
& -2\alpha e^{-(r-1)s}\tilde{\mathfrak{X}}_{d}(ye^{-s})S_{p}(y)\hat{\mathfrak{X}}_{d}(ye^{-s}) \nonumber \\
& = |\nabla \Psi_{p}(y)|^2(\hat{\mathfrak{X}}_{d}(ye^{-s})-\hat{\mathfrak{X}}_{d}^2(ye^{-s}))+\alpha|S_{p}|^2(\hat{\mathfrak{X}}_{d}(ye^{-s})-\hat{\mathfrak{X}}_{d}^2(ye^{-s}))
-\Psi_{p}^2(y)|\nabla \hat{\mathfrak{X}}_{d}(ye^{-s})|^2e^{-2s} \nonumber \\
& -\alpha e^{-2(r-1)s}\tilde{\mathfrak{X}}_{d}^2(ye^{-s})-2\Psi_{p}(y)\nabla\Psi_{p}(y)\cdot \nabla \hat{\mathfrak{X}}_d(ye^{-s})\hat{\mathfrak{X}}_d(ye^{-s})e^{-s}-2\alpha e^{-(r-1)s}\tilde{\mathfrak{X}}_{d}(ye^{-s})S_{p}(y)\hat{\mathfrak{X}}_{d}(ye^{-s}), \nonumber
\end{align}
\begin{align}\label{error02}
\mathcal E_S= & S_{p}(y)\nabla\hat{\mathfrak{X}}_{d}(ye^{-s})\cdot y e^{-s}+\nabla \tilde{\mathfrak{X}}_{d}(ye^{-s})\cdot y e^{-s}e^{-(r-1)s}+(r-1)\tilde{\mathfrak{X}}_{d}(ye^{-s})e^{-(r-1)s}\\\nonumber
&-(r-1)(\hat{\mathfrak{X}}_{d}(ye^{-s})S_{p}(y)+e^{-(r-1)s}\tilde{\mathfrak{X}}_{d}(ye^{-s}))\\\nonumber
&-y\cdot \nabla (\hat{\mathfrak{X}}_{d}(ye^{-s})S_{p}(y)+e^{-(r-1)s}\tilde{\mathfrak{X}}_{d}(ye^{-s}))\\\nonumber
&-2\alpha (S_p(y)\hat{\mathfrak{X}}_d(ye^{-s})+e^{-(r-1)s}\tilde{\mathfrak{X}}_d(ye^{-s}))\Delta (\Psi_{p}(y)\hat{\mathfrak{X}}_d(ye^{-s}))\\\nonumber
&-2\nabla (S_p(y)\hat{\mathfrak{X}}_d(ye^{-s})+e^{-(r-1)s}\tilde{\mathfrak{X}}_d(ye^{-s}))\cdot \nabla (\Psi_{p}(y)\hat{\mathfrak{X}}_d(ye^{-s}))\\\nonumber
=&\underbrace{\hat{\mathfrak{X}}_d(ye^{-s})(-(r-1)S_{p}(y)-y\cdot \nabla S_{p}(y)-2\alpha S_{p}(y)\Delta\Psi_{p}(y)-2\nabla S_{p}(y)\cdot \nabla \Psi_{p}(y))}_{0}\\\nonumber
&+2\alpha S_{p}(y)\Delta\Psi_{p}(y)(\hat{\mathfrak{X}}_{d}(ye^{-s})-\hat{\mathfrak{X}}_{d}^2(ye^{-s}))+2 \nabla S_{p}(y)\cdot \nabla \Psi_{p}(y)(\hat{\mathfrak{X}}_{d}(ye^{-s})-\hat{\mathfrak{X}}_{d}^2(ye^{-s}))\\\nonumber
&-2\alpha S_{p}(y)\hat{\mathfrak{X}}_{d}(ye^{-s})(2\nabla \Psi_{p}(y)\cdot \nabla \tilde{\mathfrak{X}}_d(ye^{-s})e^{-s}+\Psi_{p}(y)\Delta\hat{\mathfrak{X}}_{d}(ye^{-s})e^{-2s})\\\nonumber
&-2 \nabla(S_{p}(y)\hat{\mathfrak{X}}_{d}(ye^{-s}))\Psi_{p}(y)\nabla \hat{\mathfrak{X}}_d(ye^{-s})e^{-s}-2S_{p}\nabla\hat{\mathfrak{X}}_{d}(ye^{-s})\cdot \nabla\Psi_{p}(y)\hat{\mathfrak{X}}_d(ye^{-s})e^{-s}\\\nonumber
&-2\alpha e^{-(r-1)s}\Delta(\Psi_{p}(y)\hat{\mathfrak{X}}_{d}(ye^{-s}))\tilde{\mathfrak{X}}_{d}(ye^{-s})-2 e^{-(r-1)s} \cdot\nabla(\Psi_{p}(y)\hat{\mathfrak{X}}_{d}(ye^{-s}))\nabla \tilde{\mathfrak{X}}_{d}(ye^{-s}) e^{-s}.
\end{align}

\subsection{Organization of the paper}

The paper is organized as follows: Section \ref{sec:linear} is devoted to the study of the linearized operator around $S_d, \Psi_d$, and it consists of the proof of Propositions \ref{prop:maxdis}, \ref{prop:smooth} and Lemma \ref{lemma:lineardissipativity}, concerning maximality, smoothness of the unstable modes and dissipativity of the linear operator. Section \ref{sec:nonlinear} upgrades the linear estimates to nonlinear, bootstrapping the bounds of the perturbation. Finally, the Appendix contains a proof of property \eqref{eq:angular_repulsivity} for the globally self-similar solution, as well as some other properties of the solution and some functional analytic results.

\subsection{Notation}

Throughout the paper we will use the following convention:

We will use $C_x$ (resp. $C$) to denote any constant that may depend on $x$ (resp. independent of all the other parameters). These constants may change from line to line but they are uniformly bounded by a universal constant dependent on $x$ (independent of all other parameters). Similarly, we will denote by 
$X\lesssim Y$ and by 
$X\lesssim_x Y$ whenever $X\leq C Y$ and $X\leq C_x Y$ for some $C$, $C_x$ respectively.

 We use $\nabla^{l}f$ for the $l$-th order tensor containing all $l$ order derivatives of $f$. In the case where $f$ is a vector (typically $f = U$) then $\nabla^l U$ is a $(l+1)$-th tensor. We also denote by $| \nabla^k f|$ its $\ell^2$ norm. Note that this norm controls all the derivatives \textit{accounting for reorderings of indices}. For example, we have
 \begin{equation*}
|\nabla^j f|^2 = \sum_{\beta \in [3]^j} |\p_\beta f |^2 \geq \binom{j}{j'} |\p_1^{j'} \p_2^{j-j'} f|^2 .
 \end{equation*}
 Thus, we see that $|\nabla^j f|$ has a stronger control of the mixed derivatives by a combinatorial number, just because any mixed derivative appears multiple times in $\nabla^j f$ as the indices are reordered.


For $\beta=(\beta_1,\beta_2...,\beta_K)$, we let  $\partial_{\beta}=\partial_{\beta_1,\beta_2,\beta_3...\beta_K}$, and $\partial_{\beta^{(j)}}=\partial_{\beta_1,\beta_2,...\beta_{j-1},\beta_{j+1}....\beta_{K}}$ (that is, $\beta_j$ denotes the $j$-th component of $\beta$ and $\beta^{(j)}$ denotes the subindex obtained by erasing the $j$-th component in $\beta$).

We will use $y$ as our self-similar variable and we will denote $R = |y|$ for the radius (given that $r$ is reserved, since $1/r$ is the self-similar exponent of our ansatz). We will denote with $\hat{R}$ the unitary vector field in the radial direction, that is, $\hat R = \frac{y}{|y|} = \frac{y}{R}$. 

\subsection{Acknowledgements}
GCL and GS have been supported by NSF under grant DMS-2052651. GCL, JGS and JS have been partially supported by the MICINN (Spain) research grant number PID2021–125021NA–I00.
JGS has been partially supported by NSF under Grants DMS-2245017 and DMS-2247537, and by the AGAUR project 2021-SGR-0087 (Catalunya). 
JS has been partially supported by an AMS-Simons Travel Grant. GS has been supported by the Simons Foundation through the Simons Collaboration Grant on Wave Turbulence. JGS is also thankful for the hospitality of the MIT Department of Mathematics, where parts of this paper were done.




\section{Linear estimates}
\label{sec:linear}
\subsection{Linear estimates for the truncated operator}

This part is very similar as in \cite[Section 2]{CaoLabora-GomezSerrano-Shi-Staffilani:nonradial-implosion-compressible-euler-ns-T3R3} except for the fact that we work with $\nabla \Psi$ instead of $U$. We consider $\chi_1 (x), \chi_2 (x)$ to be cut-offs with value $1$ near the origin and transition regions $[6/5, 7/5]$ and $[8/5, 9/5]$ respectively. Notice that in the support of $\chi_1$ and $\chi_{2}$, as long as $s_0$ is sufficiently large, we have $(\Psi_{p}, S_{p})=(\Psi_{d},S_d)$.   
\begin{figure}[h]
\centering
\includegraphics[width=0.6\textwidth]{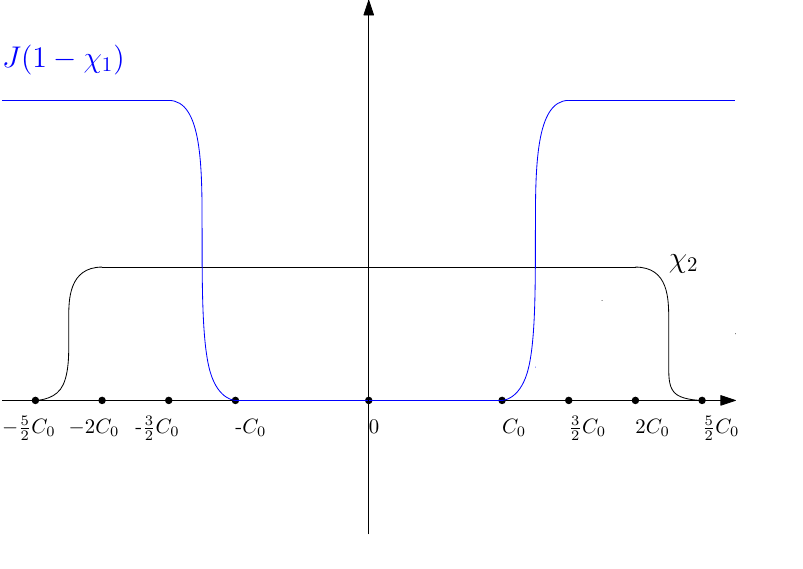}
\caption{The graph of $J(1-\chi_1)$ and $\chi_2$.}
\label{fig:cut off}
\end{figure}

Using those cut-offs, we define the cut-off linear operator as:
\begin{equation} \label{eq:cutoffL}
\mc L_{\Psi, t} = \chi_2(\frac{y}{C_0}) \mc L_\Psi - J(1-\chi_1(\frac{y}{C_0})), \quad \mc L_{S, t} = \chi_2(\frac{y}{C_0})\mc L_S - J(1-\chi_1(\frac{y}{C_0})),
\end{equation}
where $J,C_0$ are sufficiently large constant.

We will study the cut-off linearized operator $\mathcal L_t = (\mathcal L_{\Psi, t}, \mathcal L_{S, t})$ in the space $X$, which we define as follows. First, we consider $\tilde X$ to be the space $H^{m+1}_{\Psi}(Q , \mathbb{R}) \times  H^{m}_s (Q, \mathbb{R})$  where the $\Psi, s$ subscripts make reference to $\tilde \Psi$ (vector field) and $\tilde S$ (scalar field). $Q$ is a cube centered at the origin of length $4C_0$ and we take periodic boundary conditions. We select $m$ to be a sufficiently large integer. We define $X$ to be the subspace of $\tilde X$ formed by $(\Psi, S)\in \tilde X$ such that both $U$ and $S$ are compactly supported on $B(0, 3C_0)$. One can thus think of $X$ as a subspace of $H^{m+1}_\Psi (B(0, 3C_0), \mathbb R) \times H^{m}_s (B(0, 3C_0), \mathbb R)$ with the appropriate vanishing conditions at the boundary. Whenever we take Fourier series of $(\Psi, S) \in X$ it should be understood that we do Fourier series on the cube $Q$.

We put in our space $X$ the standard inner product
\begin{equation*}
\langle (f_\Psi, f_s), (g_\Psi, g_s) \rangle_X = \int_{B(0, 3C_0)} \left(  \nabla^{m+1} f_\Psi \cdot \nabla^{m+1} g_\Psi + \nabla^{m} f_s \cdot \nabla^m g_s + f_\Psi \cdot g_\Psi + f_s g_s \right)
\end{equation*}
 and we obtain the inherited norm
\begin{equation}\label{spacexnorm}
\| (f_\Psi, f_s) \|_X^2 = \int_{B(0, 3C_0)} \left( |\nabla^{m+1} f_\Psi|^2 + (\nabla^m f_s )^2 + |f_\Psi|^2 + f_s^2 \right),
\end{equation}
where $\nabla^m f_s$ is an $m$-tensor indexed by $[d]^m = \{ 1, 2, ...d \}^m$  containing all possible combinations of directional derivatives. In the case of $\nabla^{m+1} f_\Psi$, we have an $(m+1)$-tensor since we have an extra index indicating the component of the vector $f_\Psi$.

Similar as in the \cite{CaoLabora-GomezSerrano-Shi-Staffilani:nonradial-implosion-compressible-euler-ns-T3R3}, we have finite-dimension unstable space (Proposition \ref{prop:maxdis}) and smoothness of the eigenvectors in the unstable space (Proposition \ref{prop:smooth}). 

\begin{proposition}[\protect{\cite[Proposition 1.8]{CaoLabora-GomezSerrano-Shi-Staffilani:nonradial-implosion-compressible-euler-ns-T3R3}}] \label{prop:maxdis} For $\delta_g$ sufficiently small, $m$ sufficiently large and $J$ defined in \eqref{eq:cutoffL} sufficiently large, the following holds. The Hilbert space $X$ can be decomposed as $V_{\rm{sta}} \oplus V_{\rm{uns}}$, where both $V_{\rm{sta}}, V_{\rm{uns}}$ are invariant subspaces of $\mc L_t$. $V_{\rm{uns}}$ is finite dimensional and formed by smooth functions. Moreover, there exists a metric $B$ of $V_{uns}$ such that the decomposition satisfies:
\begin{align} \begin{split} \label{eq:decomposition_condition}
\Re\langle \mc L_t v,  v \rangle_B &\geq \frac{-6}{10}\delta_g \| v \|_B^2,  \qquad  \forall v \in V_{\rm{uns}}, \\
\left\| e^{s \mc L_t} v \right\|_X &\leq e^{-s\delta_g/2} \| v \|_X, \qquad \forall v \in V_{\rm{sta}},
\end{split} \end{align}
where we use $\Re$ to denote the real part.
\end{proposition}

\begin{proof}
The proof of Proposition \ref{prop:maxdis} is the same as in \cite{CaoLabora-GomezSerrano-Shi-Staffilani:nonradial-implosion-compressible-euler-ns-T3R3}. It follows from the dissipativity and maximality of $\mc L_{t}$ in lemma \ref{lemma:lineardissipativity}.
\end{proof}

Let  $X_K$ be the finite codimension subspace of $X$ where the Fourier modes smaller than $K$ vanish, that is 
\[\mathcal F(\Psi)(k) = 0, \quad \mathcal F (S)(k) = 0  \text{ for } |k| < K.
\]
We have the following lemma:
\begin{lemma} [Dissipativity and Maximality]\label{lemma:lineardissipativity} Taking $K$ sufficiently large in terms of $J$, and $J$ large enough depending on $m$, there is a universal constant $C$ (independent of $m$) such that
\begin{equation} \label{eq:diss_final}
    \int_{B(0,3C_0)}\nabla^m \mc L_{\Psi, t} \cdot \nabla^m \tilde \Psi + \nabla^m \mc L_{S, t} \cdot \nabla^m \tilde S \leq - \|(\tilde \Psi, \tilde S ) \|_X^2, \qquad \forall (\tilde \Psi, \tilde S) \in X_K.
\end{equation}
Moreover, there exists some $\lambda$ (depending on $m$, $K$ and $J$) such that $\mc L - \lambda : X \to X$ is a surjective operator.
\end{lemma}

\begin{proof}
 Here $(\nabla \tilde{\Psi},\tilde{S})$ corresponds to $(\tilde{U},\tilde S)$ in \cite{CaoLabora-GomezSerrano-Shi-Staffilani:nonradial-implosion-compressible-euler-ns-T3R3} with the properties of profile \eqref{eq:angular_repulsivity} and \eqref{eq:radial_repulsivity}. The only  difference is that, since we replace $\tilde{U}$ by $\nabla \tilde{\Psi}$, the cut-off in the definition of $\mathcal L_t$ \eqref{eq:cutoffL} is performed on $(\tilde{\Psi},\tilde{S})$ instead of $(\tilde{U},\tilde{S})$.  
Both the dissipativity and maximality work in a similar way since the change does not effect the leading order estimate ($\dot{H}^{m+1}_\Psi (Q) \times \dot{H}^m_S(Q) $). All the terms where derivatives enter the cut-off are bounded as error terms, so it is irrelevant if the cut-off is performed at the level of $\nabla \tilde{\Psi}$ or $\tilde{\Psi}$.
   The error terms are with lower derivative and can be controlled easily by taking $K$ sufficiently large.

\end{proof}

\begin{proposition}[Smoothness of eigenvectors in $V_{uns}$] 
\label{prop:smooth} Let $\delta_g \in (0, 1)$. For sufficiently large $\bar{m}$, $\bar{J}$, there exist $m,J$ such that $m>\bar{m}$, $J>\bar{J}$ and  we have that the space $V_{\rm{uns}} \subset X$ defined in Lemma \ref{prop:maxdis} is formed by smooth functions.
\end{proposition}
\begin{proof}
The major difficulty here is the $d$ dimension instead of $3$ dimension. Instead of the vector spherical harmonics used in \cite{CaoLabora-GomezSerrano-Shi-Staffilani:nonradial-implosion-compressible-euler-ns-T3R3}. We use the d-dimension spherical harmonic with
\[
S = \sum_n S_n(R) e_n,\ \quad \Psi=\sum_n \Psi_n(R) e_n,
\]
where $e_n$ are spherical 
harmonics in d-dimensions \cite{Efthimiou-Frye:spherical-harmonics}. Different from the compressible Euler, the decomposition here highly relies on the fact that $U=\nabla \Psi$ is divergence free, therefore vector spherical harmonics are not needed.
A similar proof as \cite[Proposition 2.11]{CaoLabora-GomezSerrano-Shi-Staffilani:nonradial-implosion-compressible-euler-ns-T3R3} follows.
\end{proof}
\subsection{The truncated equation and extended equation}

In order to choose the initial data of equation \eqref{perturb equ}, let
\begin{align}\label{initial data}
\begin{cases}
\tilde{\Psi}_0=\tilde{\Psi}_0^{'}+\sum_{i=1}^{N}a_{i}
v_{i,\psi}\\
\tilde{S}_0=\tilde{S}_0^{'}+\sum_{i=1}^{N}a_{i}
v_{i,s},
\end{cases}
\end{align}
where 
\begin{equation} \label{eq:tilde_is_stable}
(\chi_2 \tilde \Psi_0', \chi_2 \tilde S_0' ) \in V_{\rm{sta}},
\end{equation}
and $(v_{i,u},v_{i,s})$ is an orthogonal and normal basis of $V_{\rm{uns}}$.  Here  ${a_i}$ is an $N$-dimensional vector to be fixed later.

In order to use the linear estimate  (Proposition \ref{prop:maxdis}), we define a truncated equation of \eqref{perturb equ} as:
\begin{equation}\label{truncatedequation} \begin{cases}
\p_s \Psi_t = \chi_2 \left( \frac{y}{C_0} \right) \mathcal L_{\Psi} (\Psi_t, S_t) - J \left( 1 - \chi_1 \left( \frac{y}{C_0} \right) \right)\Psi_t + \chi_2 \left( \frac{y}{C_0}\right) \mathcal F_\Psi \\ 
\p_s S_t = \chi_2 \left( \frac{y}{C_0} \right) \mathcal L_{S} (\Psi_t, S_t) - J \left( 1 - \chi_1 \left( \frac{y}{C_0} \right) \right) S_t + \chi_2 \left( \frac{y}{C_0} \right) \mathcal F_S, \\
(\Psi_t(y, s_0), S_t (y, s_0) )=(\Psi_{t, 0}, S_{t, 0}) = \chi_2 \left( \frac{y}{C_0} \right) ( \tilde \Psi^{'}_0, \tilde S^{'}_0 ) + \sum_{i=1}^{N}a_i
(v_{i,\psi},
v_{i,s}),
\end{cases} \end{equation}
where $\mathcal F_\Psi = \mathcal N_\Psi + \mathcal E_\Psi + e^{(4-2r)s}\mathcal D$ and $\mathcal F_S = \mathcal N_S + \mathcal E_S$ are the corresponding terms in \eqref{perturb equ}. Notice that $\chi_{2}(\frac{y}{C_0})\mc E_{S}$ and $\chi_{2}(\frac{y}{C_0})\mc E_{\Psi}$ vanishes because both terms vanishes in the support of $\chi_{2}(\frac{y}{C_0}).$ Here we have 
\[
P_{\rm{sta}}(\Psi_{t, 0}, S_{t, 0})=\chi_2 \left( \frac{y}{C_0} \right) ( \tilde \Psi^{'}_0, \tilde S^{'}_0 ), \quad P_{\rm{uns}}(\Psi_{t, 0}, S_{t, 0})=\sum_{i=1}^{N}a_i
(v_{i,\psi},
v_{i,s}).
\]$P_{\rm{sta}}$ and $P_{\rm{uns}}$ are the projections to the stable and unstable spaces of $\mathcal L_t$ that appears in \ref{prop:maxdis}. Let us mention that the initial data $(\Psi_{t, 0}, S_{t, 0})$ is compactly supported since the unstable modes will consist on finitely many compactly supported functions (and therefore $P_{\rm{uns}}$ is always compactly supported). 

Let us also note that the truncated equation is not a closed equation. On the contrary, one needs to provide the truncated equation with the forcings $\mathcal F_\Psi, \mathcal F_S$, and in particular, to compute $\mathcal N_\Psi, \mathcal N_S, \mathcal D$ one needs to know $\tilde \Psi, \tilde S$. Note also that the initial data of the system is the initial data of the original (extended) equation with the appropriate cut-off.

We will rewrite the truncated equation \eqref{truncatedequation} as follows
\begin{equation} \label{eq:truncated}
\p_s (\Psi_t, S_t) = \mathcal L_t (\Psi_t, S_t) + \mathcal F_t
\end{equation}
where $\mathcal L_t$ includes the linear operator with a cut-off and the artificial damping, while $\mathcal F$ includes both coordinates of the forcing :
\begin{equation*}
\mathcal L_t (\Psi_t, S_t) 
=\left(\chi_2 \left( \frac{y}{C_0} \right) \mathcal L_{\Psi} (\Psi_t, S_t) - J \left( 1 - \chi_1 \left( \frac{y}{C_0} \right) \right)\Psi_t, \chi_2 \left( \frac{y}{C_0} \right) \mathcal L_{S} (\Psi_t, S_t) - J \left( 1 - \chi_1 \left( \frac{y}{C_0} \right) \right)S_t\right)
\end{equation*}
and
\begin{equation*}
\mathcal F_t  
=\left(\chi_2  \left( \frac{y}{C_0} \right) \mc F_{\Psi}, \chi_2 \left( \frac{y}{C_0} \right)  \mc F_{S}\right).
\end{equation*}
We will study the truncated equation in the space $X$ defined in \eqref{spacexnorm}.

This setting has the following two fundamental properties. The first one is that $\mathcal L_t$ has finitely many unstable directions, and the second one that the truncated equation agrees with the original equation for $|y| \leq C_0$. 
Now, we state a lemma describing the agreement of the truncated equation with the original system:
\begin{lemma} \label{lemma:agreement} Let $(\tilde \Psi, \tilde S)$ be a smooth solution of the original (extended) equation \eqref{perturb equ} for $s \in [s_0, s_1]$. Then, there exists a solution $(\Psi_t, S_t)$ of \eqref{eq:truncated} for $s \in [s_0, s_1]$. Moreover, both solutions agree for $|y| \leq \frac{11}{10}C_0$. 
\end{lemma}
\begin{proof} 
The existence part follows from $\mathcal L_t$ being a contraction semigroup modulo finitely many modes, together with the facts $(\Psi_{t, 0}, S_{t, 0}) \in X$ and $\mathcal F_t \in X$. \\

Let us show that they agree in the indicated region. We define $(\Psi^{\rm{diff}}, S^{\rm{diff}}) = (\tilde \Psi, \tilde S) - (\Psi_t, S_t)$. Since the forcing of the truncated equation is computed from the extended equation, that term cancels, and one obtains
\begin{equation*}
\p_s (\Psi^{\rm{diff}},  S^{\rm{diff}}) = \chi_2 \left( \frac{y}{C_0 }\right) \mathcal L (\Psi^{\rm{diff}}, S^{\rm{diff}}) + \left( 1 - \chi_2 \left( \frac{y}{C_0 }\right) \right) \mathcal L (\tilde \Psi, \tilde S) + J \left( 1 - \chi_1 \left( \frac{y}{C_0} \right) \right).
\end{equation*}
If we restrict ourselves to $|y| \leq \frac{11}{10}C_0$, the last two terms are zero due to the cut-offs, and in that region, we obtain
\begin{equation} \label{eq:diff_system}
\begin{cases}
\p_s  \Psi^{\rm{diff}} &=  -(r-2) \Psi^{\rm{diff}} -y \nabla  \Psi^{\rm{diff}} - 2 \nabla \Psi_p \cdot \nabla  \Psi^{\rm{diff}}  - 2\alpha  S_d  S^{\rm{diff}}   \\
\p_s  S^{\rm{diff}} &=  -(r-1)  S^{\rm{diff}} - y\nabla   S^{\rm{diff}}
-2 \nabla \Psi_p  \cdot  \nabla   S^{\rm{diff}}  
-2 \alpha S_d \Delta \Psi^{\rm{diff}}
-2  \nabla S_d   \cdot \nabla  \Psi^{\rm{diff}}
-2  \alpha S^{\rm{diff}} \Delta \Psi_p   
\end{cases}
\end{equation}
Furthermore, we can perform an energy estimates in the equation above and defining $E^{\rm{diff}} = \int_{|y|\leq \frac{11C_0}{10}} |\nabla \Psi^{\rm{diff}} |^2 + |S^{\rm{diff}}|^2$ we get that
\begin{align*}
\p_s E^{\rm{diff}} &\leq C E^{\rm{diff}} - \int_{|y| = 11C_0 /10}  \hat R \left( \frac{y}{2} + \nabla \Psi_p \right) \left( | \nabla \Psi^{\rm{diff}} |^2 + | S^{\rm{diff}} |^2 \right) - 2\alpha \int_{|y| = 11C_0 /10} S_d  \cdot \left(  S^{\rm{diff}} \nabla \Psi^{\rm{diff}}  \right) \\
 &\leq C E^{\rm{diff}} - \int_{|y| = 11C_0 /10}  \left( \frac{|y|}{2} + \p_R \Psi_p - \alpha |S_d| \right) \left( | \nabla \Psi^{\rm{diff}} |^2 + | S^{\rm{diff}} |^2 \right).
\end{align*}
Properties \eqref{eq:radial_repulsivity} and \eqref{critical point condition} ensure that $g(y) = |y|/2 + \p_R \Psi_p - \alpha |S_p| $ is increasing and $g(1) = 0$. Therefore, $g(11C_0/10) > 0$ if we ensure $C_0 > 1$. Taking $C_0 < \frac{1}{10} e^{s_0}$, we also have that $|S_d(C_0)| = |S_p(C_0)|$. Therefore, as long as $1\leq C_0 \leq \frac{1}{10}e^{s_0}$ we obtain that $\p_s E^{\rm{diff}} \leq C E^{\rm{diff}}$ and since $E^{\rm{diff}}$ is zero at time zero (the initial data for $(\tilde \Psi, \tilde S)$ and $(\Psi_t, S_t)$ coincides for $|y| \leq \frac{11}{10}C_0$), we conclude that $E^{\rm{diff}} = 0$. This implies that $\nabla \Psi_t = \nabla \tilde \Psi$ and $S_t = \tilde S$ in this range of $y$. Thus, we have that for $|y| \leq \frac{11}{10}C_0$, $\Psi^{\rm{diff}} = f(s)$ for some function $f(s)$ (constant in space) and $S^{\rm{diff}}  = 0$, and we just need to show  $f(s)=0$. 

We evaluate the first equation of \eqref{eq:diff_system} at $y = 0$, and use that $S^{\rm{diff}}(0, s) = 0$ with $\nabla \Psi_p(0) = 0$. We get
\begin{equation*}
\p_s f = -(r-2) f(s)
\end{equation*}
and since $f(s_0) = 0$ because of the agreement of the initial data, we conclude that $f(s) = 0$ for all $s \in [s_0, s_1]$.
\end{proof}
\section{Nonlinear stability}
\label{sec:nonlinear}
\subsection{Choice of parameters}

We will choose our parameters as follows: 
\begin{equation}\label{choiceofparameter}
\frac{1}{s_0}\ll \delta_{\rm{low}} \ll \frac{1}{E}\ll \frac{1}{k}\ll\frac{1}{m'}, \frac{1}{R_0} \ll  b,\epsilon \ll O(1)\ll E_{0,0},\ll E_{1,0}\ll...E_{\frac{k}{10},0} \ll E .
\end{equation} 

Here, $b, \eps$ are small constants that just depend on the profile and $m', R_0$ are big constants that we use in order to close a low derivative $L^2$ energy estimate. This energy estimate uses the compressible Euler structure, so it is able to close by itself except for the quantum pressure term $e^{(4-2r)} \mathcal D$. However, that term carries a small constant $e^{(4-2r)s_0}$ that is smaller than any other constant involved, so can be closed using higher derivative estimates where the control is much worse.

$k$ measures the number of derivatives in the higher derivative estimates, and $E$ bounds that norm. The precise structure for NLS requires us to consider different energies $E_{l, 0}$ at different levels of derivatives in order to close our bounds. Finally, $\delta_{\rm{low}}$ measures the size of the initial data.

Importantly, none of those constants depends on $\delta_{\rm{low}}$, $s_0$, $E_{l,0}$.

$m'$ is used in the \eqref{eq:bootstrap_Hm} as the order of the low derivative energy estimate.
$R_0$, $b$ are the parameter used to define the weight $\beta$ \eqref{weightdefi} and $\phi$ \eqref{weightdefi}.
$\epsilon$ is used to estimate the decay in \eqref{eq:bootstrap_low_der00}, \eqref{eq:bootstrap_low_der02}, \eqref{eq:bootstrap_Hm}.

 $s_0$ is the initial time taken to be sufficiently large, it can depend on all other constant.

$\delta_{\rm{low}}$ is a sufficiently small number used in the definition of initial conditions \eqref{initialcondition} and bootstrap assumptions Proposition \eqref{prop:bootstrap} to bound the decay. 

$k$, $E_{l,0}$ with $0\leq l \leq \frac{k}{10}$ are used in the \eqref{highordermestimate} to control the high derivative bound.

With the above parameters, we let the initial data \eqref{initial data} satisfies the following condition:
\begin{align}\label{initialcondition}
&\|\tilde{\Psi}_{0}\|_{L^{\infty}},\|\tilde{S}_{0}\|_{L^{\infty}}\leq \delta_{\rm{low}}^{\frac{7}{4}},\\\nonumber
&\|(\chi_{2}(\frac{y}{C_0})\tilde{\Psi}'_{0},\chi_{2}(\frac{y}{C_0})\tilde{S}'_{0})\|_{X}\leq \delta_{\rm{low}}^{\frac{7}{4}},\\\nonumber
&\|\frac{\tilde{S}_{0}}{S_{d}}\|\leq \frac{\delta_{\rm{low}}}{2}\\\nonumber
&\|\beta^{m'}\nabla \tilde{\Psi}_{0}\|_{\dot{H}^{m'}}+\|\beta^{m'}\nabla \tilde{S}_{0}\|_{\dot{H}^{m'}}\leq \frac{\delta_{\rm{low}}}{2}\\\nonumber
&E_{k,l}(s_0)\leq \frac{E_{l,0}}{2},
\end{align}
with weight \begin{align}\label{weightdefi}
\beta(y)=\begin{cases}
1 & \text{ for } |y|\leq R_0,\\
\frac{|y|^{\frac{1}{10}}}{2R_0}, &\text{ for } |y|\geq 4R_0.
\end{cases}
\end{align}
We then have the following bootstrap assumptions:
\begin{prop} \label{prop:bootstrap} For the initial data chosen above \eqref{initial data}, if the solution of \eqref{perturb equ}  exists for $s\in [s_0,s_1]$ and the unstable mode satisfies $\|P_{\rm{uns}}(\Psi_t,S_t)\|_{X} \leq \delta_{\rm{low}}^{\frac{3}{2}} e^{-\epsilon(s-s_0)}$, then we have the following estimates:

low derivative $L^{\infty}$ bound:
\begin{equation} 
\label{eq:bootstrap_low_der00}
\left| \frac{\tilde S}{S_d} \right| < \delta_{\rm{low}},
\end{equation}

\begin{equation} \label{eq:bootstrap_low_der02}
\left| \tilde S \right| < \delta_{\rm{low}} e^{-\eps (s-s_0)}, \qquad \mbox{ and }\qquad \left|\tilde \Psi\right| < \delta_{\rm{low}} e^{-\eps (s-s_0)}
\end{equation}
low derivative energy bound:
\begin{equation} \label{eq:bootstrap_Hm}
\| \beta^{m'} \nabla^{m'} \tilde U \|_{L^2} + \| \beta^{m'}  \nabla^{m'} \tilde S \|_{L^{2}} \les_{m'} \delta_{\rm{low}} e^{-\eps (s-s_0)}.
\end{equation}
 Low derivative $w$ bound:
\begin{equation} \label{eq:bootstrap_Hmw}
\| \beta^{m'} \nabla^{m'-1} w \|_{L^2} \les_{m'} 1
\end{equation}

High derivative bound: \begin{align}\label{highordermestimate}
& E_{k-l}(s)=\int_{e^{s}\mathbb{T}^{d}} |\nabla^{k-l-1}S|^2 P \phi^{l} +\int_{e^{s}\mathbb{T}^{d}} |\nabla^{k-l}\Psi|^2 P \phi^{l} +e^{(4-2r)s}\int_{e^{s}\mathbb{T}^{d}}  |\nabla^{k-l}w|^2 P\phi^{l} \leq E_{l,0},
\end{align}
with $w=\frac{1}{2}\log (P)$, $0\leq l \leq \frac{1}{10}k$, $E_{0,0}\ll E_{1,0} \ll E_{2,0}\ll...E_{\frac{1}{10}k,0}, \ll E $ and weight \begin{align}\label{weightdefi02}
\phi(y)=\begin{cases}
1 & \text{ for } |y|\leq R_0,\\
\frac{|y|^2}{2R_0}, &\text{ for } |y|\geq 4R_0.
\end{cases}
\end{align}

\end{prop}
We first claim $\mathcal{E}_{s}$, $\mathcal{E}_{\Psi}$ satisfies following estimates:
\begin{lemma}
We have 
\begin{equation}\label{errores01}
|\mathcal{E}_{\Psi}|\lesssim \mathbbm{1}_{|y|\geq \frac{1}{2}e^{s}}\frac{1}{\langle y\rangle^{2(r-1)}};\quad |\mathcal{E}_{S}|\lesssim S_d\mathbbm{1}_{|y|\geq \frac{1}{2}e^{s}}\frac{1}{\langle y\rangle^{(r-1)}};
\end{equation}
\begin{equation}\label{errores03}
\|\nabla^{m'+1}\mathcal{E}_{\Psi}\beta^{m'}\|_{L^2}\leq e^{- s};\quad \|\nabla^{m'}\mathcal{E}_{S}\beta^{m'}\|_{L^2}\leq e^{- s};
\end{equation}
\begin{equation}\label{errores05}
\langle y \rangle^{-(r-1)}\lesssim S_d \lesssim \langle y \rangle^{-(r-1)}
\end{equation}
\begin{equation}\label{errores06}
e^{-(r-1)s}\lesssim S_d.
\end{equation}
\end{lemma}
\begin{proof}
The estimates \eqref{errores01}, \eqref{errores05}, \eqref{errores06}  follow trivially from the definition of $\mathcal{E}_{S}$, $\mathcal{E}_{\Psi}$, \eqref{errordefi}, \eqref{error02},  and the cut-off definition \eqref{cutoffprofile}, \eqref{cutoffprofile02}. We will show the estimate \eqref{errores03} for the first term $|\nabla \Psi_{p}(y)|^2(\hat{\mathfrak{X}}_{d}(ye^{-s})-\hat{\mathfrak{X}}^{2}_{d}(ye^{-s}))$ in $\mathcal{E}_{\Psi}$ . Other terms can be estimated in a similar way.
In fact, we have
\begin{align*}
&|\nabla^{m'+1}(|\nabla \Psi_{p}(y)|^2(\hat{\mathfrak{X}}_{d}(ye^{-s})-\hat{\mathfrak{X}}^{2}_{d}(ye^{-s})))|\\
=&|\sum_{j}C_{m,j}\nabla^{j}(|\nabla \Psi_{p}|^2)\nabla^{m'+1-j}(\hat{\mathfrak{X}}_{d}(ye^{-s})-\hat{\mathfrak{X}}^{2}_{d}(ye^{-s}))|\\
\lesssim& \langle y\rangle ^{-2(r-1)-j}e^{-s(m'+1-j)}\mathbbm{1}_{e^s\leq |y|\leq 10e^{s}}.
\end{align*}
Thus we can bound
\begin{align*}
&\quad \|(\nabla^{m'+1}(|\nabla \Psi_{p}(y)|^2(\hat{\mathfrak{X}}_{d}(ye^{-s})-\hat{\mathfrak{X}}_{d}(ye^{-s})))\beta^{m'}\|_{L^2}\\
&\lesssim  \int_{e^{s}}^{10C_0 e^{s}}\langle y\rangle^{-2(r-1)-j}\langle y\rangle^{\frac{m'}{10}}\langle y \rangle ^{d-1} dy e^{-s(m'+1-j)}\\
&\lesssim e^{s(d-2(r-1)-j+\frac{m'}{10})}e^{-s(m'+1-j)}\\
&\lesssim e^{-s(m'+1-j-\frac{m'}{10}-d+2(r-1)+j)}\\
&\lesssim e^{-s}.
\end{align*}

\end{proof}
\subsection{Linear bounds in the $X$ space} \label{subsubsec:linear}
The objective of the subsection will be to show
\begin{equation} \label{eq:Hm_bound}
\| \nabla \tilde \Psi \|_{H^{m} (|y| \leq C_0)} + \| \tilde{S} \|_{H^{m} (|y| \leq C_0)} \leq \delta_{\rm{low}} e^{- \frac32 \eps (s-s_0)}.
\end{equation}

The last ingredient we need in order to show \eqref{eq:Hm_bound} is to bound $\mathcal F_t$ in $X$ norm. Interpolation Lemma \ref{lemma:GN_general01} (adapted trivially to dimension $d$ instead of $3$) in the case $p = q = r = \infty$ tells us that the high-derivative $L^2$ bounds \eqref{highordermestimate} together with \eqref{eq:bootstrap_low_der00} yield:

\begin{equation*}
|\nabla^i \tilde S | \les \frac{E^{i/J}}{\langle y \rangle^{ai}} \left(\delta_{\rm{low}} e^{-\eps (s-s_0)} \right)^{\frac{J-i}{J}} \qquad \mbox{ and }\qquad | \nabla^i \tilde \Psi | \les \frac{E^{i/J}}{\langle y \rangle^{ai}} \left(  \delta_{\rm{low}} e^{-\eps (s-s_0)} \right)^{\frac{J-i}{J}},
\end{equation*}
\begin{equation*}
|\nabla^i \tilde S | \les \frac{E^{i/J}}{\langle y \rangle^{ai}} \left( S_d \delta_{\rm{low}}  \right)^{\frac{J-i}{J}} \qquad \mbox{ and }\qquad | \nabla^i \tilde \Psi | \les \frac{E^{i/J}}{\langle y \rangle^{ai}} \left( \Psi_p \delta_{\rm{low}}  \right)^{\frac{J-i}{J}}
\end{equation*}
with $a=\frac{1}{10}.$

Let us recall that $X$ is the space where we measure the stable modes of $\tilde \Psi, \tilde S$ on a compact support, and it is in a lower regularity norm $\dot H^{m+1}_\Psi \times \dot H^m_S$, for $m \ll k$. Now assuming that $\delta_{\rm{low}} \ll \frac{1}{E^2}$ and that $k$ is much larger compared to the regularity of $X$ (space used to measure linear stability), we have that
\begin{equation} \label{eq:dumas1}
\| (\mathcal N_\Psi, \mathcal N_S )\|_X \les e^{-\frac74 \eps (s-s_0)} \delta_{\rm{low}}^{7/4}
\end{equation}

With respect to the dissipation $\nabla^m \left(\frac{\Delta ( S^{1/\alpha} )}{S^{1/\alpha}}\right)$, we start computing:
\begin{equation*}
\frac{\Delta (S^{1/\alpha} )}{S^{1/\alpha}} = \sum_i \frac{\p_i \left(\frac{1}{\alpha} S^{1/\alpha - 1} \p_i S \right)}{S^{1/\alpha}} =  \frac{\Delta S}{\alpha S} - \frac{ (\alpha -1) |\nabla S|^2 }{\alpha^2 S^{2}}
\end{equation*}

and therefore, we can bound 
\begin{equation*}
\left| \nabla^m \left(\frac{\Delta ( S^{1/\alpha} )}{S^{1/\alpha}} \right)\right| \les_m \sum_{|i| = m+2} \frac{| \nabla^{i_1} S  |}{S} \cdot  \frac{| \nabla^{i_2} S |}{S} \ldots  \frac{| \nabla^{i_{\ell(i)}} S |}{S} 
\end{equation*}
where $|i| = m+2$ refers to all multiindices with positive integer entries $i_j$ that add up $m+2$ and $\ell (i)$ denotes the length of the multiindex. 

From our bootstrap hypothesis $|\tilde S/S_d| < \delta_{\rm{low}} $ \eqref{eq:bootstrap_low_der00}, we know that $2S_d \geq S \geq S_d/2$. Moreover, interpolating $S \leq 2S_d$ with $| \nabla^J S | \les E / \langle y \rangle^{aJ}$,  we obtain that
\begin{align}
\left| \nabla^m \left(\frac{\Delta ( S^{1/\alpha} )}{S^{1/\alpha}} \right)\right| &\les_{E, m} \sum_{|i| = m+2}\frac{S_d^{1 - i_1/J} \langle y \rangle^{-ai_1}}{S_d} \cdot  \frac{S_d^{1 - i_2/J} \langle y \rangle^{-ai_2}}{S_d} \ldots \frac{S_d^{1 - i_{\ell(i)}/J} \langle y \rangle^{-ai_{\ell(i)}}}{S_d} \notag \\
&\les \frac{\langle y \rangle^{-a (m+2)}}{S_d^{(m+2)/J}} \label{eq:highder_D}
\end{align}

Taking $s_0$ sufficiently large (viscosity sufficiently small), we have that
\begin{equation} \label{eq:dumas2}
\| \mathcal D \|_{\dot H^{m+1}} \les   \delta_{\rm{low}}^{7/4}e^{(4-2r)(s-s_0)}
\end{equation}
as well. Lastly, note that $\mathcal E_\Psi$, $\mathcal E_S$ vanish in the region where $\mathfrak{X} = 1$, so in particular they do not contribute to $\mathcal F_t$ (since it comes with a factor $\chi_2 (|y|/C_0)$). Thus, from \eqref{eq:dumas1}--\eqref{eq:dumas2}, together with the fact that $\chi_2 \in \dot H^{m}$ and that $\dot H^m (\mathbb T )$ is an algebra (for sufficiently large $m$) we have that
\begin{equation} \label{eq:dumas3}
\| \mathcal F_t \|_X \les \delta_{\rm{low}}^{7/4}  e^{-\frac74 \eps (s-s_0)} 
\end{equation}
Now, taking a stable projection, and using that $\mathcal L_t$ is invariant on $V_{\rm{sta}}$, we have that
\begin{equation*}
(\p_s - \mathcal L_t) P_{\rm{sta}} (\tilde \Psi_t, \tilde S_t) = P_{\rm{sta}} \mathcal F_t,
\end{equation*}
Using Duhamel's principle and \eqref{eq:decomposition_condition}, we obtain
\begin{align} 
\left\| P_{\rm{sta}}(\tilde \Psi_t, \tilde S_t) \right\|_X  &\leq \left\|  e^{(s-s_0) \mathcal L_t}   (\Psi_{t, 0}, S_{t, 0} ) \right\|_X + \int_{s_0}^s \left\| e^{(s-s')\mathcal L_t} P_{\rm{sta}} \mathcal F_t ( s') \right\|_X ds' \notag \\
&\les  e^{-s\delta_g/2} \| (\Psi_{t, 0}, S_{t, 0} ) \|_X + \delta_{\rm{low}}^{\frac74} \int_{s_0}^s e^{\frac12 (s-s')\delta_g} e^{\frac{7}{4}\eps (s'-s_0)}  ds' \notag \\
 &\les e^{-\frac32 \eps (s-s_0)} \left( \| (\Psi_{t, 0}, S_{t, 0} ) \|_X  + \delta_{\rm{low}}^{\frac74}\right) \label{eq:hmart}
\end{align}
We also used that $P_{\rm{sta}} : V\to V_{\rm{sta}}$ is a bounded operator. This is trivial since $P_{\rm{sta}}f = f - \sum_{i=1}^N \langle \psi_i^\ast, f \rangle_X \psi_i$, where the $\psi_i, \psi_i^\ast$ are the  right  eigenfunctions (resp. left eigenfunctions) of $\mathcal L_t$ due to Proposition \ref{prop:maxdis}.

Now, recall that our initial data satisfies 
\begin{equation}
\| \chi_2 (|y| / C_0 ) ( \Psi_0, S_0 ) \|_X \leq \delta_{\rm{low}}^{7/4}.
\end{equation}
Therefore, we obtain that 
\begin{equation*}
\| (\Psi_t, S_t ) \|_X \leq \| P_{\rm{uns}}(\Psi_t, S_t) \|_X + \| P_{\rm{sta}}( \Psi_t, S_t ) \|_X \leq 2\delta_{\rm{low}}^{3/2} e^{-\frac32 \eps (s-s_0)}
\end{equation*}
where we recall that $\| P_{\rm{uns}}(\Psi_t, S_t) \|_X \leq \delta_{\rm{low}}^{3/2} e^{-\frac32 \eps (s-s_0)}$ by hypothesis. Lastly, using Lemma \ref{lemma:agreement} we conclude the desired bound \eqref{eq:Hm_bound}.

\subsection{Closing the low derivative energy bounds}\label{lowderivativeenergy}
First, let us make some observations regarding the parameters. We have that $\frac{1}{e^{s_0}} \ll \delta_{\rm{low}} \ll \frac{1}{m} \ll \frac{1}{C_0}$, where $C_0$ is the cut-off length for the stability (constant in time in self-similar variables) and $e^s$ is the cut-off length for the damping of $S_d$ (constant length $L$ in physical variables). The number $m$ corresponds to the number  of derivatives needed  for the linear stability. All those parameters are allowed to depend on $\eps, b, R_0, m'$, which are the parameters of the low derivative energy bounds. Those low derivative energy bounds just depend on the profile and the parameter. 

The objective of this and the next  section  will be to close the bootstrap assumptions at the middle and low   derivative levels. The ingredients needed to close the low derivative are using the linear stability analysis for $|y| \leq C_0$ and trajectory estimates outside. The low derivative energy bounds follow from the energy estimates reminiscent of the compressible Euler structure (in contrast with the high derivative bounds, in which one needs to carefully combine the compressible Euler structure with the NLS one).

Up to this point, we have proved \eqref{eq:Hm_bound}. This section is devoted to proving the low derivative energy estimate \eqref{eq:bootstrap_Hm}.

We prefer to work in $\tilde{U}, \tilde{S}$ variables and deduce $\dot H^{m'}$ bounds for both quantities. We obtain the equations
\begin{align}\label{perturbationeq}
\p_s \tilde U &= \underbrace{ -(r-1) \tilde U -y \nabla \tilde U - 2U_{d}\cdot \nabla \tilde U  - 2 \tilde U  \cdot \nabla U_{d}  - 2\alpha  S_d \nabla \tilde S  - 2\alpha \tilde S \nabla S_d  }_{\mathcal L_{U}} \\\nonumber
&\qquad \underbrace{ - 2\alpha  \tilde S \nabla \tilde S  - 2 \tilde U \nabla \tilde U }_{\mathcal N_{U}} + \nabla \mathcal E_\Psi  +
 e^{(4-2r)s}  \nabla \mathcal D \\\nonumber
 \p_s \tilde S  & = 
\underbrace{
-(r-1) \tilde S - y\nabla  \tilde S -2  \nabla S_d   \cdot \tilde U-2  \nabla  \tilde S \cdot U_{d}
-2  S_d \div (\tilde U)-2  \tilde S \div (U_{d}) }_{\mathcal L_{S}} \\\nonumber
&\qquad 
\underbrace{ -2  \nabla  \tilde S \cdot \tilde U-2  \tilde S \div ( \tilde U) }_{\mathcal N_{S}} + \mathcal E_S
\end{align}

Let

$$E_{\rm{low}} = \frac12 \left( \| \nabla^{m'}\tilde U (y, s)\beta(y)^{m'} \|_{L^2} + \| \nabla^{m'}\tilde S (y, s)\beta (y)^{m'} \|_{L^2}\right)$$
\begin{align*}
\p_s E_{\rm{low}} &= 
\int_{e^s\mathbb{T}^{d}}\nabla^{m'} \tilde U \left( \nabla^{m'} \mathcal L_U + \nabla^{m'} \mathcal N_U + \nabla^{m'+1} \mathcal E_\Psi + e^{(4-2r)s} \nabla^{m'+1} \mathcal D \right) \beta^{2m'} \\
&\qquad + \int_{e^s\mathbb{T}^{d}} \nabla^{m'} \tilde S \left( \nabla^{m'} \mathcal L_S + \nabla^{m'} \mathcal N_S  + \nabla^{m'} \mathcal E_S \right) \beta^{2m'}\\
&\quad+\frac{1}{2}\int_{\partial e^s\mathbb{T}^{d}} y\cdot \vec{n}(|\nabla^{m'}\tilde{U}|^2+|\nabla^{m'}\tilde{S}|^2)\beta^{2m'} \\
\end{align*}

We clearly have that 
\begin{equation} \label{eq:Hmprime_bounds_N}
\|\beta^{m'} \nabla^{m'}\mathcal N_U \|_{L^2} + \|\beta^{m'} \nabla^{m'}\mathcal N_S \|_{L^2} \les E_{\rm{low}}^{1/2}  \left( \| \tilde U \|_{W^{m'+1, \infty}} + \| \tilde S \|_{W^{m'+1, \infty}} \right) \les E_{\rm{low}}^{1/2} \delta_{\rm{low}}^{1/2} e^{-\frac12 \eps (s-s_0)}  
\end{equation}
using the interpolation lemma \ref{lemma:GN_general01} with $| \nabla^J \tilde U | $ and $| \nabla^J \tilde S |$ in order to bound the $W^{m' + 1, \infty}$ norm.

With respect to the dissipative term, we use Cauchy-Schwarz
\begin{equation} 
\int_{e^s\mathbb{T}^{d}}\nabla^{m'} \tilde U \nabla^{m'+1} \mathcal D \beta^{2m'}\les \| \nabla^{m'+1} \mathcal D \beta^{m'} \|_{L^2}  E_{\rm{low}}^{1/2}.
\end{equation}
We recall that 
\begin{equation*}
\mathcal D = e^{(4-2r)s} \frac{\Delta \sqrt{P}}{\sqrt{P}} = e^{(4-2r)s}\left( \frac{1}{2} \Delta \log (P) + \frac{|\nabla \sqrt{P}|^2}{P^2} \right) = e^{(4-2r)s} \left( \frac{1}{2} \Delta w  + | \nabla w |^2 \right)
\end{equation*}
Therefore
\begin{equation} \label{eq:mera1}
|\nabla^{m'+1} \mathcal D | \les_{m'} e^{(4-2r)s} \left( | \nabla^{m'+3} w| +  \sum_{i = 1}^{m'+2} | \nabla^i w | \cdot | \nabla^{m'+3-i} w | \right)
\end{equation}

Now, from our bootstrap assumptions (and recalling that $| S_d | \geq \langle y \rangle^{-n_d-(r-1)}$), we have that:
\begin{equation} \label{eq:sada}
\| w \langle y \rangle^{\eps} \|_{L^\infty} \les_{n_d} 1, \qquad \| \beta^{m'} \nabla^{m'-1} w \|_{L^2} \les_{m'} 1, \qquad \mbox{ and } \qquad \| \beta^{2m'} \nabla^{9k/10} w \|_{L^2} \leq E_{l, 0}
\end{equation}
Interpolation (for example, equation \eqref{eq:GNresultinfty}) yields:
\begin{equation*}
| \nabla^j w | \les_{m, n_d} 1, \mbox{ and } | \nabla^{m'+3-j}  w | \les_{m, n_d} E_{l, 0}, \qquad \text{ for } j \in \{ 1, 2, 3, 4 \}
\end{equation*}
and also for $j=0$ in the case of the late equation. Combining this with \eqref{eq:mera1}, we have that
\begin{equation} \label{eq:mera2}
|\nabla^{m'+1} \mathcal D | \les_{m'} e^{(4-2r)s} \left(| \nabla^{m'+3} w| + \sum_{j = 5}^{m'-2} | \nabla^j w | \cdot | \nabla^{m'+3-j} w | \right)
\end{equation}

Now, we interpolate the first to equations of \eqref{eq:sada} via \eqref{eq:GNresult}. We take parameters $p = \infty$, $q = 2$, $i = j$, $m = m'+3$, $\theta_1 = \frac{4}{m'-1-d/2} + \frac{j-4}{m'-5} \cdot\frac{m'-5-d/2}{m'-1-d/2}$, which is between $\frac{j}{m'-1}$ and $1$ for $4\leq j \leq m'-1$, and the corresponding $\bar r_1$. With respect to the term with $m'+3-j$ derivatives, we correspondingly take $i = m'+3-j$ together with $\theta_2 = \frac{4}{m'-1-d/2} + \frac{m'-1-j}{m'-5} \cdot\frac{m'-5-d/2}{m'-1-d/2}$, which is analogously between $\frac{m'+3-j}{m'-1}$ and $1$. We obtain that 
\begin{equation*}
\left\| \beta^{m'} | \nabla^j w | \cdot | \nabla^{m'+3-j} w | \right\|_{L^2} \leq \left\| \beta^{m' \theta_1 + \eps} | \nabla^j w | \right\|_{L^{\bar r_1}} \cdot \left\| \beta^{m' \theta_2 + \eps} | \nabla^j w | \right\|_{L^{\bar r_2}} \les_{m'} E_{l, 0}^2,
\end{equation*}
\begin{equation*}
\left\| \beta^{m'} |\nabla^{m'+3} w|\right\|_{L^2}\lesssim_{m'}E_{l,0}^{2}.
\end{equation*}
Plugging this into \eqref{eq:mera2}, we deduce that
\begin{equation}
\int_{e^s\mathbb{T}^{d}} \nabla^{m'} \tilde U \nabla^{m'+1} \mathcal D \beta^{2m'} \les E_{\rm{low}}^{1/2} e^{-\delta_{\rm{dis}}s/2},
\end{equation}
taking $s_0$ sufficiently large with respect to $E_{l, 0}$, $m$, $n_d$.
Taking $s_0$ to be sufficiently large, and $\eps$ smaller than $\frac25 (2r-4)$ we conclude
\begin{equation} 
\int_{e^s\mathbb{T}^{d}}\nabla^{m'} \tilde U \nabla^{m'+1} \mathcal D \beta^{2m'}\les \delta_{\rm{low}}^{5/2} e^{-\frac52 \eps(s-s_0)}.
\end{equation} 

Let us note that, from \eqref{errores03},
\begin{equation*}
\int\beta^{2m'} \left(   \nabla^{m'}  \tilde U \cdot \nabla^{m'+1} \mathcal E_\Psi +  \nabla^{m'} 
 \tilde S \cdot \nabla^{m'} \mathcal E_S \right) \les e^{-s} E_{\rm{low}}^{1/2} \les \frac{1}{e^{s_0}} e^{-\frac52 \eps (s-s_0)}
\end{equation*}

Therefore, we have that
\begin{equation*}
\left| \p_s E_{\rm{low}}  -\frac{1}{2}\int_{\partial e^s\mathbb{T}^{d}} y\cdot \vec{n}(|\nabla^{m'}\tilde{U}|^2+|\nabla^{m'}\tilde{S}|^2)\beta^{2m'}- \int_{e^s\mathbb{T}^{d}}\beta^{2m'} \left( \nabla^{m'} \tilde U \cdot \nabla^{m'} \mathcal L_U + \nabla^{m'} \tilde S \cdot \nabla^{m'} \mathcal L_S \right) \right| \les  \delta_{\rm{low}}^{5/2}  e^{-\frac52 \eps (s-s_0)} 
\end{equation*}
where we recall $\frac{1}{e^{s_0}} \ll \delta_{\rm{low}}^2$.

With respect to the term with $\mathcal L_U$ or $\mathcal L_S$, with $U_{d}=\nabla \Psi_{d}$, we have
\begin{align*}
&\quad\int_{e^s\mathbb{T}^{d}}\left( \nabla^{m'} \mathcal L_U \cdot \nabla^{m'} \tilde U + \nabla^{m'} \mathcal L_S \cdot \nabla^{m'} \tilde S \right) \beta^{2m'}
\leq C E_{\rm{low}} + C_{m'} E_{\rm{low}}^{1/2} \sum_{i=0}^{m'-1} \left[ \int_{e^s\mathbb{T}^{d}}\left( | \nabla^{i} \tilde U |^2 + |\nabla^{i} \tilde S|^2 \right) \frac{\beta^{2m'}}{\langle y \rangle^{2(m'-1-i)} } \right] \\
&\quad+ \int_{e^s\mathbb{T}^{d}}\nabla^{m'} \tilde U \left( -y\nabla^{m'+1} \tilde U - 2U_{d}\nabla^{m'+1} \tilde U  - 2\alpha S_d \nabla^{m'+1} \tilde S - m' \nabla^{m'} \tilde U - 2 m'\nabla U_{d}\nabla^{m'} \tilde U  - 2m'\alpha \nabla S_d \nabla^{m'} \tilde S  \right)  \beta^{2m'} \\ 
&\quad + \int_{e^s\mathbb{T}^{d}}\nabla^{m'} \tilde S \left( -y\nabla^{m'+1} \tilde S - 2U_{d}\nabla^{m'+1} \tilde S  - 2\alpha S_d \nabla^{m'} \div  (\tilde U ) -m' \nabla^{m'} \tilde S - 2 m' \nabla U_{d}\nabla^{m'} \tilde S  - 2 m' \alpha \nabla S_d \nabla^{m'-1} \div  (\tilde U ) \right)  \beta^{2m'}\\
&\leq C' E_{\rm{low}} + C_m E_{\rm{low}}^{1/2}  \sum_{i=0}^{m'-1} \left[ \int_{e^s\mathbb{T}^{d}}\left( | \nabla^{i} \tilde U |^2 + |\nabla^{i} \tilde S|^2 \right) \frac{\beta^{2m'}}{\langle y \rangle^{2(m'-1-i)} } \right]\\ 
 &\quad -m' \int_{e^s\mathbb{T}^{d}}\nabla^{m'} \tilde U \left( \nabla^{m'} \tilde U + 2 \nabla U_{d}\nabla^{m'} \tilde U  + 2\alpha \nabla S_d \nabla^{m'} \tilde S  \right)  \beta^{2m'} \\ 
&\quad -m' \int_{e^s\mathbb{T}^{d}}\nabla^{m'} \tilde S \left(   \nabla^{m'} \tilde S + 2  \nabla U_{d}\nabla^{m'} \tilde S  + 2  \alpha \nabla S_d \nabla^{m'-1} \div  (\tilde U ) \right)  \beta^{2m'}\\
&\quad+ \int_{e^s\mathbb{T}^{d}}( |\nabla^{m'} \tilde U|^2 + | \nabla^{m'} \tilde S|^2 ) \beta^{2m'} \left( \frac{\div (y \beta^{2m'} )}{2\beta^{2m'}} + \frac{\div (U_p \beta^{2m'})}{\beta^{2m'}}\right) - 2 \alpha \int_{e^s\mathbb{T}^{d}}\beta^{2m'} S_d \left( \nabla^{m'+1} \tilde S \nabla^{m'} \tilde U + \nabla^{m'} \tilde S \nabla^{m'} \div ( \tilde U )\right) \\
&\quad-\frac{1}{2}\int_{\partial e^s\mathbb{T}^{d}} y\cdot \vec{n}(|\nabla^{m'}\tilde{U}|^2+|\nabla^{m'}\tilde{S}|^2)\beta^{2m'}\\
& =: \gimel-\frac{1}{2}\int_{\partial e^s\mathbb{T}^{d}} y\cdot \vec{n}(|\nabla^{m'}\tilde{U}|^2+|\nabla^{m'}\tilde{S}|^2)\beta^{2m'}
\end{align*}
In the last inequality we are using that for a scalar function $g$ and a vector function $F$, 
\begin{equation*}
\nabla^{j} (F \cdot \nabla g) \nabla^j g = \sum_{i=1}^{d} \sum_{\gamma \in [d]^j} \p_{\gamma} (F_i \p_i g) \p_\gamma g = F \cdot \nabla^{j+1} g \nabla^j g + \sum_{i=1}^d \sum_{k = 1}^j \sum_{\gamma \in [d]^j } \p_{\gamma_k} F_i \p_{\gamma^{(j)}} \p_i g \p_\gamma g + \text{l.o.t.}
\end{equation*}
where $[d]^j = \{ 1, 2, \ldots d \}^j$ is the set of $j$-multiindices indicating $j$ directions, $\gamma_k$ denotes the $k$-th direction and $\gamma^{(k)}$ denotes the $(j-1)$-multiindex formed by all directions except the $k$-th one. Moreover, $\text{l.o.t.}$ denotes all the terms with less than $j$ derivatives falling on $g$. The terms with $j$ derivatives on $g$ add up
\begin{equation*}
 \sum_{k = 1}^j  \sum_{i=1}^d \sum_{i' = 1}^d \sum_{\tilde \gamma \in [d]^{j-1} } \p_{i'} F_i \p_i \p_{\tilde \gamma} g \p_{i'} \p_{\tilde \gamma} g  = j \sum_{i, i' = 1}^d \p_{i'} F_i \p_i \nabla^{j-1} g \p_{i'} \nabla^{j-1} g = j \left( \nabla F \cdot \nabla^j g\right)\nabla^j g
\end{equation*}
Therefore, when we write the expression $\nabla F \cdot \nabla^j g$ it should be always understood that the subindex from $F$ is contracted with the first $\nabla$ in $\nabla^j$, and the other $j$ derivative indices (one hitting $F$ and $j-1$ hitting $g$) are left free, so that $\nabla F \cdot \nabla^j g$ is a $j$-th order tensor. Let us also note that in the case of $F$ being a radially symmetric vector and $g, h$, being any scalar, letting $F_R$ be the radial component of $F$, we have :
\begin{align}\label{eq:radialcomp}
\nabla h \cdot \nabla F \cdot \nabla g &= \sum_{i, j = 1}^d \p_i h \p_i \left( \frac{F_R y_j}{|y|} \right) \p_j g = \sum_{i, j = 1}^d  \p_i h \left[ \left( \p_R F_R - \frac{F_R}{R} \right)  \frac{y_i y_j }{R^2}  + \delta_{i, j} \frac{F_R}{R} \right] \p_j  g \notag \\
& = \left( \p_R F_R - \frac{F_R}{R} \right) \p_R h \p_R g + \frac{F_R}{R} \nabla h\cdot \nabla g = \p_R F_R \p_R h \p_R g + \frac{F_R}{R} \nabla_\theta g \cdot \nabla_\theta h
\end{align}
where $\nabla_\theta$ refers to all the components of $\nabla$ orthogonal to the radial direction.

We continue the computation of $\gimel$. Since $U_d$ We have that
\begin{align*}
\gimel &= C' E_{\rm{low}} + C_{m'} E_{\rm{low}}^{1/2}  \sum_{i=0}^{m'-1} \left[ \int_{e^s\mathbb{T}^{d}}\left( | \nabla^{i} \tilde U |^2 + |\nabla^{i} \tilde S|^2 \right) \frac{\beta^{2m'}}{\langle y \rangle^{2(m'-1-i)} } \right] \\
&\qquad
 -m' \int_{e^s\mathbb{T}^{d}}\left( | \nabla^{m'}  \tilde U |^2  + 2 \p_RU_{d}| \p_R \nabla^{m'-1} U|^2 + 2 \frac{U_d}{R} |\nabla_\theta \nabla^{m'-1} U|^2   \right) \beta^{2m'} \\ 
 &\qquad
 -m' \int_{e^s\mathbb{T}^{d}}\left( | \nabla^{m'}  \tilde S |^2  + 2 \p_RU_{d}| \p_R \nabla^{m'-1} S|^2 + 2 \frac{U_d}{R} |\nabla_\theta \nabla^{m'-1} S|^2   \right) \beta^{2m'} \\ 
& \qquad + 2\alpha m' \int_{e^s\mathbb{T}^{d}}|\p_R S_d| \beta^{2m'} (| \nabla^{m'} \tilde U|^2 + | \nabla^{m'} \tilde S|^2 )  \\
& \qquad + \int_{e^s\mathbb{T}^{d}}( |\nabla^{m'} \tilde U|^2 + | \nabla^{m'} \tilde S|^2 ) \left(\frac{ \div (y \beta^{2m'}) }{2} +\div (U_p \beta^{2m'})\right) + 2 \alpha \int_{e^s\mathbb{T}^{d}}\nabla (\beta^{2m'} S_d) \cdot \left( \nabla^{m'} \tilde S \nabla^{m'} \tilde U\right) \\
&\leq C' E_{\rm{low}} + C_{m'} E_{\rm{low}}^{1/2} \sum_{i=0}^{m'-1} \left[ \int_{e^s\mathbb{T}^{d}}\left( | \nabla^{i} \tilde U |^2 + |\nabla^{i} \tilde S|^2 \right) \frac{\beta^{2m'}}{\langle y \rangle^{2(m'-i-1)} } \right]\\
&\qquad 
 -m' \int_{e^s\mathbb{T}^{d}}\left( |\p_R \nabla^{m'-1} \tilde U|^2 + | \p_R\nabla^{m'-1} \tilde S|^2 \right) \left( 1 + 2 \p_RU_{d}- 2\alpha |\p_R S_d |  \right)  \beta^{2m'} \\ 
 &\qquad 
 -m' \int_{e^s\mathbb{T}^{d}}\left( |\nabla_\theta \nabla^{m'-1} \tilde U|^2 + | \nabla_\theta \nabla^{m'-1} \tilde S|^2 \right) \left( 1 + 2 \frac{U_d}{R} - 2\alpha |\p_R S_d |  \right)  \beta^{2m'} \\ 
& 
\qquad + m'\int_{e^s\mathbb{T}^{d}}( |\nabla^{m'} \tilde U|^2 + | \nabla^{m'} \tilde S|^2 ) \beta^{2m'} \left(y \cdot \frac{\nabla \beta}{2\beta} +U_{d}\cdot \frac{2\nabla \beta}{\beta} + \alpha S_d \frac{2 |\nabla \beta |}{\beta }\right) 
\end{align*}

Finally, let $$\Xi = \min \{  1 + 2 \p_R U_p - 2\alpha |\p_R S_d |,  1 + 2 \frac{U_p}{R} - 2\alpha |\p_R S_d | \} $$ and recall that $\min \Xi > 0$ due to the properties of the profile. Note also that $\Xi$ tends to $1$ at infinity. Thus
\begin{align*}
 \p_s E_{\rm{low}}  &\leq  \delta_{\rm{low}}^{5/2}  e^{-\frac52 \eps (s-s_0)}  +  C' E_{\rm{low}}+  C_{m'} E_{\rm{low}}^{1/2} \sum_{i=0}^{m'-1} \left[ \int \left( | \nabla^{i} \tilde U |^2 + |\nabla^{i} \tilde S|^2 \right) \frac{\beta^{2m'}}{\langle y \rangle^{2(m'-i-1)} } \right] \\
& \qquad + m' \int ( |\nabla^{m'} \tilde U|^2 + | \nabla^{m'} \tilde S|^2 ) \beta^{2m'} \left( - \Xi  + \left( 1  + \frac{2 U_p}{R}   + 2\alpha \frac{S_d}{R} \right)  \frac{R \nabla \beta}{\beta}\right) 
\end{align*}

Now, we pick $\eps$ sufficiently small so that $2\eps \leq \min \Xi$. Then, since $\Xi \to 1$, $U_p/R \to 0$, $S_d/R \to 0$, we can pick $R_0$ sufficiently large so that the right parenthesis is always bigger than $3\eps$.

We also bound
\begin{equation*}
\int_{e^s\mathbb{T}^{d}}\left( | \nabla^{i} \tilde U |^2 + |\nabla^{i} \tilde S|^2 \right) \frac{\beta^{2m'}}{\langle y \rangle^{2(m'-i-1)} } \les \delta_{\rm{low}}^\theta  e^{-\theta \eps (s-s_0) }E_{\rm{low}}^{\frac{1-\theta}{2}}
\end{equation*}
by interpolation via Lemma \ref{lemma:GN_general01} in the appendix  with the substitutions  $(f,m,p,q,i)=(\tilde{U}, m',2,\infty, i)$, and $(f,m,p,q,i)=(\tilde{S}, m',2,\infty, i)$   for some $\theta \geq \frac{1}{2m'}$. Since $\delta_{\rm{low}} \ll E_{\rm{low}}$, we obtain 
\begin{equation*}
\p_s \left( E_{\rm{low}} e^{2\eps (s-s_0)}\right) \leq \delta_{\rm{low}}^{5/2} e^{-\frac12 \eps (s-s_0)} + C \left(  E_{\rm{low}} e^{2\eps (s-s_0)} \right) + C_{m'} \sup_{\theta \in [1/(2m'), 1]}\delta_{\rm{low}}^\theta E_{\rm{low}}^{1-\theta/2} e^{2\eps (1-\theta/2) (s-s_0)} - \eps m'E_{\rm{low}}e^{2\eps (s-s_0)} 
\end{equation*}
Letting $\bar E_{\rm{low}} =  E_{\rm{low}} e^{2\eps (s-s_0)}$, and picking $m'$ sufficiently large so that $\eps m'/2 > C$ we get
\begin{equation*}
\p_s \bar E_{\rm{low}}  \leq \delta_{\rm{low}}^{5/2} e^{-\frac12 \eps (s-s_0)} + C_{m'} \sup_{\theta \in [1/(2m'), 1]}\delta_{\rm{low}}^\theta \bar E_{\rm{low}}^{1-\theta/2} - \frac{\eps m'}{2} \bar E_{\rm{low}} 
\end{equation*}
Let us now suppose by contradiction that there exists a first time $s'$ at which $\bar E_{\rm{low}}^{1/2} = K_{m'} \delta_{\rm{low}}$ for $K_{m'} = 2 \langle C_{m'} \rangle^{2m'}$ (the reason of this choice will be clear after the computation). Therefore at $s=s'$ we would have that $\p_s \bar E_{\rm{low}} \geq 0$, so that 
\begin{align*}
0 &\leq \delta_{\rm{low}}^{5/2} e^{-\frac12 \eps (s-s_0)} + C_{m'} \delta_{\rm{low}}^{1/(2m')} \left(  K^2 \delta_{\rm{low}} \right)^{2-2/(4m')} - \frac{ \eps m'}{2} K^2 \delta_{\rm{low}}^2 \\
&\leq  \delta_{\rm{low}}^2 \left(1 + C_{m'} K^{1 - 1/(2m')} - \frac{\eps m'}{2} K \right)
\end{align*}
Taking $m'$ to be larger than $\frac{2}{\eps}$ (this choice only depends on the profiles) and then taking $K_{m'} = 2 \langle C_{m'} \rangle^{2m'}$, the last term is dominant in the parenthesis above and we get a contradiction. This shows that $\bar E_{\rm{low}}^{1/2} \leq K_{m'} \delta_{\rm{low}}$ for all $s \in [s_0, s_1]$, concluding the desired bound
\begin{equation*}
E_{\rm{low}}^{1/2} \les_{m'} \delta_{\rm{low}} e^{-\eps (s-s_0)}.
\end{equation*}

\subsection{Closing the low derivative energy bounds for $w$}

Let us recall the equation satisfied by $w$:
\begin{equation*}
    \p_s w = -2\nabla w \cdot U - y\nabla w  - \text{div}(U) + \frac{1-r}{2\alpha}
\end{equation*}
Letting $E_w = \int \beta^{2m'} | \nabla^{m'-1} w |^2$, we have that
\begin{align*}
\frac{1}{2} \p_s E_w &= -\int_{e^s\mathbb{T}^{d}} \beta^{2m'} U \cdot \nabla \left( |\nabla^{m'-1} w|^2 \right) - 2(m'-1) \int_{e^s\mathbb{T}^{d}} \beta^{2m'} \nabla^{m'-1} w \cdot \nabla U \cdot \nabla^{m'-1} w \\
&\quad - (m'-1) \int_{e^s\mathbb{T}^{d}} | \nabla^{m'-1} w |^2 \beta^{2m'} - \frac12 \int \beta^{2m'} y\cdot \nabla \left( | \nabla^{m'-1} w |^2 \right)\\
&\quad + O(E_{\rm{low}}^{1/2} E_w^{1/2} )+ E_w^{1/2}  \cdot O_{m'} \left( \sum_{j=1}^{m'-2} \int \beta^{2m'} | \nabla^j w |^2 |\nabla^{m'-j} U |^2 \right)^{1/2}+\frac{1}{2}\int_{e^s\mathbb{T}^{d}} y\cdot \vec{n}(|\nabla^{m'-1}w|^2)\beta^{2m'} \\
&=  (m'-1)\int_{e^s\mathbb{T}^{d}}  \beta^{2m'} \left( -1 + \frac{\nabla \beta \cdot (y+2U_p)}{\beta} \right) |\nabla^{m'-1} w|^2  -  \int_{e^s\mathbb{T}^{d}} 2\beta^{2m'} \left( |\p_R \nabla^{m'-2} w|^2 \p_R U_p + | \nabla_\theta \nabla^{m'-2} w|^2 \frac{U_p}{R} \right)\\
&\quad  + O( E_w )+ E_w^{1/2}  \cdot O_{m'} \left( \sum_{j=1}^{m'-2} \int_{e^s\mathbb{T}^{d}} \beta^{2m'} | \nabla^j w |^2 |\nabla^{m'-j} U |^2 \right)^{1/2} \\
\end{align*}
where we used \eqref{eq:radialcomp} in order to rewrite the $\nabla^{m'-1}w\cdot \nabla U \cdot \nabla^{m'-1} w$ term.

 Now, using the profile properties in the same way as we did for $E_{\rm{low}}$ (see subsection \ref{lowderivativeenergy}), we see that $m'$ sufficiently large ensures
\begin{equation*}
\p_s E_w \leq - E_w + E_w^{1/2}  \cdot O_{m'} \left( \sum_{j=1}^{m'-2} \int_{e^s\mathbb{T}^{d}}  \beta^{2m'} | \nabla^j w |^2 |\nabla^{m'-j} U |^2 \right)^{1/2}
\end{equation*}
Lastly, interpolation yields
\begin{equation} \label{eq:girona}
\p_s E_w \leq -E_w +  C(m') E_w^{1-1/(2m')},
\end{equation}
since any bound on $\nabla^i U$ follows from the $E_{\rm{low}}$ assumption and the fact that the profile is smooth. We see that there exists some sufficiently large constant $C'(m')$ such that $E_w = C'$ makes the right hand side of \eqref{eq:girona} negative, and therefore we have a uniform bound $E_w \leq C'(m')$ as we wanted.

\subsection{Closing the low derivative $L^{\infty}$ estimate}
The direct control of the $X$ norm \eqref{eq:Hm_bound} ($\dot H^m$ in a compact domain) closes our bootstrap assumption for $|y| \leq C_0$, so we only need to close the estimates of $|\tilde S / S_d|, |\tilde \Psi / \Psi_p|$, and $\tilde{S}$, $\tilde{\Psi}$ in the region $|y| \geq C_0$. Moreover, the control in the $\dot H^{m'}$ norm gives us control in $L^\infty$ with constant $\delta_{\rm{low}}$ by the Sobolev embedding.
\begin{lemma}\label{lemmalinfinity}
Under the bootstrap assumptions \eqref{eq:bootstrap_low_der00}, \eqref{eq:bootstrap_low_der02}, \eqref{eq:bootstrap_Hm}, \eqref{highordermestimate}, we have
\begin{equation} \label{eq:bootstrap_low_der*}
\left| \frac{\tilde S}{S_d} \right| < \frac{1}{2}\delta_{\rm{low}}, 
\end{equation}
and
\begin{equation} \label{eq:bootstrap_low_der02*}
\left| \tilde S \right| <\frac{1}{2} \delta_{\rm{low}} e^{-\eps (s-s_0)}, \qquad \mbox{ and }\qquad \left|\tilde \Psi\right| < \frac{1}{2}\delta_{\rm{low}} e^{-\eps (s-s_0)}
\end{equation}
\end{lemma}
\begin{proof}
First of all, by taking $(r,p,q, i,m,\phi)=(\infty,\infty,2,i,m',\beta)$ in interpolation lemma \ref{lemma:GN_general01}, we have that 
\begin{equation*}
\|\beta \nabla \tilde{U}\|_{L^{\infty}} +\|\beta \nabla \tilde{\Psi}\|_{L^{\infty}} \lesssim_{m'}\delta_{\rm{low}}e^{-\epsilon(s-s_0)}.
\end{equation*}
Thus we obtain
\begin{align}\nonumber
| \mathcal N_S | &\les_{m'} | \nabla \tilde S | | \tilde U| + |  \tilde S | | \nabla \tilde U | \les \delta_{\rm{low}}^{3/2}   S_d e^{-\eps (s-s_0)}\langle y\rangle^{-\frac{1}{10}}, \\
| \mathcal N_{\Psi} | &\les_{m'}  | \tilde S |^2 + |  \nabla \tilde \Psi |^2   \les \delta_{\rm{low}}^{3/2} e^{-\eps (s-s_0)}\langle y\rangle^{-\frac{1}{5}}.
\label{eq:NPsi}
\end{align}
We recall that all the inequalities refer to the region $|y| \geq C_0$ so one should think about the profiles $\Psi_d$, $S_d$ to be interchangeable with their asymptotic behaviour for $|y| \geq C_0$.

We also have that
\begin{equation*}
| \mathcal D | \les e^{(4-2r)s} \left( \frac{| \nabla S |^2}{S^2} + \frac{|\Delta S |}{S}\right) \les e^{(4-2r)s} \left( \frac{| \nabla S |^2}{S_d^2} + \frac{|\Delta S |}{S_d}\right)
\end{equation*}
Now, using the interpolation between the $H^{m'}$ bounds and \eqref{eq:bootstrap_low_der00}, we can gain almost $\frac{1}{5}$ of a derivative, and any constant gets absorbed by $e^{(4-2r)s_0}$, which means that we obtain the bound
\begin{equation*}
| \mathcal D | \les \delta_{\rm{low}}^{3/2} e^{-\frac32\epsilon (s-s_0)} \frac{1}{\langle y \rangle^{\frac{1}{10}}}.
\end{equation*}


We clearly have the bounds
\begin{equation} \label{eq:moldova}
| \nabla^i S_d | \les_i \frac{S_d}{\langle y \rangle^i}, \qquad \mbox{ and } \qquad | \nabla^i \Psi_d | \les_i \frac{\Psi_d}{\langle y \rangle^i}
\end{equation}
Interpolating \eqref{eq:bootstrap_low_der02} with the $H^{m'}$ bounds  \eqref{eq:bootstrap_Hm}, we obtain:
\begin{align*}
| \nabla^i \tilde S | \les_{m'} e^{-\eps (s-s_0)}\delta_{\rm{low}} \frac{1}{\langle y \rangle^{\frac{m'}{10}\theta}} \les 
 e^{-\eps (s-s_0)} \delta_{\rm{low}} \frac{1}{\langle y \rangle^{\frac{i}{10}}}, \\
 | \nabla^i \tilde \Psi | \les_{m'} e^{-\eps (s-s_0)}\delta_{\rm{low}} \frac{1}{\langle y \rangle^{\frac{(m'+1)}{10}\theta}} \les 
 e^{-\eps (s-s_0)} \delta_{\rm{low}}\frac{1}{\langle y \rangle^{\frac{i}{10}}}
\end{align*} Moreover, we have 

\begin{align*}
\left| \mathcal L_\Psi + (r-2)\tilde \Psi + y\nabla \tilde \Psi \right| &\les | \nabla \Psi_d | | \nabla \tilde \Psi | + S_d |\tilde S |\les_{m'} \delta_{\rm{low}} e^{-\eps (s-s_0)}  \left( \frac{| \Psi_p |}{\langle y \rangle^{\frac{11}{10}}} +  S_d\right)\les_{m'} \delta_{\rm{low}} e^{-\eps (s-s_0)}\langle y \rangle ^{-(r-1)}\\
\left| \mathcal L_S + (r-1) \tilde S + y\nabla \tilde S \right| &\les | \nabla S_d | | \nabla \tilde \Psi| + | \nabla \tilde S | | \nabla \Psi_p | + S_d  |\Delta \tilde \Psi| + | \tilde S | | \Delta \Psi_d |  \\
&\les_{m'} \delta_{\rm{low}} e^{-\eps (s-s_0)}  \left( \frac{S_d}{\langle y \rangle^{\frac{1}{5}}} \right)
\end{align*}
In order to deal with $\nabla \tilde{\Psi},\nabla \tilde{S}, \Delta \tilde{\Psi}$ we do interpolation with $E_{\rm{low}}$.

Therefore, we obtain that
\begin{align*}
\left| (\p_s + (r-2)  + y\nabla) \tilde \Psi - \mathcal E_\Psi\right| &\les_{m'}  \delta_{\rm{low}} e^{-\eps (s-s_0)} \left( \delta_{\rm{low}}^{1/2}  \frac{1}{\langle y\rangle^{\frac{1}{10}}} + \frac{1}{\langle y \rangle^{(r-1)}} \right)\les_{m'}  \delta_{\rm{low}} e^{-\eps (s-s_0)} \langle y\rangle^{-\frac{1}{10}} 
\\
\left| (\p_s + (r-1)  + y\nabla) \tilde S  - \mathcal E_S\right| &\les_{m'}  \delta_{\rm{low}} e^{-\eps (s-s_0)}  S_d \left( \delta_{\rm{low}}^{1/2}\langle y\rangle^{-\frac{1}{10}}   +   \frac{1}{\langle y \rangle^{\frac{1}{10}}} \right)\lesssim_{m'}\delta_{\rm{low}} e^{-\eps (s-s_0)}  S_d \langle y \rangle^{-\frac{1}{10}}.
\end{align*}

We will work with trajectories as follows. We will work in the region $s \geq s_0$, $|y| \geq C$, such that $(y_\star, s_\star)$ is always  on the boundary (that is $(y_\star, s_\star) \in B(0, C)^c \times \{ s_0 \} \cup \{ -C, C \} \cup [s_0, \infty)$). The point $(y, s)$ will be on the exponential trajectory starting from $(y_\star, s_\star)$, so that $y = y_\star e^{s-s_\star}$.

We consider $\omega_{\Psi, (y_\star, s_\star)} (s) :=  e^{(r-2)(s-s_\star)} \tilde \Psi (y_\star e^{s-s_\star}, s)$, and $\omega_{S, (y_\star, s_\star)} (s) := e^{(r-1)(s-s_\star)} \tilde S (y_\star e^s_\star, s)$. Similarly, we define $\bar \omega_{\Psi, y_0}$ and $\bar \omega_{S, y_0}$ relative to $\Psi_d$, $S_d$. We obtain that
\begin{align} \begin{split} \label{eq:zagreb} 
| \p_s \omega_{\Psi, (y_\star, s_\star) } | &\les e^{(r-2)(s-s_\star)} \mathcal E_\Psi (y_\star e^{s-s_\star}) + \delta_{\rm{low}} e^{-\eps (s-s_0)}\frac{1}{\langle y_{*} e^{s-s_\star }\rangle ^{\frac{1}{10}}}, \\
| \p_s \omega_{S, (y_\star, s_\star)} | &\les e^{(r-1)(s-s_\star)} \mathcal E_S (y_\star e^{s-s_\star}) + \delta_{\rm{low}} e^{-\eps (s-s_0)} \bar \omega_{S, (y_\star, s_\star)} \frac{1}{\langle y_{*} e^{s-s_\star }\rangle ^{\frac{1}{10}}}
\end{split} \end{align}
Note that from \eqref{errores01},
\begin{align*}
|\mathcal E_\Psi | \les \mathbbm{1}_{|y| \geq  e^s} \frac{1}{\langle y \rangle ^{2(r-1)}}\\
|\mathcal E_s | \les \mathbbm{1}_{|y| \geq  e^s} \frac{S_{d}}{\langle y \rangle ^{(r-1)}}.
\end{align*}

From \eqref{eq:zagreb}, we conclude
\begin{align} \begin{split} \label{eq:oman}
| \p_s \omega_{\Psi, (y_\star, s_\star)} | &\les_{m'}  e^{(r-1)(s-s_\star )}\left(\frac{\mathbbm {1}_{|y_\star|\geq e^{s_\star }}}{\langle y_\star e^{s-s_\star }\rangle^{r-1}}+\delta_{\rm{low}}e^{-\epsilon(s-s_0)}\frac{1}{\langle y_\star e^{s-s_\star }\rangle^{\frac{1}{10}}}\right)  \\
| \p_s \omega_{S, (y_\star, s_\star)} | &\les_{m'} |\bar \omega_{S, (y_\star, s_\star)}| \left( \frac{\mathbbm{1}_{|y_\star| \geq e^{s_\star}}}{ \langle y_\star e^{s-s_\star} \rangle^{r-1 } } +  \delta_{\rm{low}} e^{-\eps (s-s_0)}  \frac{1}{\langle y_\star e^{s-s_\star }\rangle^{\frac{1}{10}}} \right)
\end{split} \end{align}

\textbf{Proof of \eqref{eq:bootstrap_low_der02*}.}\\
First, let us observe
\begin{align} \begin{split} \label{eq:kuwait}
| \bar \omega_{S, (y_\star, s_\star)} | &= e^{(r-1)(s-s_\star)} \hat{\mathfrak{X}}_d(y_\star/(e^s_\star))  S_p (|y_\star| e^{s-s_\star}) +e^{-(r-1)s_{\star}}\tilde{\mathfrak{X}}_d(|y_\star| e^{s-s_\star})\\
&\les \hat{\mathfrak{X}}_d (y_\star/(e^s_\star)) \frac{1}{\langle y_\star \rangle^{r-1}}+e^{-(r-1)s_\star}\leq \frac{1}{C} \ll 1.
\end{split} \end{align}
Now, let us recall that our bound \eqref{eq:bootstrap_low_der02*} follows trivially for $|y| \leq C$ from the linear stability bound. Recall $(y_\star, s_\star)$ is situated on the boundary of $|y| \geq C$, $s \geq s_0$, and its exponential trajectory reaches $(y, s)$ (that is, $y_\star e^{s-s_\star} = y$). At such boundary point $(y_\star, s_\star)$, the bootstrap bound is satisfied with an improved constant $\frac{1}{10}$.

Now, we combine \eqref{eq:kuwait} with \eqref{eq:oman} noting that we can absorb the constant (that depends on $m'$) of \eqref{eq:oman} with the $\frac{1}{C}$ gain from \eqref{eq:kuwait}: 

\begin{align}\label{eq:caribe}
|\omega_{\Psi, (y_\star, s_\star)} (s)| &\leq  |\omega_{\Psi, (y_\star, s_\star)} (s_\star)| + \int_{s_\star}^{s} e^{(r-1) (s'-s_\star)} \frac{\mathbbm{1}_{|y_\star| \geq  e^{s_\star}}}{\langle y_\star e^{s'-s_\star}\rangle^{r-1}} ds' \\\nonumber
&\quad+ \delta_{\rm{low}} e^{-\eps (s_\star - s_0)} \int_{s_\star}^{s} e^{(r-1-\eps) (s'-s_\star)} \left(\frac{1}{\langle y_\star e^{s'-s_\star}\rangle^{\frac{1}{10}}}\right)  ds' \notag \\\nonumber
&\leq
 \frac{\delta_{\rm{low}}}{10} e^{-(s_\star - s_0)} + \frac{\mathbbm{1}_{|y_\star| \geq  e^{s_\star}}}{\langle y_\star \rangle^{r-1}} (s-s_{\star}) + \delta_{\rm{low}} e^{-\eps (s_\star-s_0)} \left( \frac{1}{|y_\star|^{\frac{1}{10}}} \right)\frac{e^{(r-1-\epsilon)(s-s_{\star})}-1}{r-1-\epsilon}
\end{align}
where we used that $\frac{1}{ e^{s_0}} \ll \frac{\delta_{\rm{low}}}{C}$, wtih $\delta_{\rm{low}} \ll \eps$ and $|y_\star| \geq C \gg 1$.  Undoing the change from $\omega_{\Psi, (y_\star, s_\star)}$ to $\tilde \Psi$, we have that
\begin{align*}
| \tilde \Psi (y, s) | =e^{-(r-1) (s-s_\star)} | \omega_{\Psi, (y_\star, s_\star)} (s) | &\leq \frac{1}{10} e^{-\eps (s-s_0)} \delta_{\rm{low}}+  e^{-(r-1) (s-s_\star)} \frac{\mathbbm{1}_{|y_\star| \geq  e^{s_\star}}}{\langle y_\star \rangle^{r-1}} (s-s_{\star})\\&\quad+ \delta_{\rm{low}} e^{-\eps (s_\star-s_0)} \left( \frac{1}{|y_\star|^{\frac{1}{10}}} \right)\frac{e^{(r-1-\epsilon)(s-s_{\star})}-1}{r-1-\epsilon}\\
&\leq \frac{1}{4}  e^{-\eps (s-s_0)} \delta_{\rm{low}},
\end{align*}
where we use $r-1\geq 1 \gg \epsilon $, and either $|y_\star|=C$, or $s_{\star}=s_0$.
The same reasoning can be applied in a completely analogous way to $\omega_{S, (y_\star, s_\star)}$ in order to deduce the bound for $\tilde S$. \\

\textbf{Proof of \eqref{eq:bootstrap_low_der*} } \\

Let us notice that from \eqref{eq:SSprofiles},
\begin{align*}
\frac{\p_s \bar \omega_{S, (y_\star, s_\star)}}{\bar \omega_{S, (y_\star,s_\star)}} &= \frac{\p_s \left( e^{(r-1)(s-s_\star)} S_p (y_\star e^{s-s_\star}) \hat{\mathfrak{X}}_d (y_\star e^{-s_\star})+e^{-(r-1)s_\star }\tilde{\mathfrak{X}}_d(y_\star e^{-s_{\star}})\right)}{e^{(r-1)(s-s_\star)}(S_p (y_\star e^{s-s_\star}) \hat{\mathfrak{X}}_d (y_{\star} e^{-s_\star})+e^{-(r-1)s)}\tilde{\mathfrak{X}}_d (y_{\star} e^{-s_\star})} \\
&= \frac{(r-1)S_p (y_\star e^{s-s_\star}) \hat{\mathfrak{X}}_d (y_\star e^{-s_\star})+\nabla S_{p}(y_\star e^{s-s_\star})\cdot y_{\star}e^{s-s_{\star}}\hat{\mathfrak{X}}_d (y_\star e^{-s_\star})}{ S_p (y_\star e^{s-s_\star}) \hat{\mathfrak{X}}_d(y_{\star} e^{-s_\star})+e^{-(r-1)s}\tilde{\mathfrak{X}}_d (y_{\star} e^{-s_\star})} \\
& = O \left( \frac{(|\nabla S_p \nabla \Psi_p | + | S_p \Delta \Psi_p |)\Big|_{x = y_\star e^{s-s_\star}}(\hat{\mathfrak{X}}_d)(y_{\star} e^{-s_\star}) }{S_p(y_\star e^{s-s_\star})\hat{\mathfrak{X}}_d(y_{\star} e^{-s_\star})+e^{-(r-1)s}\tilde{\mathfrak{X}}_d (y_{\star} e^{-s_\star})} \right) = O \left( \frac{1}{\langle y_\star e^{s-s_\star} \rangle^{r}}\right).
\end{align*}

Combining that with \eqref{eq:oman}, 
\begin{align*}
\p_s \frac{\omega_{S, (y_\star, s_\star)}}{ \bar \omega_{S, (y_\star, s_\star) }} &\les_{m'} \left( \frac{\mathbbm{1}_{|y_\star| \geq e^{s_\star}}}{ \langle y_\star e^{s-s_\star} \rangle^{r-1 } } +  \delta_{\rm{low}} e^{-\eps (s-s_0)}  \frac{1}{\langle y_\star e^{s-s_\star }\rangle^{\frac{1}{10}}} \right)+  \frac{\omega_{S, (y_\star, s_\star)}}{\bar \omega_{S, (y_\star, s_\star) }}\frac{1}{\langle y_\star e^{s-s_\star} \rangle^r} \\
&\les \frac{\mathbbm{1}_{|y_\star| \geq  e^{s_\star}}}{\langle y_\star  e^{s-s_\star} \rangle^{r-2}} +\delta_{\rm{low}} e^{-\eps (s-s_0)}  \frac{1}{\langle y_\star e^{s-s_\star }\rangle^{\frac{1}{20}}}+\delta_{\rm{low}} \frac{1}{\langle y_\star e^{s-s_\star} \rangle^{r-1} } 
\end{align*}
where we used the bootstrap assumption $\frac{\tilde S}{S_d} \leq \delta_{\rm{low}}$, which implies that $\left| \frac{\omega_{\tilde{S}, (y_\star, s_\star)}}{\bar \omega_{S_d, (y_\star, s_\star) }} \right| \leq \delta_{\rm{low}}$.

Now, let us recall that in the range $|y| \leq C$ our bounds follow from the linear stability bounds and Sobolev embedding. For any $(y, s)$ with $|y| \geq C$ and $s \geq s_0$ we find an exponential trajectory starting at the point $(y_\star, s_\star )$ on the boundary of the region. That is, either $|y_\star| = C$ or $s_\star = s_0$. Moreover, since $(y, s)$ is in that trajectory, we have that $y = y_\star e^{s-s_\star}$. Let us note that at the boundary points $(y_\star, s_\star)$ our bootstrap bound holds with an improved bound $\frac{1}{10}$.

Now, similar to the integral in \eqref{eq:caribe}:
\begin{align*}
\frac{\omega_{S, (y_\star, s_\star)}(s)}{\bar \omega_{S, (y_\star, s_\star)}(s)} &\leq \left| \frac{\tilde S(y_\star, s_\star)}{ S_d (y_\star)} \right|
+ \int_{s_\star}^{s} \frac{\mathbbm{1}_{|y_\star| \geq  e^{s_\star}}}{\langle y_\star e^{s-s_\star}\rangle^{r-2}} ds'  + \delta_{\rm{low}}   \int_{s_\star}^{s} \frac{1}{\langle y_\star e^{s-s_\star}\rangle^{\frac{1}{20}}} ds' + \delta_{\rm{low}} \int_{s_\star}^s  \frac{1}{\langle y_\star e^{s-s_\star}\rangle^{r-1}}  ds' \notag \\
&\leq \frac{1}{10} \delta_{\rm{low}} + \frac{1}{ e^{s_0(r-2)}} + \frac{40\delta_{\rm{low}}}{C}  \leq \frac{1}{5} \delta_{\rm{low}}.
\end{align*}
Now, we undo the change of variables from $S$ to $\omega_{S}$ and note that 
\begin{equation*}
\left| \frac{\omega_{S, (y_\star, s_\star)}(s)}{\bar \omega_{S, (y_\star, s_\star)}(s)} \right| = \left| \frac{e^{(r-2)(s-s_\star)} \tilde S(y_\star e^{s-s_\star}, s) }{ e^{(r-2)(s-s_\star)} S_d (y_\star e^{s-s_\star}) } \right| = \frac{|\tilde S(y, s)|}{ |S_d (y)| }.
\end{equation*}
\end{proof}

 \subsection{Closing the high derivative bounds}
Let $w = \frac{1}{2}\log (P)$. Recall that we have equations:
\begin{align}\label{equationhighener}
\begin{cases}
&\partial_{s}\Psi=-|\nabla \Psi|^2-\alpha S^2 +e^{(4-2r)s}(\Delta w+|\nabla w|^2)-y\cdot \nabla \Psi-(r-2)\Psi\\
&\p_{s}P=-2\nabla P\cdot \nabla \Psi-2P\Delta \Psi-y\cdot \nabla P+\frac{1-r}{\alpha}P\\
&\p_{s}S=-2\nabla S\cdot \nabla \Psi -2\alpha S \Delta \Psi-y\cdot \nabla S-(r-1)S\\
&\p_{s}w=-2\nabla w\cdot \nabla\Psi-\Delta\Psi-y\cdot \nabla w+\frac{1-r}{2\alpha}
\end{cases}
\end{align}
 
 We first introduce the following lower order estimates given by the bootstrap argument and interpolation lemma \ref{lemma:GN_general01}:
\begin{lemma}\label{lowerinterpolation}
Under the bootstrap assumptions \eqref{eq:bootstrap_low_der00}, \eqref{eq:bootstrap_low_der02}, \eqref{eq:bootstrap_Hm},\eqref{highordermestimate}  , we have for $ j=1,2$
\[
\|\nabla^{j}P\|_{L^{\infty}}+\|\nabla^{j}\Psi\|_{L^{\infty}}+\|\nabla^{j}w\|_{L^{\infty}}\lesssim 1.
\]

\begin{proof}
The estimates for
$\|\nabla^{j}P\|_{L^{\infty}}+\|\nabla^{j}\Psi\|_{L^{\infty}}$ directly follow from the interpolation between \eqref{eq:bootstrap_low_der00} and \eqref{highordermestimate}.
For $\|\nabla w\|_{L^{\infty}}$, from \eqref{perturb defi} and \eqref{eq:bootstrap_low_der00} we have 
\begin{align*}
|\nabla w|&= \left|\frac{1}{2\alpha}\frac{\nabla S}{S}\right|=\left|\frac{1}{2\alpha}\frac{\nabla S_{d}}{S}+\frac{1}{2\alpha}\frac{\nabla \tilde{S}}{S}\right|\lesssim 1+\left|\frac{\nabla\tilde{S}}{S_{d}}\right|.
\end{align*}
Moreover, from \eqref{highordermestimate} and $k\ll E$, we have
\begin{equation}\label{highordermestimateforperturb}
\|\nabla^{k-l}\tilde{S}\|_{L^{2}(P\phi^{l})}\lesssim\|\nabla^{k-l}S\|_{L^{2}(P\phi^{l})}+\|\nabla^{k-l}S_d\|_{L^{2}(P\phi^{l})}\lesssim E.
\end{equation}
By interpolation between \eqref{eq:bootstrap_low_der00} and \eqref{highordermestimateforperturb} when $l=\frac{1}{10}k$, we have
\[
\|\nabla \tilde{S}(\frac{1}{S_{d}})^{\theta}\phi^{\frac{l}{2}(1-\theta)}\sqrt{P}^{1-\theta}\|_{L^{\infty}}\lesssim \|\frac{\tilde{S}}{S_{d}}\|_{L^{\infty}}^{\theta}\|\nabla^{k-l}\tilde{S}\|_{L^{2}(P\phi^{l})}^{1-\theta}\lesssim \delta_{\rm low}^{\frac{1}{2}}E^{\frac{1}{2}}\lesssim 1,
\]
with $\theta=1-\frac{\frac{1}{d}}{\frac{k-l}{d}-\frac{1}{2}}$.
Then
\[
\|\frac{\nabla \tilde{S}}{S_{d}}\|_{L^{\infty}}\lesssim \|(\frac{1}{S_{d}})^{1-\theta}P^{-\frac{1}{2}(1-\theta)}\phi^{-\frac{l}{2}(1-\theta)}\|_{L^\infty}\lesssim 1.
\]
$\|\nabla^2 w\|$ can be controlled in a similar way.
\end{proof}
\end{lemma}
\begin{lemma}\label{interpolation01}
We have the following interpolation estimates:
For $j+j'=k+2-l$, $j,j'\geq 3$, $0\leq l\leq \frac{k}{10}$, we have
\begin{align}\label{inter01}
e^{\frac{1}{2}(4-2r)s}\|\nabla^{j} w \nabla ^{j'}\Psi\|_{L^{2}(P\phi^{l})}\leq O(E_{l,0}^{\frac{1}{2}-\frac{1}{4k}}),
\end{align}
\begin{align}\label{inter02}
e^{(4-2r)s}\|\nabla^{j} w \nabla ^{j'}w\|_{L^{2}(P\phi^{l})}\leq O(E_{l,0}^{\frac{1}{2}-\frac{1}{4k}}),
\end{align}
\begin{align}\label{inter03}
\|\nabla^{j} \Psi \nabla ^{j'}\Psi\|_{L^{2}(P\phi^{l})}\leq O(E_{l,0}^{\frac{1}{2}-\frac{1}{4k}}).
\end{align}
For $j+j'=k+1-l$, $j'\geq 3$, $j\geq 2$, we have
\begin{align}\label{inter04}
\|\nabla^{j} S \nabla ^{j'}\Psi\|_{L^{2}(P\phi^{l})}\leq O(E_{l,0}^{\frac{1}{2}-\frac{1}{4k}}).
\end{align}
For $j+j'=k-l$, $j\geq 2$, $j'\geq 2$, we have
\begin{align}\label{inter05}
\|\nabla^{j} S \nabla ^{j'}S\|_{L^{2}(P\phi^{l})}\leq O(E_{l,0}^{\frac{1}{2}-\frac{1}{4k}}).
\end{align}
\end{lemma}
\begin{proof}

We will use the  interpolation Lemma \ref{lemma:GN_general01}. We also claim that weight $P$ satisfies the conditions in the lemma uniformly in time $s$. In fact, from \eqref{errores05}, we have $S_d$ satisfies the weight condition. Moreover, \eqref{rescaleden} and the bootstrap bound \eqref{eq:bootstrap_low_der*} tells us the weight condition holds for $P$ as long as the weight condition holds for $S_d$. 

We start with \eqref{inter03}.
First, 
\begin{align*}
&\quad \|\nabla^{j}\Psi\nabla^{j'}\Psi\|_{L^2(P\phi^{l})}\\
&\lesssim \|\nabla^{j}\Psi_{p}\nabla^{j'}\Psi_{p}\|_{L^2(P\phi^{l})}+\|\nabla^{j}\Psi_{p}\nabla^{j'}\tilde{\Psi}\|_{L^2(P\phi^{l})}+\|\nabla^{j}\tilde{\Psi}\nabla^{j'}\tilde{\Psi}\|_{L^2(P\phi^{l})}\\
&\lesssim O(k)+\|\nabla^{j}\Psi_{p}\nabla^{j'}\tilde{\Psi}\|_{L^2(P\phi^{l})}+\|\nabla^{j}\tilde{\Psi}\nabla^{j'}\tilde{\Psi}\|_{L^2(P\phi^{l})}.
\end{align*}
For $\|\nabla^{j}\tilde{\Psi}\nabla^{j'}\tilde{\Psi}\|_{L^2(P\phi^{l})}$, by letting $\theta_1=\frac{j-2}{k-l-2}$, $\theta_2=\frac{j'-2}{k-l-2}$, $\frac{\sigma_1}{2}=\frac{j-2}{d}+\theta_1(\frac{1}{2}-\frac{k-l-2}{d})$, $\frac{\sigma_2}{2}=\frac{j'-2}{d}+\theta_2(\frac{1}{2}-\frac{k-l-2}{d}),$
 we have
\begin{align*}
&\|\nabla^{j}\tilde{\Psi}\nabla^{j'}\tilde{\Psi}\|_{L^2(P\phi^{l})}\\
\lesssim &\|\nabla^{j}\tilde{\Psi}(P\phi^{l})^{\frac{\theta_1}{2}}\|_{L^{\frac{2}{\sigma_1}}}\|\nabla^{j'}\tilde{\Psi}(P\phi^{l})^{\frac{\theta_2}{2}}\|_{L^{\frac{2}{\sigma_2}}}\\
\lesssim_{k} & \|\nabla^{k}\tilde{\Psi}(P\phi^{l})^{\frac{1}{2}}\|_{L^{2}}^{\theta_1} \|\nabla^{2}\tilde{\Psi}\|_{L^{\infty}}^{1-\theta_1}\|\nabla^{k}\tilde{\Psi}(P\phi^{l})^{\frac{1}{2}}\|_{L^{2}}^{\theta_2} \|\nabla^{2}\tilde{\Psi}\|_{L^{\infty}}^{1-\theta_2}\\
\lesssim_{k}&\|\nabla^{k}\tilde{\Psi}(P\phi^{l})^{\frac{1}{2}}\|_{L^{2}} \|\nabla^{2}\tilde{\Psi}\|_{L^{\infty}}\\
\lesssim_{k} & E_{l,0}^{\frac{1}{2}}\delta_{0}^{\frac{1}{2}}\leq E_{l,0}^{\frac{1}{2}-\frac{1}{4k}}.
\end{align*}
Here we use the estimate for $\|\nabla S\|_{L^{\infty}}$, $\|\nabla^2 \Psi\|_{L^{\infty}}$ in Lemma \ref{lowerinterpolation} and the interpolation lemma \ref{lemma:GN_general01} by letting $(f, m, p ,q, i, \frac{1}{\bar{r}},\theta)=(\nabla^{2}\tilde{\Psi}, k-l-2, \infty, 2,j-2, \frac{\sigma_1}{2}, \theta_1)$ or $(\nabla^{2}\tilde{\Psi}, k-l-2, \infty, 2,j-2, \frac{\sigma_2}{2}, \theta_2).$
Similarly, we have
\begin{align*}
&\|\nabla^{j}\Psi_{p}\nabla^{j'}\tilde{\Psi}\|_{L^2(P\phi^{l})}\\
\lesssim_{k} & \|\nabla^{k}\Psi_{p}(P\phi^{l})^{\frac{1}{2}}\|_{L^{2}}^{\theta_1} \|\nabla^{2}\Psi_{p}\|_{L^{\infty}}^{1-\theta_1}\|\nabla^{k}\tilde{\Psi}(P\phi^{l})^{\frac{1}{2}}\|_{L^{2}}^{\theta_2} \|\nabla^{2}\tilde{\Psi}\|_{L^{\infty}}^{1-\theta_2}\\
\lesssim_{k} &E_{l,0}^{\frac{1}{2}}\delta_{0}^{\frac{1}{k-l-2}}\\
\leq &E_{l,0}^{\frac{1}{2}-\frac{1}{4k}},
\end{align*}
where we use $\delta_{0} \ll \frac{1}{E_{l,0}^{\frac{1}{2}}}.$
We claim \eqref{inter04} can be controlled similarly by taking $\theta_1=\frac{j-1}{k-2}$, $\theta_2=\frac{j'-2}{k-2}$. \eqref{inter04} can be controlled similarly by taking $\theta_1=\frac{j-1}{k-2}$, $\theta_2=\frac{j'-1}{k-2}$.
For \eqref{inter02}, with $f=\nabla^{2}w, m=k-2,p=\infty, q=2$ and $i=j-1$ or $j'-1$, from interpolation lemma \ref{lemma:GN_general01} we have 
\begin{align*}
\|\nabla^{j}w\nabla^{j'}w\|_{L^2(P\phi^{l})}&\lesssim \|\nabla^{k}w (P\phi^{l})^{\frac{1}{2}}\|_{L^2}\|\nabla^{2}w\|_{L^{\infty}}\\
&\lesssim E_{l,0}^{\frac{1}{2}}e^{-\frac{1}{2}(4-2r)s}.
\end{align*}
Since $s\gg E_{l,0}$, we have \eqref{inter02}.
Moreover, from the estimate for $\|\nabla^{2}w\|_{L^{\infty}}$ in Lemma \ref{lowerinterpolation}, we have
\begin{align*}
\|\nabla^{j}w\nabla^{j'}\Psi\|_{L^2(P\phi^{l})}&\lesssim \|\nabla^{k}w (P\phi^{l})^{\frac{1}{2}}\|_{L^2}^{\theta_{1}}\|\nabla^{2}w\|_{L^{\infty}}^{1-\theta_1}\|\nabla^{k}\Psi(P\phi^{l})^{\frac{1}{2}}\|_{L^2}^{\theta_{2}}\|\nabla^{2}\Psi\|_{L^{\infty}}^{1-\theta_2}\\
&\lesssim E_{l,0}^{\frac{1}{2}}e^{-\frac{1}{2}(4-2r)s\theta_{1}}\\
&\lesssim E_{l,0}^{\frac{1}{2}}e^{-\frac{1}{2}(4-2r)s\frac{k-l-3}{k-l-2}}.
\end{align*}
Then \eqref{inter02} follows since  $s\gg kE$. 
\eqref{inter04} follows a similar way as \eqref{inter05}.
\end{proof}
Now we show the estimate \eqref{highordermestimate} using the following lemma.

\begin{lemma}
Under the bootstrap assumptions in Proposition \ref{prop:bootstrap}, we have the following estimate: 
\begin{align*}
\frac{dE_{k-l}(s)}{ds}\leq -\delta k E_{k-l}(s)+O(E_{l,0}),
\end{align*}
with some $\delta>0$ independent of $k$.
\end{lemma}
\begin{proof}

We have 
\begin{align*}
&\quad\frac{1}{2}\partial_{s}E_{k-l}(s)\\
=&\frac{1}{2}\int_{e^s\mathbb{T}^{d}}(|\nabla^{k-l-1}S|^2+|\nabla^{k-l}\Psi|^2+e^{(4-2r)s}|\nabla^{k-l}w|^2)\phi^{l}\partial_{s}P\\
&+\int_{e^s\mathbb{T}^{d}}P\phi^{l}(\nabla^{k-l-1}S)(\nabla^{k-l-1}\partial_{s}S)\\
&+\int_{e^s\mathbb{T}^{d}}P\phi^{l}(\nabla^{k-l}\Psi)(\nabla^{k-l}\partial_{s}\Psi)\\
&+e^{(4-2r)s}\int_{e^s\mathbb{T}^{d}} P\phi^{l} (\nabla^{k-l}w)\partial_{s}\nabla^{k-l}w+(4-2r)e^{(4-2r)s}\int_{e^s\mathbb{T}^{d}}P\phi^{l}|\nabla^{k-l}w|^2\\
&+\frac{1}{2}\int_{\partial e^s\mathbb{T}^{d}} y\cdot \vec{n}(|\nabla^{k-l-1}S|^2+|\nabla^{k-l}\Psi|^2+e^{(4-2r)s}|\nabla^{k-l}w|^2)\phi^{l}P \\
&=E_{k-l,P}+E_{k-l,S}+E_{k-l,\Psi}+O(E_{l,0})+E_{k-l,w}+E_{boundary}.
\end{align*}

From equation  \eqref{weightdefi}, we have
\begin{align*}
E_{k-l,P}=\frac{1}{2}\int_{e^s\mathbb{T}^{d}} (|\nabla^{k-l-1}S|^2+|\nabla^{k-l}\Psi|^2+e^{(4-2r)s}|\nabla^{k-l}w|^2)\phi^{l}[-2\nabla P\cdot \nabla \Psi-2P\Delta\Psi-y\cdot\nabla P+\frac{1-r}{\alpha}P].
\end{align*}
By interpolation lemma \ref{interpolation01}, we only need to control several top order terms for $E_{k-l,S}$, $E_{k-l,\Psi}$, and $E_{k-l,w}.$ We have 
\begin{align*}
\quad E_{k-l,S}&=\int_{e^s\mathbb{T}^{d}}P\phi^{l}\nabla^{k-l-1}S(-2\sum_{i=1}^{d}\nabla^{k-l-1}\partial_{i}S\cdot \partial_{i}\Psi-2\alpha S\Delta\nabla^{k-l-1}\Psi-\sum_{i=1}^{d}y_i\cdot \nabla^{k-l-1}\partial_{i}S)\\
&\quad+(k-l-1)\int_{e^s\mathbb{T}^{d}}P\phi^{l}\nabla^{k-l-1}S(-2\sum_{i=1}^{d}\nabla^{k-l-2}\p_{i}S\p_{i}\nabla\Psi-2\alpha \nabla S\nabla^{k-l-2}\Delta\Psi-\nabla^{k-l-1}S)\\
&\quad+O(E_{l,0})\\
&=E_{k-l,S,1}+E_{k-l,S,2}+O(E_{l,0}),
\end{align*}
where we have absorbed the term corresponding to $(r-1)\nabla^{k-l-1}S$ in $O(E_{l,0}).$
Similarly, 
\begin{align*}
&\quad E_{k-l,\Psi}\\
&=\int_{e^s\mathbb{T}^{d}}\nabla^{k-l}\Psi(-2\sum_{i=1}^{d}(\nabla^{k-l}\partial_{i}\Psi)\p_{i}\Psi-2\alpha S \nabla^{k-l}S-\sum_{i=1}^{k}y_i\partial_{i}\nabla^{k-l}\Psi)P\phi^{l}\\
&\quad+ (k-l)\int_{e^s\mathbb{T}^{d}}\nabla^{k-l}\Psi (-2\sum_{i=1}^{d}\nabla^{k-l-1}\partial_{i}\Psi \partial_{i}\nabla \Psi-2\alpha \nabla^{k-l-1}S\nabla S -\nabla^{k-l}\Psi)P\phi^{l}\\
&\quad+e^{(4-2r)s}\int_{e^s\mathbb{T}^{d}}\nabla^{k-l}\Psi(\nabla^{k-l}\Delta w+2\sum_{i=1}^{d}\partial_{i}w\nabla^{k-l}\p_{i}w)P\phi^{l}\\
&\quad+2(k-l)e^{(4-2r)s}\int_{e^s\mathbb{T}^{d}}\nabla^{k-l}\Psi(\sum_{i=1}^{d}\nabla^{k-l-1}\p_{i}w\p_{i}\nabla w)P\phi^{l}\\
&\quad +O(E_{l,0})\\
&=E_{k-l,\Psi,1}+E_{k-l,\Psi,2}+E_{k-l,\Psi,3}+E_{k-l,\Psi,4}+O(E_{l,0})
\end{align*}
and
\begin{align*}
E_{k-l,w}&=e^{(4-2r)s}\int_{e^s\mathbb{T}^{d}}\nabla^{k-l}w(-2\sum_{i=1}^{d}\p_{i}w\p_{i}\nabla^{k-l}\Psi-2\sum_{i=1}^{d}\p_{i}\nabla^{k-l}w\p_{i}\Psi-\Delta\nabla^{k-l}\Psi-\sum_{i=1}^{d}y_i\p_{i}\nabla^{k-l}w)P \phi^l\\
&\quad+e^{(4-2r)s}(k-l)\int\nabla^{k-l}w(-2\sum_{i=1}^{d}\partial_{i}\nabla^{k-l-1}w\partial_{i}\nabla \Psi-\nabla^{k-l}w)P\phi^{l}\\
&\quad-2(k-l)e^{(4-2r)s}\int_{e^s\mathbb{T}^{d}}\nabla^{k-l}w(\sum_{i=1}^{d}\partial_{i}\nabla w\partial_{i}\nabla^{k-l-1}\Psi)P \phi^l+O(E_{l,0})\\
&=E_{k-l,w,1}+E_{k-l,w,2}+E_{k-l,w,3}+O(E_{l,0})
\end{align*}
Since $k\ll s$, we have 
\begin{equation}\label{hdecayestimate01}
    E_{k-l,\Psi, 4}=O(E_{l,0}); \quad E_{k-l,w, 3}=O(E_{l,0}).
\end{equation}
Now we integrate by parts in $E_{k-l,S,1}$, $E_{k-l,\Psi,1}$ and $E_{k-l,w,1}$. We have
\begin{align*}
E_{k-l,S,1}=&-2\alpha \int_{e^s\mathbb{T}^{d}}S \Delta \nabla^{k-l-1}\Psi\cdot \nabla^{k-l-1}S P\phi^{l}\\
&\quad + \int_{e^s\mathbb{T}^{d}}|\nabla^{k-l-1}S|^2(P\Delta\Psi+\nabla \Psi\cdot \nabla P+\frac{1}{2}y\cdot \nabla P)\phi^{l}\\
&\quad + \int_{e^s\mathbb{T}^{d}}|\nabla^{k-l-1}S|^2(\frac{1}{2}P y\cdot\nabla(\phi^{l})+P \nabla\Psi \cdot  \nabla(\phi^{l}))\\
&\quad-\frac{1}{2}\int_{\partial e^s\mathbb{T}^{d}} y\cdot \vec{n}|\nabla^{k-l-1}S|^2\phi^{l}P \\
&\quad+O(E_{l,0}).
\end{align*}
Similarly, we have 
\begin{align*}
E_{k-l,\Psi,1}&=-2\alpha\int\nabla^{k-l}\Psi\nabla^{k-l}S SP\phi^{l}\\
&+\int_{e^s\mathbb{T}^{d}}|\nabla^{k-l}\Psi|^2(\Delta\Psi P+\nabla\Psi\cdot \nabla P+\frac{1}{2}y\cdot \nabla P)\phi^{l}\\
&+\int_{e^s\mathbb{T}^{d}}|\nabla^{k-l}\Psi|^2(\frac{1}{2}Py\cdot \nabla(\phi^{l})+P\nabla\Psi\cdot \nabla (\phi^{l}))\\
&-\frac{1}{2}\int_{\partial e^s\mathbb{T}^{d}} y\cdot \vec{n}|\nabla^{k-l}\Psi|^2\phi^{l}P \\
&+O(E_{l,0}),
\end{align*}
and 
\begin{align*}
&\quad 
E_{k-l,w,1}\\
&=e^{(4-2r)s}\int_{e^s\mathbb{T}^{d}}\nabla^{k-l}w(-2\sum_{i=1}^{d}\partial_{i}w\partial_{i}\nabla^{k-l}\Psi-\Delta \nabla^{k-l}\Psi)P\phi^{l}\\
&\quad+e^{(4-2r)s}\int_{e^s\mathbb{T}^{d}}|\nabla^{k-l}w|^{2}(\Delta\Psi P+\nabla\Psi \cdot \nabla P+\frac{1}{2}y\cdot P)\phi^{l}\\
&\quad +e^{(4-2r)s}\int_{\partial e^s\mathbb{T}^{d}}|\nabla^{k-l}w|^2(\frac{1}{2}Py\cdot \nabla (\phi^l)+P\nabla\Psi\cdot \nabla(\phi^{l}))\\
&\quad -\frac{1}{2}\int_{\partial e^s\mathbb{T}^{d}} y\cdot \vec{n}e^{(4-2r)s}|\nabla^{k-l}w|^2\phi^{l}P  \\
&\quad +O(E_{l,0}).
\end{align*}
Then 
\begin{align*}
&\quad E_{k-l,S,1}+E_{k-l,\Psi,1}+E_{k-l,w,1}+E_{k-l,P}+E_{k-l,\Psi,3}+E_{boundary}\\
&=\underbrace{-2\int_{e^s\mathbb{T}^{d}}\alpha S\Delta \nabla^{k-l-1}\Psi \nabla^{k-l-1}S P\phi^{l}-2\alpha \int_{e^s\mathbb{T}^{d}} \nabla^{k-l}\Psi \nabla^{k-l}S SP\phi^{l}}_{B_1}\\
&\underbrace{\quad+e^{(4-2r)s}\int_{e^s\mathbb{T}^{d}}\nabla^{k-l}w(-2\sum_{i=1}^{d}\p_{i}w\p_{i}\nabla^{k-l}\Psi -\Delta \nabla^{k-l}\Psi)P\phi^{l}}_{B_2}\\
&\underbrace{\quad+e^{(4-2r)s}\int_{e^s\mathbb{T}^{d}}\nabla^{k-l}\Psi(\nabla^{k-l}\Delta w+2\sum_{i=1}^{d}\partial_{i}w\nabla^{k-l}\p_{i}w)P\phi^{l}}_{B_3}\\
&\underbrace{\quad+\int_{e^s\mathbb{T}^{d}}(|\nabla^{k-l-1}S|^2+|\nabla^{k-1}\Psi|^2+|\nabla^{k-1}w|^2e^{(4-2r)s})(\frac{1}{2}Py\cdot \nabla \phi^{l}+P\nabla \Psi\cdot \nabla (\phi^{l}))}_{B_{4}}+O(E_{l,0}).
\end{align*}
Notice that $B_1$, $B_{2}$ and $B_{3}$ contain terms with derivative more than $k-l$ th order. We use integration by parts to treat $B_1$ and show the cancellation between $B_2$ and $B_3$. From integration by parts, we have
\begin{align*}
B_1&=2\alpha\sum_{i=1}^{d}\int_{e^s\mathbb{T}^{d}}\partial_{i}\nabla^{k-l-1}\Psi\partial_{i}(PS\phi^{l})\nabla^{k-l-1}S\\
&\leq 2\alpha \int_{e^s\mathbb{T}^{d}}|\nabla^{k-l-1}S||\nabla^{k-l}\Psi||\nabla(PS\phi^{l})|\\
&\leq \alpha \int_{e^s\mathbb{T}^{d}}(|\nabla^{k-l-1}S|^2+|\nabla^{k-l}\Psi|)^2|PS|\nabla(\phi^{l})|\\
&+O(E_{l,0}),
\end{align*}
when we use $\|\frac{\nabla P}{P}\|_{L^{\infty}}=\|\nabla w\|_{L^{\infty}}\lesssim 1$ and Cauchy-Schwartz in the last inequality. Moreover,
\begin{align*}
B_2&=e^{(4-2r)s}\int_{e^s\mathbb{T}^{d}}\nabla^{k-l}w\left(-2\sum_{i=1}^{d}\partial_{i}w\partial_{i}\nabla^{k-l}\Psi-\Delta\nabla^{k-l}\Psi\right)P\phi^{l}\\
&=e^{(4-2r)s}\left(\int\sum_{i=1}^{d}\nabla^{k-l}\p_{i}\Psi\nabla^{k-l}\p_{i}w P\phi^{l}\right.\\
&\quad +\left.\int_{e^s\mathbb{T}^{d}}(\sum_{i=1}^{d}\nabla^{k-l}w\p_{i}\nabla^{k-l}\Psi)(\p_{i}P\phi^{l}+P\partial_{i}(\phi^{l}))-2\int(\sum_{i=1}^{d}\p_{i}w\p_{i}\nabla^{k-l}\Psi)\nabla^{k-l}wP\phi^{l}\right).
\end{align*}
Since $2P\p_{i}w=\p_{i}P$, we have
\begin{align*}
B_{2}&=e^{(4-2r)s}\int( \sum_{i=1}^{d}\nabla^{k-l}\p_{i}\Psi\nabla^{k-l}\p_{i}w) P\phi^{l}+e^{(4-2r)s}\int_{e^s\mathbb{T}^{d}}(\sum_{i=1}^{d}\nabla^{k-l}w\p_{i}\nabla^{k-l}\Psi)P\partial_{i}(\phi^{l})\\
&=e^{(4-2r)s}\int_{e^s\mathbb{T}^{d}}(\sum_{i=1}^{d}\nabla^{k-l}\p_{i}\Psi\nabla^{k-l}\p_{i}w )P\phi^{l}+O(E_{l,0}),
\end{align*}
where we use $|E_{k-l+1,0}(s)|\leq E_{l-1,0}\ll O(E_{l,0})$ from bootstrap assumption \eqref{highordermestimate}.
Similarly, 
\begin{align*}
B_{3}=-e^{(4-2r)s}\int( \sum_{i=1}^{d}\nabla^{k-l}\p_{i}\Psi\nabla^{k-l}\p_{i}w) P\phi^{l}+O(E_{l,0}).
\end{align*}
Hence $B_2+B_3=O(E_{l,0})$.
In conclusion, we have
\begin{align}\label{hdecayestimate02}
&\quad E_{k-l,S,1}+E_{k-l,\Phi,1}+E_{k-l,w,1}+E_{k-l,P}+E_{k-1,\Phi,3}+E_{boundary}\\\nonumber
&\leq\int_{e^s\mathbb{T}^{d}}(|\nabla^{k-l-1}S|^2+|\nabla^{k-l}\Psi|^2+e^{(4-2r)s}|\nabla^{k-l}w|^2)(\frac{1}{2} P|y||\nabla (\phi^{l})|+\alpha S P|\nabla (\phi^{l})|+P|\nabla \Psi| |\nabla(\phi^{l})|)\\\nonumber
&\quad +O(E_{l,0}).
\end{align}
Now we control the remaining terms : $E_{k-l,S,2}+E_{k-l,\Psi,2}+E_{k-l,w,2}$. Those terms give the decay. 
We have 
\begin{align*}
&\quad E_{k-l,S,2}+E_{k-l,\Psi,2}+E_{k-l,w,2}\\
&=(k-l)\int_{e^s\mathbb{T}^{d}}\nabla^{k-l-1}S[-2\sum_{i=1}^{d}\nabla^{k-l-2}\partial_{i}S\partial_{i}\nabla\Psi-2\alpha \nabla S\nabla^{k-l-2}\Delta\Psi-\nabla^{k-l-1}S]\phi^{l}P\\
&\qquad +\nabla^{k-l}\Psi[-2\sum_{i=1}^{d}\nabla^{k-l-1}\partial_{i}\Psi\nabla\partial_{i}\Psi-2\alpha\nabla^{k-l-1}S \nabla S-\nabla^{k-l}\Psi]\phi^{l}P\\
&\qquad+e^{(4-2r)s}\nabla^{k-l}w[-2\sum_{i=1}^{d}\partial_{i}\nabla^{k-l-1}w\partial_{i}\nabla \Psi-\nabla^{k-l}w]\phi^{l}P+O(E_{l,0}).
\end{align*}
By using \eqref{cutoffdensity}, \eqref{perturb defi}
\begin{align*}
\p_{i}\nabla \Psi&=\p_{i}\nabla (\Psi_{p}\hat{\mathfrak{X}}_{d}(ye^{-s})+\tilde{\Psi})=\Breve{X} \p_{i}(\bar{U}_{p,R}\frac{y}{R})+O(\delta_{\rm{low}}^{\frac{1}{2}})+O(e^{-s}), \\
\nabla S&=\p_{i}(\hat{\mathfrak{X}}_{d}(ye^{-s})S_{p}+\tilde{\mathfrak{X}}_{d}(ye^{-s})e^{-(r-1)s}+\tilde{S})=\Breve{X}\partial_{R}S_{p}\frac{y}{R}+O(\delta_{\rm{low}}^{\frac{1}{2}})+O(e^{-s}),
\end{align*}
 with $\Breve{X}=\hat{\mathfrak{X}}_{d}(ye^{-s})$,
and $s\gg \frac{1}{k}\gg \delta_{\rm{low}}^{\frac{1}{2}}$, we have 
\begin{align*}
&E_{k-l,S,2}+E_{k-l,\Psi,2}+E_{k-l,w,2}\\
&=\underbrace{(k-l)\sum_{j=1}^{d}\int_{e^s\mathbb{T}^{d}}P\phi^{l}\nabla^{k-l-2}\partial_{j}S[-2\sum_{i=1}^{d}\nabla^{k-l-2}\partial_{i}S\Breve{X}\partial_{i}(\bar{U}_{p,R}\frac{y_j}{R})-\nabla^{k-l-2}\partial_{j}S]}_{B_5}\\
&\underbrace{\quad +  P\phi^{l}\nabla^{k-l-1}\partial_{j}\Psi[-2\sum_{i=1}^{d}\nabla^{k-2-l}\p_{i}\Psi\Breve{X}\partial_{i}(\bar{U}_{p,R}\frac{y_j}{R})-\nabla^{k-l-1}\partial_{j}\Psi]}_{B_5}\\
&\underbrace{\quad +e^{(4-2r)s}P\phi^{l}\nabla^{k-l-1}\partial_{j}w[-2\sum_{i=1}^{d}\nabla^{k-l-1}\partial_{i}w\partial_{i}(\bar{U}_{R}\frac{y_j}{R})-\nabla^{k-l-1}\partial_{j}w]}_{B_5}\\
&\underbrace{+(k-l)\sum_{j=1}^{d}\int_{e^s\mathbb{T}^{d}} P\phi^{l}\nabla^{k-l-1}\partial_{j}\Psi[-2\alpha \Breve{X}\partial_{R}S_{d}\frac{y_j}{R}\nabla^{k-l-1}S]}_{B_6}\\
&\underbrace{+(k-l)\sum_{j=1}^{d}\int_{e^s\mathbb{T}^{d}}P\phi^{l}\nabla^{k-l-2}\partial_{j}S[-2\alpha \Breve{X}\partial_{R}S_{d}\frac{y_j}{R}\nabla^{k-l-2}\Delta\Psi]}_{B_{7}}\\
&+O(E_{l,0})
\end{align*}
Now we treat the terms separately. Notice that from \eqref{eq:radialcomp}, we have \begin{align*}
B_{5}&=(k-l)\int_{e^s\mathbb{T}^{d}}P\phi^{l}(|\nabla^{k-l-2}\nabla_{R}S|^2+|\nabla^{k-l-1}\nabla_{R}\Psi|^2+e^{-(4-2r)s}|\nabla^{k-l-1}\nabla_{R}w|^2)(-1-2\Breve{X}\partial_{R}\bar{U}_{p,R})\\\nonumber
 &\quad + (k-l)\int_{e^s\mathbb{T}^{d}}P\phi^{l}(|\nabla^{k-l-2}\nabla_{\theta}S|^2+|\nabla^{k-l-1}\nabla_{\theta}\Psi|^2+e^{-(4-2r)s}|\nabla^{k-l-1}\nabla_{{\theta}}w|^2)(-1-2\Breve{X}\frac{\bar{U}_{p,R}}{R})
\end{align*}
Moreover, we have 
\begin{align*}
&|\sum_{j=1}^{d}\nabla^{k-l-1}\partial_{j}\Psi\partial_{R}S_{d}\frac{y_j}{R}\nabla^{k-l-1}S|\\
=&|\nabla^{k-l-1}\partial_{R}\Psi\partial_{R}S_{d}\nabla^{k-l-1}S|\\
\leq& |\nabla^{k-l-1}\partial_{R}\Psi||\partial_{R}S_{d}||\nabla^{k-l-1}S|.
 \end{align*}
 Then \begin{align*}
 B_{6}\leq  (k-l)\int_{e^s\mathbb{T}^{d}}P\phi^{l}(|\nabla^{k-l-2}\nabla_{R}S|^2+|\nabla^{k-l-2}\nabla_{\theta}S|^2+|\nabla^{k-l-1}\nabla_{R}\Psi|^2)(2\alpha \Breve{X}|\partial_{R}S_{d}|)\\\nonumber
 \end{align*}

For $B_7$, we use integration by parts and have
\begin{align*}
&\quad \sum_{j=1}^{d}\int_{e^s\mathbb{T}^{d}}P\phi^{l}\nabla^{k-l-2}\partial_{j}S[-2\alpha \Breve{X}\partial_{R}S_{d}\frac{y_j}{R}\nabla^{k-l-2}\Delta\Psi]\\
&=\sum_{i=1}^{d}\sum_{j=1}^{d}\int_{e^s\mathbb{T}^{d}}\partial_{j}( P\phi^{l} 2\alpha \Breve{X}\partial_{R}S_{d}\frac{y_j}{R})\nabla^{k-l-2}S\nabla^{k-l-2}\p_{i}^{2}\Psi\\
&\quad +\sum_{i=1}^{d}\sum_{j=1}^{d}\int_{e^s\mathbb{T}^{d}}P\phi^{l}(2\alpha \Breve{X}\partial_{R}S_{d}\frac{y_j}{R})\nabla^{k-l-2}S\nabla^{k-l-2}\p_{i}^{2}\p_{j}\Psi\\
&=\sum_{i=1}^{d}\sum_{j=1}^{d}\int_{e^s\mathbb{T}^{d}}\partial_{j} (P\phi^{l} 2\alpha \Breve{X}\partial_{R}S_{d}\frac{y_j}{R})\nabla^{k-l-2}S\nabla^{k-l-2}\p_{i}^{2}\Psi\\
&\quad-\sum_{i=1}^{d}\sum_{j=1}^{d}\int_{e^s\mathbb{T}^{d}}\p_{i} (P\phi^{l}2\alpha \Breve{X}\partial_{R}S_{d}\frac{y_j}{R})\nabla^{k-l-2}S\nabla^{k-l-2}\p_{i}\p_{j}\Psi\\
&\quad-\sum_{i=1}^{d}\sum_{j=1}^{d}\int_{e^s\mathbb{T}^{d}}P\phi^{l}(2\alpha \Breve{X}\partial_{R}S_{d}\frac{y_j}{R})\nabla^{k-l-2}\partial_{i}S\nabla^{k-l-2}\p_{i}\p_{j}\Psi\\
&=O(E_{l,0})-\sum_{i=1}^{d}\sum_{j=1}^{d}\int_{e^s\mathbb{T}^{d}}P\phi^{l}(2\alpha \Breve{X}\partial_{R}S_{d}\frac{y_j}{R})\nabla^{k-l-2}\partial_{i}S\nabla^{k-l-2}\p_{i}\p_{j}\Psi\\\
&\leq O(E_{l,0})+\int_{e^s\mathbb{T}^{d}} P\phi^{l} |(2\alpha \Breve{X}\partial_{R}S_{d})|(\frac{|\nabla^{k-l-1}S|^2}{2}+\frac{|\nabla^{k-l-1}\partial_{R}\Psi|^2}{2}).
\end{align*}
Here in the last step we use $\left|\frac{\partial_{R}S_{d}}{R}\right|\lesssim 1$ \eqref{eq:regularityatorigin}, $\left|\frac{\nabla P}{P}\right|\lesssim |\nabla w|\lesssim 1$ , interpolation lemma \ref{interpolation01} and Cauchy-Schwartz inequality.
 Therefore 
 \begin{align}\label{hdecayestimate03}
 &\quad E_{k-l,S,2}+E_{k-l,\Psi,2}+E_{k-l,w,2}\\\nonumber
 &\leq (k-l)\int_{e^s\mathbb{T}^{d}}P\phi^{l}(|\nabla^{k-l-2}\nabla_{R}S|^2+|\nabla^{k-l-1}\nabla_{R}\Psi|^2+e^{-(4-2r)s}|\nabla^{k-l-1}\nabla_{R}w|^2)(-1-2\Breve{X}\partial_{R}\bar{U}_{p,R}+2\alpha \Breve{X}|\partial_{R}S_{d}|)\\\nonumber
 &\quad + (k-l)\int_{e^s\mathbb{T}^{d}}P\phi^{l}(|\nabla^{k-l-2}\nabla_{\theta}S|^2+|\nabla^{k-l-1}\nabla_{\theta}\Psi|^2+e^{-(4-2r)s}|\nabla^{k-l-1}\nabla_{{\theta}}w|^2)(-1-2\Breve{X}\frac{\bar{U}_{p,R}}{R}+2\alpha \Breve{X}|\partial_{R}S_{d}|)+O(E_{l,0}).
 \end{align}
 Combining \eqref{hdecayestimate01}, \eqref{hdecayestimate02} and \eqref{hdecayestimate03}, we have 
 \begin{align*}
 \partial_{s}E_{k-l}(s)\leq -\eta k E_{k-l}(s)+O(E_{l,0}),
 \end{align*}
 as long as  we have
 \begin{align*}
 \inf_{y}(1+2\Breve{X}\partial_{R}\bar{U}_{p,R}-2\alpha \Breve{X}|\partial_{R}S_{d}|-\frac{1}{2}|y|\frac{|\nabla (\phi^{l})|}{(k-l)\phi^{l}}-\alpha S\frac{|\nabla (\phi^{l})|}{(k-l)\phi^{l}}-|\nabla \Psi| \frac{|\nabla (\phi^{l})|}{(k-l)\phi^{l}})>0,
 \end{align*}
 and
  \begin{align*}
 \inf_{y}(1+2\Breve{X}\frac{\bar{U}_{p,R}}{R}-2\alpha \Breve{X}|\partial_{R}S_{d}|-\frac{1}{2}|y|\frac{|\nabla (\phi^{l})|}{(k-l)\phi^{l}}-\alpha S \frac{|\nabla (\phi^{l})|}{(k-l)\phi^{l}}-|\nabla \Psi| \frac{|\nabla (\phi^{l})|}{(k-l)\phi^{l}})>0,
 \end{align*}
 which hold by \eqref{eq:radial_repulsivity}, \eqref{eq:angular_repulsivity}, the decay property \eqref{eq:profiles_decay}, $\frac{l}{k-l}\leq \frac{1}{5}$ and $\frac{|\nabla \phi|}{\phi}\leq 2$ from \eqref{weightdefi02}.
\end{proof}
\subsection{Topological argument section for the unstable modes}
Now we are let to show $P_{uns}(\Psi_{t}, S_{t}) \leq \delta^{\frac{3}{2}}e^{-\epsilon(s-s_0)}.$ This comes from a standard topological argument. From the same topological argument in \cite[Section 3.5]{CaoLabora-GomezSerrano-Shi-Staffilani:nonradial-implosion-compressible-euler-ns-T3R3}, we have the following proposition.
\begin{prop}\label{prop:topo}
There exists specific initial conditions $\{a_{i}\}$ in \eqref{initial data}, such that
\[
\|P_{uns}(\Psi_{t}(\cdot,s),S_{t}(\cdot,s))\|_{X} \leq \delta_{0}^{\frac{3}{2}}e^{-\frac{4}{3}\epsilon(s-s_0)}.
\]
\end{prop}
Combining Propositions \ref{prop:bootstrap} and \ref{prop:topo}, we have Theorem \ref{th:periodic}.

\section{Appendix}

\begin{lemma} \label{lemma:GN_general01} Let $f: e^s\mathbb{T}^d \to \mathbb R$ be  a periodic  function on $H^{m}(e^s\mathbb{T}^d)$ with $L\geq 1$ . Let $I_{-1} = [0, 1]$ and $I_j = [2^{j}, 2^{j+1}]$. Let $\phi (x), \psi(x) \geq 0$ be weights such that there exist values $\phi_j, \psi_j$ and a constant $C$ satisfying that $\phi (x) \in [\phi_j/C, C \phi_j]$ and $\psi(x) \in [\psi_j/C, C\psi_j]$ for every $x$ with $|x| \in I_j$. Moreover, let us assume the weights are approximately $1$ at the origin, that is $\phi_{-1} = \psi_{-1} = \phi_0 = \psi_0 = 1$.  Let $1 \leq i \leq m$. Assume that the parameters $p, q, \bar{r}, \theta$ satisfy:
\begin{equation} \label{eq:GN_conditions}
    \frac{1}{\bar{r}} = \frac{i}{d} + \theta \left( \frac{1}{q} - \frac{m}{d} \right) + \frac{1-\theta}{p}, \qquad \mbox{ and } \qquad \frac{i}{m} \leq \theta < 1,
\end{equation}
Letting $\bar{r} = \infty$, we have

For $\bar{r} \in [1,\infty)$, any $\eps > 0$, and under the extra assumption:
\begin{equation}\label{eq:GN_extracond}
    \left( \frac{\phi(x)}{\langle x \rangle \psi(x) }\right)^{m\theta} \langle x \rangle^{3\theta (1/q-1/p)} \lesssim \langle x\rangle ^{-\epsilon}
\end{equation}

we have the weighted Gagliardo-Nirenberg inequalities:
\begin{equation} \label{eq:GNresult}
    \| \psi^{m(1-\theta)} \phi^{m\theta} \nabla^i f \|_{L^{\bar{r}}} \les \| \psi^m f \|_{L^p}^{1-\theta} \| \phi^m \nabla^m f \|_{L^q}^{\theta} + \| \psi^m f \|_{L^p}
\end{equation}
The implicit constants \eqref{eq:GNresult} may depend on the parameters $p, q, i, m, \theta, C$, (as well as $r$, $\eps >0$, $l$ for \eqref{eq:GNresult}) but they are independent of $f$ and $\psi$, $\phi$ and $e^s$. 
\end{lemma}
\begin{proof} We divide $\mathbb R^d$ dyadically, defining $A_{-1} = B(0, 1)$ and $A_j = B(0, 2^{j+1}) \setminus B(0, 2^{j})$ for $j \geq 0$. Let $D_{k}=\cup_{j=-1}^{k}A_{j}$.  We first show
\begin{equation} \label{eq:GNresult02}
    \| \psi^{m(1-\theta)} \phi^{m\theta} \nabla^i f \|_{L^{\bar{r}}(D_{k})} \les \| \psi^m f \|_{L^p(D_{k})}^{1-\theta} \| \phi^m \nabla^m f \|_{L^q(D_{k})}^{\theta} + \| \psi^m f \|_{L^p(D_{k})},
\end{equation}
with implicit constant independent of $k$. From Gagliardo-Nirenberg inequality for bounded domains,
we have 
\begin{equation} \label{eq:Aj}
    \| \mathbbm{1}_{A_j} \nabla^i f \|_{L^{\bar{r}}} \les  \| \mathbbm{1}_{A_j} \nabla^m f \|_{L^q}^{\theta}  \| \mathbbm{1}_{A_j}f \|_{L^p}^{1-\theta} + 2^{3j(-i/3+1/\bar{r}-1/p)} \| \mathbbm{1}_{A_j}  f \|_{L^p},
\end{equation}
for $j \geq 0$.

Now, introducing the weights, we see that
\begin{align}
    \| \mathbbm{1}_{A_j} \psi^{m(1-\theta)} \phi^{m\theta} \nabla^i f \|_{L^{\bar{r}}} &\les  \psi_j^{m(1-\theta)} \phi_j^{m\theta} \| \mathbbm{1}_{A_j} \nabla^m f \|_{L^q}^{\theta}  \| \mathbbm{1}_{A_j}  f \|_{L^p}^{1-\theta} + 2^{3j (-i/3+1/\bar{r}-1/p)} \psi_j^{-m\theta} \phi_j^{m\theta} \| \psi^m \mathbbm{1}_{A_j}  f \|_{L^p} \notag \\
    &\les  
    \| \mathbbm{1}_{A_j} \phi^m \nabla^m f \|_{L^q}^{\theta}  \| \psi^m \mathbbm{1}_{A_j}  f \|_{L^p}^{1-\theta} + 2^{j (-i+3/\bar{r}-3/p)} \left( \frac{\phi_j}{\psi_j} \right)^{m\theta} \| \psi^m \mathbbm{1}_{A_j} f \|_{L^p}  \label{eq:marat}
\end{align}
where we recall that our weights are within $[\phi_j/C, \phi_j C]$ and $[\psi_j/C, \psi_j C]$, respectively, for $x \in A_j$. Recall also that $\les$ is allowed to depend on $C, i, m, p, q, \bar{r}, \eps$ (in particular, it can absorb constants $C^m$).

Now, note from \eqref{eq:GN_conditions}
\begin{equation} \label{eq:robespierre}
    -i + \frac{3}{\bar{r}} - \frac{3}{p} = - m \theta + \theta \left(\frac{3}{q} - \frac{3}{p} \right).
\end{equation}
Using that together with the fact that $\phi, \psi$ are within a constant factor of $\phi_j, \psi_j$ for $x \in A_j$, from \eqref{eq:GN_extracond}, we obtain that 
\begin{equation*}
\sum_{j=0}^{\infty }2^{j (-i+3/\bar{r}-3/p)} \left( \frac{\phi_j}{\psi_j} \right)^{m\theta} = \sum_{j=0}^{\infty }\left( 2^j \right)^{\theta (3/q - 3/p)} \left( \frac{ \phi_j}{\langle x \rangle\psi_j} \right)^{m\theta} \les 
\sum_{j=0}^{\infty }\langle 2^{j} \rangle^{\theta (3/q - 3/p)} \left( \frac{ \phi (2^{j})}{2^{j} \psi (2^{j})} \right)^{m\theta} \les 1.
\end{equation*}
Then we have 
\begin{align}\label{interplationholder}
&\quad \| \psi^{m(1-\theta)} \phi^{m\theta}\nabla^{i}(f(x))\|_{L^{\bar{r}}(D_{k})}^{\bar{r}}\\\nonumber
&=\sum_{j=-1}^{k}(\| \psi^{m(1-\theta)} \phi^{m\theta}\nabla^{i}(f(x))\|_{L^{\bar{r}}(A_{j})})^{\bar{r}}\\\nonumber
&\les \sum_{j=-1}^{k} \|  \phi^{m}\nabla^{m}(f(x))\|_{L^{q}(A_{j})}^{\bar{r}\theta}\|\psi^{m} f( x)\|_{L^{p}(A_{j})}^{\bar{r}(1-\theta)}+\sum_{j=0}^{\infty} 2^{j (-i+3/\bar{r}-3/p)}  (\frac{\phi_j}{\psi_j})^{m\theta}\|\psi^{m}f(x)\|_{L^{p}(A_{j})}^{\bar{r}}\\\nonumber
&\les  \sum_{j=-1}^{k} \|\phi^{m}\nabla^{m}(f(x))\|_{L^{q}(A_{j})}^{\bar{r}\theta}\|\psi^{m}f( x)\|_{L^{p}(A_{j})}^{\bar{r}(1-\theta)}+\sup_{j} \|\psi^{m}f(x)\|_{L^{p}(A_{j})}^{\bar{r}}\\\nonumber
&\les  \sum_{j=-1}^{k} (\|\phi^{m}\nabla^{m}(f(x))\|_{L^{q}(A_{j})}^{q \frac{\bar{r}\theta}{q}}\|\psi^{m}f( x)\|_{L^{p}(A_{j})}^{p \frac{\bar{r}(1-\theta)}{p}})+\sup_{j}\|\psi^{m}f(x)\|_{L^{p}(A_j)}^{\bar{r}}.
\end{align}

We claim that for $l_1,l_2>0, l_1+l_2\geq 1$, $a_j,b_j\geq 0$,  
\begin{align}\label{holderine}
\sum_{j}a_j^{l_1}b_j^{l_2}\leq (\sum_{j}a_j)^{l_1}(\sum_{j}b_j)^{l_2}.
\end{align}

Since $\frac{\bar{r}\theta}{q}+\frac{\bar{r}(1-\theta)}{p}=1-\frac{i\bar{r}}{3}+\theta\frac{m\bar{r}}{3}\geq 1$, combining \eqref{holderine} and \eqref{interplationholder}, we then have \eqref{eq:GNresult02}:
\begin{align*}
&\quad \|\psi^{m(1-\theta)} \phi^{m\theta}\nabla^{i}(f(x))\|_{L^{\bar{r}}(D_{k})}^{\bar{r}}\\
&\les  (\sum_{j=0}^{k} \|\phi^{m}\nabla^{m}(f(x))\|_{L^{q}(A_{j})}^{q})^{\frac{\bar{r}\theta}{q}}(\sum_{j=1}^{k}\|\psi^{m}f( x)\|_{L^{p}(A_{j})}^{p})^{\frac{\bar{r}(1-\theta)}{p}}+\|\psi^{m}f(x)\|_{L^{p}(D_{k})}^{\bar{r}},\\
&\les\|\phi^{m}\nabla^{m}(f(x))\|_{L^{q}(D_{k})}^{\bar{r}\theta}\|\psi^{m}f(x)\|_{L^{p}(D_{k})}^{\bar{r}(1-\theta)}+\|\psi^{m}f(x)\|_{L^{p}(D_{k})}^{\bar{r}}.
\end{align*} 
For $e^{s}\mathbb{T}_{L}^3$, it is clear that there exists $k$, such that $D_{k-1}\subset \mathbb{T}_{L}^3 \subset D_{k}.$ 
Since $f$ is an periodic function on $e^s\mathbb{T}_{L}^3$, there is an natural periodic extension $E(f)$ to $D_k$ and satisfying
\[
\|\nabla^{i}f\|_{L^{l}(e^s\mathbb{T}^d)} \leq \|\nabla^{i}E(f)\|_{L^{l}(D_k)}\leq d^d\|\nabla^{i}f\|_{L^{l}(e^s\mathbb{T}^d)}.
\]
for all $i,l$.
Then \eqref{eq:GNresult} follows from \eqref{eq:GNresult02}.
\end{proof}
\begin{lemma} \label{lemma:GN_general02} Let $f:\mathbb R^d \to \mathbb R$. Let $I_{-1} = [0, 1]$ and $I_j = [2^{j}, 2^{j+1}]$. Let $\phi (x), \psi(x) \geq 0$ be weights such that there exist values $\phi_j, \psi_j$ and a constant $C$ satisfying that $\phi (x) \in [\phi_j/C, C \phi_j]$ and $\psi(x) \in [\psi_j/C, C\psi_j]$ for every $x$ with $|x| \in I_j$. Moreover, let us assume the weights are approximately $1$ at the origin, that is $\phi_{-1} = \psi_{-1} = \phi_0 = \psi_0 = 1$.  Let $1 \leq i \leq m$. Assume that the parameters $p, q, \bar{r}, \theta$ satisfy:
\begin{equation} 
    \frac{1}{\bar{r}} = \frac{i}{d} + \theta \left( \frac{1}{q} - \frac{m}{d} \right) + \frac{1-\theta}{p}, \qquad \mbox{ and } \qquad \frac{i}{m} \leq \theta < 1,
\end{equation}
Letting $\bar{r} = \infty$, we have
\begin{equation} \label{eq:GNresultinfty}
    \left| \nabla^i f \right| \les \| \psi^m f \|_{L^p}^{1-\theta} \| \phi^m \nabla^m f \|_{L^q}^{\theta} \psi^{-m(1-\theta)} \phi^{-m\theta}  + \left\| \psi^m f \right\|_{L^p} \cdot  \langle x \rangle^{3\theta (1/q-1/p)-m\theta} \psi^{-m}.
\end{equation}
If moreover $p = \infty, q = 2$ and $\psi = 1$, we obtain
\begin{equation} \label{eq:GNresultinfty_simplified}
    \left| \nabla^i f \right| \les \| f \|_{L^\infty}^{1-\frac{i}{m-3/2}} \| \phi^m \nabla^m f \|_{L^q}^{\frac{i}{m-3/2}}  \phi^{-m\frac{i}{m-3/2}}  + \left\| f \right\|_{L^\infty} \cdot  \langle x \rangle^{-i}.
\end{equation}
For $\bar{r} \in [1,\infty)$, any $\eps > 0$, and under the extra assumption:
\begin{equation}
    \left( \frac{\phi(x)}{\langle x \rangle \psi(x) }\right)^{m\theta} \langle x \rangle^{3\theta (1/q-1/p)} \lesssim \langle x\rangle ^{-\epsilon}
\end{equation}

we have the weighted Gagliardo-Nirenberg inequalities:
\begin{equation} \label{eq:GNresult2}
    \| \psi^{m(1-\theta)} \phi^{m\theta} \nabla^i f \|_{L^{\bar{r}}} \les \| \psi^m f \|_{L^p}^{1-\theta} \| \phi^m \nabla^m f \|_{L^q}^{\theta} + \| \psi^m f \|_{L^p}
\end{equation}
The implicit constants in \eqref{eq:GNresultinfty} and \eqref{eq:GNresult2} may depend on the parameters $p, q, i, m, \theta, C$, (as well as $r$, $\eps >0$, $l$ for \eqref{eq:GNresult2}) but they are independent of $f$ and $\psi$, $\phi$. 
\end{lemma}
\begin{proof}
The proof is similar as in lemma \ref{lemma:GN_general01}, by taking $k=\infty$ in \eqref{eq:GNresult02}.
\end{proof}

\begin{lemma}[Angular repulsivity] \label{lemma:angular_repulsivity} 
We have that the profiles $(\bar U, \bar S)$ from \cite{Merle-Raphael-Rodnianski-Szeftel:implosion-i}, in the case of $(d, p) = (8, 3)$ (which corresponds to $(d, \gamma) = (8, 2)$) satisfy:
\begin{equation} \label{eq:repuls}
1 + \frac{\bar U_R}{R} - \alpha | \p_R \bar S | > \eta,
\end{equation}
for some $\eta > 0$ sufficiently small that is allowed to depend on $r$ or $\gamma$.
\end{lemma}

\begin{remark} In order to facilitate the translation of formulas between this paper and the previous literature \cite{Merle-Raphael-Rodnianski-Szeftel:implosion-i, Buckmaster-CaoLabora-GomezSerrano:implosion-compressible, CaoLabora-GomezSerrano-Shi-Staffilani:nonradial-implosion-compressible-euler-ns-T3R3}, we will prove this Lemma with the scaling convention of those papers. That is, here the profiles $\bar U_R, \bar S$ satisfy the self-similar equations:
\begin{align} \label{eq:newSS}
0   &= (r-1) \bar U_R + R\p_R \bar U_R + \bar U_R \p_R \bar U_R + \alpha \bar S^2 \\ 
0  & = (r-1)  \bar S + R \p_R  \bar S
+ \bar U_R  \p_R \bar S 
+ \alpha  \bar S \text{div} (\bar U_R \cdot \hat{R})
\end{align}
which is a rescaled version of \eqref{eq:SSprofiles} in $(U, S)$ variables. In order to obtain profiles solving \eqref{eq:SSprofiles}, one can just divide the profiles solving \eqref{eq:newSS} by $2$.
\end{remark}

\begin{proof}
The proof follows very closely the one from \cite{CaoLabora-GomezSerrano-Shi-Staffilani:nonradial-implosion-compressible-euler-ns-T3R3}. The only difference consists on the fact that the expressions are different, since $d=3$ in that paper and here we are taking $(d, \gamma) = (8, 2).$ Moreover, the expressions can be simplified further since $\gamma$ here is fixed, rather than a free parameter. Before showing the Lemma, we specify the formulas from \cite{Merle-Raphael-Rodnianski-Szeftel:implosion-i} in this case. 

We will adopt the notation of \cite{Buckmaster-CaoLabora-GomezSerrano:implosion-compressible} (with the exception that $d \neq 3$ here) and denote $U = \frac{\bar U}{R}$, $S = \frac{\bar S}{R}$. We also use $W, Z$ coordinates, which are given by $W = U + S$ and $Z = U-S$. We take $\xi = \log(R)$. Now, the ODE that the self-similar profiles satisfy becomes autonomous:
\begin{equation*}
\p_\xi W (\xi) = \frac{N_W (W, Z)}{D_W (W, Z)}, \qquad \mbox{ and } \quad \p_\xi Z = \frac{N_Z (W, Z)}{D_Z (W, Z)}.
\end{equation*}
The polynomials are given by
\begin{equation}
D_W(W, Z) = 1 + \frac{3}{4}W + \frac{1}{4}Z, \qquad \mbox{ and }\qquad D_Z(W,Z) = 1 + \frac{1}{4} W + \frac{3}{4} Z,
\end{equation}
and
\begin{align*}
N_W(W, Z) &= -r W-\frac{13 W^2}{8}-\frac{W Z}{4}+\frac{7 Z^2}{8}, \\
N_Z(W, Z) &= -r Z+\frac{7 W^2}{8}-\frac{W Z}{4}-\frac{13 Z^2}{8}.
\end{align*}

The profiles pass through the point $P_s$, which is one of the two solutions to $N_Z(W, Z) = D_Z(W, Z) = 0$ (the rightmost one). The point $\bar P_s$ is the other solution. By rescaling, one can assume the solution passes through $P_s$ at $R=1$. We have:
\begin{align*}
P_s = (W_0, Z_0) &= \left( \frac{1}{14} (-3 r+3
   \mathcal{R}_1+10),\frac{1}{14}
   (r-\mathcal{R}_1-22)\right) \,, \\ 
\bar P_s &= \left( \frac{1}{14} (-3 r-3
   \mathcal{R}_1+10),\frac{1}{14}
   (r+\mathcal{R}_1-22)\right) \,,
\end{align*}
where
\begin{equation} \label{eq:R1}
\mathcal{R}_1 = \sqrt{(r-44) r+92} \,.
\end{equation}

Now the ODE gives us the formula for $W_1$, the first-order Taylor coefficient of the solution at $P_s$. The equation $\p_\xi Z = \frac{N_Z}{D_Z}$ is singular of the type $0/0$ at $P_s$, but applying L'H\^opital, we see that $Z_1 \nabla D_Z (P_s) \cdot (W_1, Z_1) = \nabla N_Z(P_s) \cdot (W_1, Z_1)$. The resulting second-degree equation yields two possible solutions for $Z_1$. The smooth profiles require one specific choice for $Z_1$ (see the aforementioned references for details). One obtains the formulas:
\begin{align*}
W_1 &= \frac{20 (r-1)}{-r+\mathcal{R}_1+8}-\frac{2}{7} (2 r+5) \,,\\
Z_1 &=  \frac{1}{147} \left(\frac{980 r+\sqrt{2}
   \mathcal{R}_2-980}{r-\mathcal{R}_1-8}+7 (94-17 r)\right)\,,
\end{align*}
where
\begin{align*}
\mathcal{R}_2 &=7 \sqrt{7} \sqrt{r (r (r (79 r-79 \mathcal{R}_1-2906)+2
   (584 \mathcal{R}_1+6733))-24 (107
   \mathcal{R}_1+1062))+2704 \mathcal{R}_1+23424}\,.
\end{align*}

Another point of interest of the phase portrait is the solution to $N_Z = 0$ and $N_W = 0$. There are four solutions to $N_Z = N_W = 0$. Two of them ($(0, 0)$ and $(-r, -r)$), and of the other two only one of them lies in the region $W>Z$. The only solution to $N_Z = N_W = 0$ with $W > Z$ is given by $P_\star$:
\begin{equation*}
P_\star = \left(  \frac{1}{5} \left(2 \sqrt{2}-1\right)
   r,  -\frac{1}{5} \left(1+2 \sqrt{2}\right)
   r\right) \,.
\end{equation*}

Moreover, the proof of the existence of the profiles from \cite{Merle-Raphael-Rodnianski-Szeftel:implosion-i} ensures the properties:
\begin{equation} \label{eq:DWDZ}
W>Z, D_W > 0, D_Z < 0, \quad \forall \xi < 0, \qquad \mbox{ and } \qquad W>Z, D_W > 0, D_Z > 0 \quad \forall \xi > 0.
\end{equation} 
These regions can be referred to as 'left of the phase portrait' (region where $D_W > 0, D_Z < 0$, corresponding to $ R<1$) and 'right of the phase portrait' ($D_W, D_Z > 0$, corresponding to $R>1$).
We can see a plot of the phase portrait in Figure \ref{fig:setup}.

We recall from \cite[Equation (1.6)]{Merle-Raphael-Rodnianski-Szeftel:implosion-nls} that the parameter $r$ lies in $1<r<r^\ast$, where in our $(d, \gamma ) = (8, 2)$ case yields
\begin{equation} \label{eq:defr}
r^\ast(\gamma) = \frac{10}{2+2 \sqrt{2}} = 2.07107\ldots 
\end{equation}

\begin{figure}[h]
\centering
\includegraphics[width=0.5\textwidth]{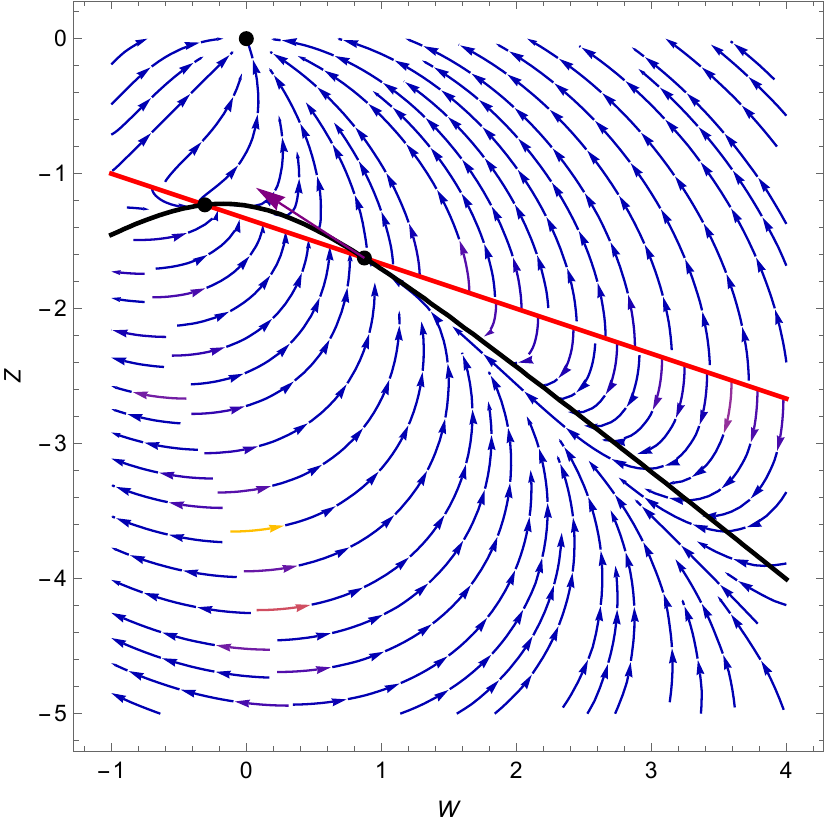}
\caption{Phase portrait in $(W, Z)$ coordinates for $r = 2.01$ (with $d = 8, p = 3$). The red line represents $D_Z = 0$ and the black curve represents $N_Z = 0$. The purple arrow corresponds to the vector in the direction of $(W_1, Z_1)$. The point $P_s$ is the rightmost intersection of $N_Z = 0$ (black curve) with $D_Z = 0$ (red line).}
\label{fig:setup}
\end{figure}

\textbf{Part I. $R<1$: Right of the phase portrait.}

We start showing Lemma \ref{lemma:angular_repulsivity} for $\xi < 0$. The proof is analogous to \cite[Lemma 4.1]{CaoLabora-GomezSerrano-Shi-Staffilani:nonradial-implosion-compressible-euler-ns-T3R3}, but with different formulas. We refer to that paper for more details. First, changing to $W, Z$ variables and multiplying by $D_W D_Z$, we are left to show
\begin{equation} \label{eq:part1}
N_W D_Z + N_Z D_W > 0.
\end{equation}

The curve solving $N_W D_Z + N_Z D_W = 0$ inside the region $D_W > 0, D_Z < 0$ can be solved in $(U, S)$ variables more easily as
\begin{equation*}
b_S = 2 \sqrt{ b_U+1} \sqrt{r+ b_U} \sqrt{ \frac{ b_U }{ 8  b_U+2(r-1)} }, \quad \mbox{ and } \quad b_U \in \left(  U(P_s),  \frac{-(r-1)}{4} \right).
\end{equation*}
We have that $\lim_{r \to r^\star} U(P_s) = 1-\sqrt{2}$ and $U(P_0) = \frac{-(r-1)}{4}$ (with limit $\frac{6-5\sqrt{2}}{4} = -0.267\ldots$ as $r\to r^\ast$), therefore $-1 < U(P_s) <  U(P_0) < 0$ for $r$ sufficiently close to $r^\ast$. We change again $(b_U, b_S)$ to $W, Z$ coordinates and denote such curve by $b$. See Figure \ref{fig:partI} for a picture of $b(t)$.

The approach is to show \eqref{eq:part1} close to $P_s$ and $P_0$. Then, if \eqref{eq:part1} was going to fail, it would need to cross $b$ in both directions (one to go into $N_W D_Z + N_Z D_W < 0$ and another to come back). We will also show that the vector field $(N_W D_Z, N_Z D_W)$ along $b$ points always to the same side of $b$. Then, one concludes that \eqref{eq:part1} always holds.

\begin{figure}[h]
\centering
\includegraphics[width=0.5\textwidth]{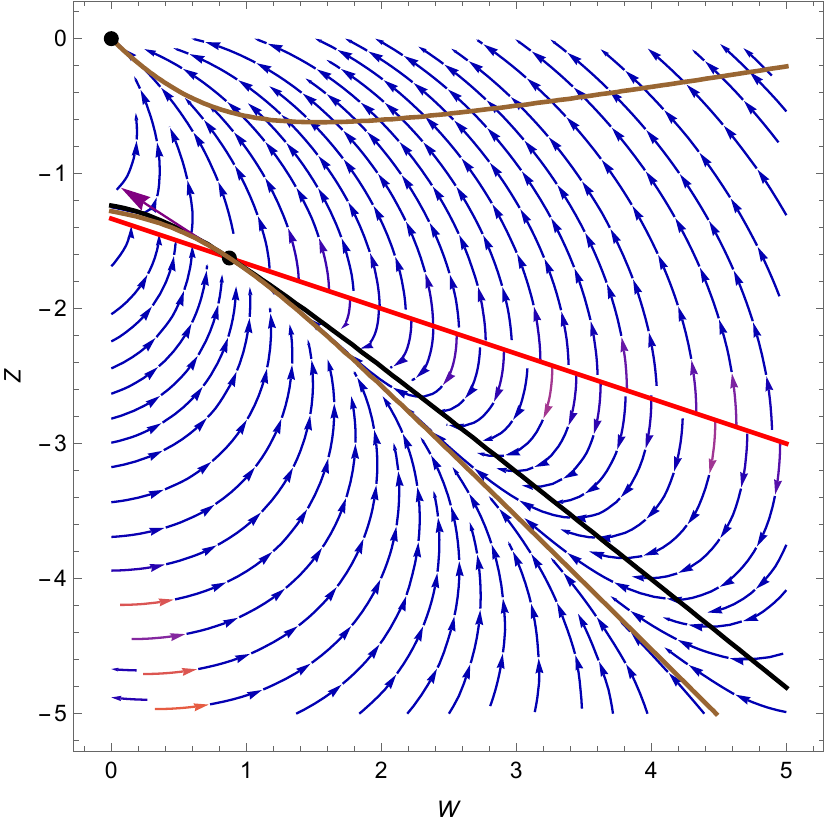}
\caption{Phase portrait in $(W, Z)$ coordinates for $r = 2.01$ (with $d = 8, p = 3$). The brown curve represents $N_W D_Z + N_Z D_W = 0$ and the lower branch, to the right of $P_s$, would correspond to the curve $b(t)$ for $t \geq 0$. See Figure \ref{fig:setup} for the meaning of the rest of the colors in the plot.}
\label{fig:partI}
\end{figure}

\, \\

\textbf{Part I.1. Constant sign along $b$} \\
 
 Since $N_W D_Z = -N_Z D_W$ along $b$ and $\nabla ( N_W D_Z + N_Z D_W )$ is a normal vector to $b$, it suffices to show that
\begin{equation*}
(-1, 1) \cdot \nabla ( N_W D_Z + N_Z D_W ) = -(W - Z) (-1 + r + 2 W + 2 Z),
\end{equation*}
has a constant sign. We see that $W-Z>0$ from \eqref{eq:DWDZ}. Regarding the second factor note $\frac{W+Z}{2} < U(P_0) = \frac{-(r-1)}{4}$, so we have that $-1+r+2(W+Z) < r-1-(r-1) = 0$, so it also has a constant sign.

\, \\

\textbf{Part I.2 Taylor analysis at $P_s$ and $P_0$.} \\

Regarding checking \eqref{eq:part1} at $P_s$, we need to show $W_1 + Z_1 < 0$. The proof of $W_1 + Z_1 < 0$ is deferred to Lemma \ref{lemma:auxiliary}. Regarding checking \eqref{eq:part1} at $P_0$, we have that the solution near $R=0$ has the form:
\begin{equation*}
W = \frac{w_0}{R} + \sum_{i=1}^\infty w_i R^{i-1}, \qquad \mbox { and } \qquad Z = \frac{-w_0}{R} + \sum_{i=1}^\infty w_i  (-R)^{i-1},  
\end{equation*}
where $w_1 = \frac{-(r-1)}{4}$ and $w_2, w_3, \ldots$ can be expressed in terms of $w_0$ (changing $w_0$ corresponds to a dilation of the profile, so it leaves $w_1$ unchanged but modifies the rest of the coefficients). Plugging this into expression \eqref{eq:part1} and using $w_1 = \frac{-(r-1)}{4}$:
\begin{equation}
\frac{N_W}{D_W} + \frac{N_Z}{D_Z} = \left(\frac{(r-5) (r-1) (3 r+1) }{8 w_0^2} - 16 w_3 \right) R^2 + O(R^4).
\end{equation}

The equation $\p_\xi W = \frac{N_W}{D_W}$ allows us to solve for $w_3$ in terms of the previous coefficients, getting
\begin{equation}
w_3 = \frac{(r-5) (r-1) (3 r+1)}{160 w_0^2},
\end{equation}
and therefore
\begin{equation}
\frac{N_W}{D_W} + \frac{N_Z}{D_Z} = \frac{(r-5) (r-1) (3 r+1)}{40 w_0^2} R^2 + O(R^4).
\end{equation}
which is smaller than $0$ since $1<r<r^\ast = \frac{10}{2+2 \sqrt{2}} < 5$. Thus, $N_W/D_W + N_Z/D_Z$ is negative for $R>0$ sufficiently small.

\, \\
\textbf{Part II: $R>1$. Left of the phase portrait}\\

Before delving into the proof, we will show a few properties regarding the behaviour of the solution for $R>1$ that will be used later in the proof of \eqref{eq:repuls}.

\textbf{Part II.1 }
First of all, let us argue that the solution stays inside the triangle determined by $D_Z > 0$, $W>Z$ and $W < W_0$. The two first inequalities come from \eqref{eq:DWDZ}. With respect to the last one, we perform a barrier argument. In order to show that $W < W_0$, it suffices to show that $(N_W/D_W, N_Z/D_Z) \cdot (-1, 0) > 0$. Since $D_W > 0$ (due to $W>Z$ and $D_Z > 0$), this just corresponds to showing that $N_W < 0$ in the vertical segment from $P_s$ to $(W_0, W_0)$. We have that
\begin{align*}
N_W(W_0, Z_0 + t) &= \underbrace{ -\frac{2}{49} \left(-2 r^2+2r \sqrt{r^2-44 r+92} +5 \sqrt{r^2-44
   r+92}-59 r+110\right) }_{a_0} \\
   &\quad + \underbrace{ \frac{1}{28} \left(-5 \sqrt{r^2-44
   r+92}+5 r-82\right) t+\frac{7 t^2}{8} }_{a_1 t + a_2 t^2}
\end{align*}
With respect to the second term we have that for $r = r^\ast$:
\begin{equation*}
a_1 t + a_2 t^2 = \frac{1}{8} t \left(7 t+20 \sqrt{2}-52\right) < 0, \quad \mbox{ for } t\in (0, W_0-Z_0] = (0, 8-4\sqrt{2}].
\end{equation*}

Regarding $a_0$, we have that $\lim_{r \to r^\ast} a_0 = 0$, so we need a more delicate analysis to show that it is negative for $r$ close to $r^\ast$ from below. It suffices to show that
\begin{equation*}
(2 r + 5) \sqrt{r^2-44 r+92} > 2 r^2 + 59 r - 110
\end{equation*}
sufficiently close to $r^\ast$. Squaring and rearranging and taking out a factor $392$, we are left to show that $-r^3 - 9 r^2 + 35 r - 25 > 0$ for $r$ sufficiently close to $r^\ast$ from below. The value of such polynomial at $r^\ast$ is $0$, but the value of its derivative is $20(-5 + 3\sqrt{2}) < 0$, so this shows that it is positive in some interval $(r^\ast - \eps, r^\ast )$ for some $\eps > 0$ sufficiently small.

\, \\
\textbf{Part II.2. Showing that the solution lies in $N_W < 0$.} \\

Let us show moreover that the solution has to lie in the region where $N_W < 0$. This follows again the strategy from \cite{CaoLabora-GomezSerrano-Shi-Staffilani:nonradial-implosion-compressible-euler-ns-T3R3}. $N_W = 0$ is a hyperbola and one of its two branches can be parametrized by $Z$ as follows 
\begin{equation*}
p_W(Z) = \frac{1}{13} \left(\sqrt{16 r^2+8 r Z+92 Z^2}-4 r-Z\right).
\end{equation*}
It satisfies the following properties: \begin{itemize}
\item $N_W(W, Z) < 0$ if $W>p_W(Z)$, in other words, $N_W < 0$ for points to the right of the branch.
\item The minimum of $p_W(Z)$ is $0$ at $Z = 0$. Thus, the leftmost point of the branch is $(0, 0)$. Moreover, $p_W(Z)$ is decreasing for $Z < 0$ and increasing afterwards.
\item $(p_W(Z), Z)$ lies on $S > 0$ (that is $p_W(Z) > Z$) for $Z < 0$.
\end{itemize}
See Figure \ref{fig:partII3-3} for a picture of $N_W = 0$.

\begin{figure}[h]
\centering
\includegraphics[width=0.5\textwidth]{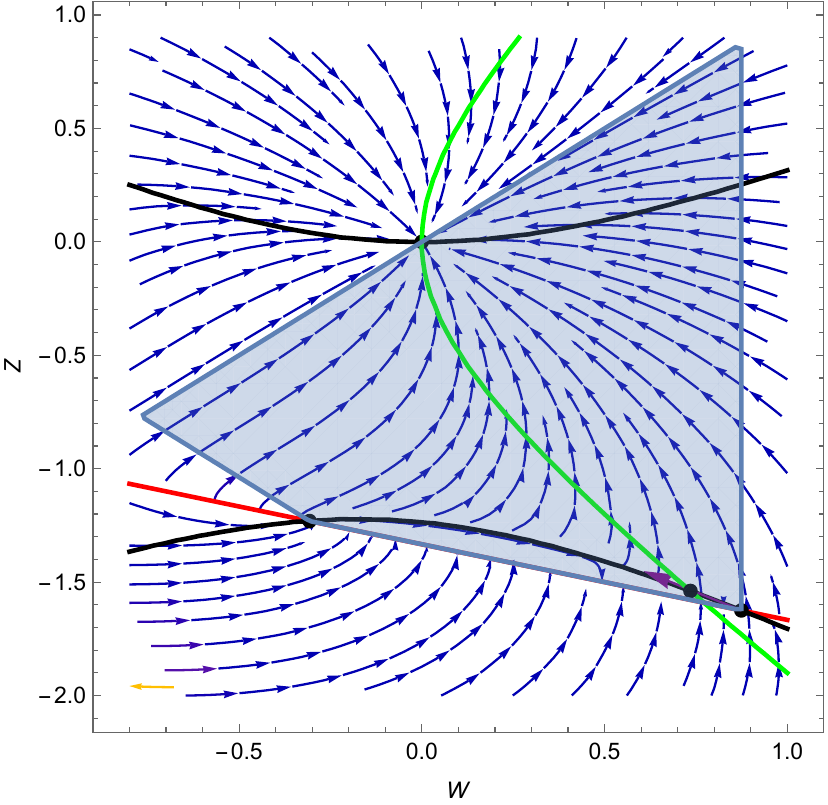}
\caption{Phase portrait in $(W, Z)$ coordinates for $r = 2.01$ (with fixed $d = 8, p = 3$). The green line represents $N_W = 0$ and we have also marked the point $P_\star$ which is the only solution to $N_W = N_Z = 0$ in the semiplane $W>Z$. The shaded quadrilateral corresponds to the restrictions $W < W_0$, $W>Z$, $D_Z > 0$, $U > U(\bar P_s)$.}
\label{fig:partII3-3}
\end{figure}

Defining
\begin{equation*}
   P_i = (W_i, Z_i) = \left( \frac{1}{26} \left(3 \sqrt{9 r^2-20 r+92}-9
   r+10\right),\frac{1}{26} \left(-\sqrt{9 r^2-20 r+92}+3
   r-38\right) \right),
\end{equation*}
which is the intersection of $(p_W(Z), Z)$ with $D_Z = 0$, one just needs to show that the solution cannot traverse $(p_W(Z), Z)$ for $Z \in (Z_i, 0)$ (we refer again to \cite{CaoLabora-GomezSerrano-Shi-Staffilani:nonradial-implosion-compressible-euler-ns-T3R3} for details and Figures illustrating the situation).

Since $N_W = 0$ over $(p_W(Z), Z)$, the barrier condition just reads $(0, N_Z / D_Z) \cdot (1, -p_W'(Z)) > 0$, so using that $D_Z > 0$ and that $p_W$ is decreasing for $Z < 0$, we just need to show $N_Z > 0$ over $(p_W(Z), Z)$ for $Z \in (Z_i, 0)$. 

The only solution to $N_W = N_Z = 0$ in the open halfplane $W > Z$ is given by
\begin{equation}
(W_\ast, Z_\ast) = \left( \frac{1}{5} \left(2 \sqrt{2}-1\right) r,-\frac{1}{5}
   \left(1+2 \sqrt{2}\right) r \right).
\end{equation}
It is obvious that $Z_\ast < 0$, and moreover we have that 
\begin{align*}
Z_\ast - Z_i \Big|_{r = r^\ast} = 0, \quad \mbox{ and } \quad \frac{d}{dr} (Z_\ast - Z_i ) \Big|_{r = r^\ast} =  - \frac{6 (113 + 111\sqrt{2} )}{1915} < 0
\end{align*}
Therefore, when $r$ is sufficiently close to $r^\ast$ from below, we have that $Z_i < Z_\ast < 0$. The sign of $N_Z(p_W(Z), Z)$ is constant for $Z \in (Z_\ast, 0)$ and moreover positive due to the expression $N_Z(p_W(Z), Z) = -rZ + O(Z^2)$. Thus, we just need to rule out the solution traversing $(p_W(Z), Z)$ from right to left for $Z \in (Z_i, Z_\ast )$.

We argue by contradiction. We define the region $\Delta$ determined by $N_Z < 0$, $N_W > 0$, $D_Z > 0$ and check that if the solution was going to traverse $(p_W(Z), Z)$ for $Z \in (Z_i, Z_\ast)$ (thus entering $\Delta $) it would not be able to exit $\Delta$ from any of the three sides. We refer to \cite[Lemma 4.1, Part II.1]{CaoLabora-GomezSerrano-Shi-Staffilani:nonradial-implosion-compressible-euler-ns-T3R3} for the detailed argument.

Once equipped with the knowledge that the trajectory lies on $N_W < 0$, we divide the proof in showing:
\begin{equation} \label{eq:part2_1}
1 + \frac{\bar U_R}{R} + \alpha \p_R \bar S > 0, \quad \mbox{ and } \quad
1 + \frac{\bar U_R}{R} - \alpha \p_R \bar S > 0.
\end{equation}
Going to $(U, S)$ variables, equation \eqref{eq:part2_1} can be written as
\begin{equation*}
\left( 1 + \frac{ W +  Z}{2}
\pm \alpha \frac{W - Z}{2} \right) \pm \frac{\alpha}{2} \left( \frac{N_W}{D_W} - \frac{N_Z}{D_Z} \right) > 0
\end{equation*}
(one inequality per choice of $\pm$). Given that the left parenthesis is $D_W$ or $D_Z$ (depending on the sign), and both $D_W, D_Z > 0$, it suffices to show
\begin{align} \label{part22}
\Xi_1 := D_W^2 D_Z + \frac{\alpha}{2} N_W D_Z - \frac{\alpha}{2} N_Z D_W &> 0 \,, \mbox{ and } \\
\label{part12}
 D_Z^2 D_W + \frac{\alpha}{2} N_Z D_W - \frac{\alpha}{2} N_W D_Z &> 0 \,,
\end{align}

where we remind the reader that because of our choice of parameters (specifically $\gamma = 2$), we have that $\alpha = \frac12$.

\, \\
\textbf{Part II.3 Showing \eqref{part22}} \\

First, we express $\Xi_1$ in $(U, S)$ coordinates, obtaining
\begin{equation*}
\Xi_1 = -\frac{1}{4} S (r (U+2)+U (7 U+6)-2)-\frac{1}{4} S^2
   (U+1) + (U+1)^3.
\end{equation*}

There are two real solutions $S_-(U) <  S_+(U)$ solving $\Xi_1 = 0$. Since $U = \frac{W+Z}{2} = \frac{D_Z + D_W}{2} - 1 > -1$, we have that $\Xi_1$ is positive when $S_- < S < S_+$. Since $S>0$ for our solution ($W > Z$), we just need to check $S < S_+(U)$.

We will show that the solution lies in the region $U > U(\bar P_s)$ (we defer that proof to the end of Part II.3). Thus, the solution lies inside the open quadrilateral region $Q$ determined by $D_Z > 0$, $W < W_0$, $W-Z > 0$, $U > U(\bar P_s)$. In Figure \ref{fig:partII3-3}, $Q$ corresponds to the shaded region. We will show $\Xi_1 |_{\partial Q} \geq 0$. This means that $S \leq S_+(U)$ for every point of $\partial Q$, implying that $S < S_+(U)$ on $Q$, and concluding the proof of \eqref{part22}. 

The side $W=Z$ obviously satisfies $S < S_+(U)$ since $S_+(U) > 0$.

Restricting $\Xi_1$ to $W = W_0$ we obtain a third-degree polynomial in $Z$ vanishing at $P_s$ (due to $N_Z(P_s) = D_Z(P_s) = 0$). We have:
\begin{align*}
\frac{\Xi_1}{Z-Z_0} = \frac{1}{4}   (r-Z_0 (Z_0+2)-2) 
+\frac{1}{16} (r-15 Z_0 - 23) (Z - Z_0)
+\frac{5}{16}
   (Z-Z_0)^2
\end{align*}
All the coefficients on the right hand side are positive for $r$ sufficiently close to $r^\ast$, since they converge to $-3+\frac{9}{2\sqrt{2}}$, $\frac{1}{16} (17 - 10\sqrt{2})$ and $\frac{5}{16}$ respectively as $r\to r^\ast$ (and all those numbers are positive).

Regarding the sign of $\Xi_1$ on the side $D_Z = 0$ of $Q$, we note that on this side $\Xi_1 = \frac{-\alpha}{2} N_Z D_W$, which has the sign of $-N_Z$. Solving $D_Z = 0$, we have that $N_Z$ over the line $D_Z = 0$ is the second degree polynomial $N_Z(W, -(4+W)/3)$, which vanishes at $\bar P_s$ and $P_s$ (solutions of $N_Z = D_Z = 0$). Since the quadratic coefficient is $\frac79 > 0$, we have that the second-degree polynomial is negative between the roots and positive outside.

Lastly, we need to treat the side $U = U(\bar P_s)$ of our quadrilateral region $Q$. We know that $S = S_+(U)$ at $\bar P_s$, since $\Xi_1 = 0$ there ($D_Z = N_Z = 0$). Therefore, we have that $S < S_+(U)$ for any other point along the line $U = U(\bar P_s)$ with smaller value of $S$. 

Let us finally show that $U > U(\bar P_s)$. We construct the barrier $b(t) = ( \bar W_0 - t, \bar Z_0 + t)$ and in order to ensure that the solution remains on its upper-right part we need to show that the following expression is positive:
\begin{align*}
\Xi_3 &= N_W(b(t)) D_Z(b(t)) +  N_Z(b(t)) D_W(b(t)) \\
& = \frac{1}{49} t \left( t (35 r-14 \mathcal{R}_1-133 ) +6 r^2+6 r \mathcal{R}_1+58 r-6 \mathcal{R}_1-64 \right).
\end{align*}
We just need to show that the parenthesis is positive, and since it is affine, it suffices to do so at the endpoints $t = 0$ and $t = \frac{\bar W_0 - \bar Z_0 }{2}$. We have that
\begin{align*}
&\left(  6 r^2+6 r \mathcal{R}_1+58 r-6 \mathcal{R}_1-64 \right) \Big|_{r=r^\ast} = -984+764\sqrt{2} > 0 \\
&\left( \frac{\bar W_0 - \bar Z_0 }{2} (35 r-14 \mathcal{R}_1-133 ) +6 r^2+6 r \mathcal{R}_1+58 r-6 \mathcal{R}_1-64 \right) \Big|_{r=r^\ast} = -492 + 382 \sqrt{2} > 0.
\end{align*}

\, \\
\textbf{Part II.4. Showing \eqref{part12}.} \\ 
Multiplying by $D_Z D_W$, we reduce the proof of \eqref{part12} to show positivity for
\begin{align*}
\Xi_2 := N_Z D_W - N_W D_Z  = S \left( -\frac{S^2}{2}+U (9 U+10) + r (U+2) \right).
\end{align*}

In particular, in the halfplane $S > 0$ we have that $\Xi_2 > 0$ whenever $$S < S_+(U) = \sqrt{2r(2+U) + 2U(10 + 9U)}.$$

Note that this definition of $S_+(U)$ is different from the one we gave in Part II.2. Each definition applies to its respective section. Note that $S_+(U)$ may be complex, and in that case there is no possible $S$ such that $\Xi_2 \geq 0$. 

We divide our argument in two regions $U \geq U(P_s)$ and $U < U(P_s)$. Regarding $U \geq U(P_s)$, we start noting that $\Xi_2 \geq 0$ along the line $D_Z = 0$, for points with $U \geq U(P_s)$. This comes from the observation that $\Xi_2 = -N_Z D_W$ along $D_Z = 0$, the fact that $D_W > 0$ along $D_Z = 0$, $W > Z$ and the observation from Part II.3 that, along $D_Z = 0$, $N_W$ is positive between its two roots ($P_s$ and $\bar P_s$) and $N_Z < 0$ outside. Denoting $S_{DZ}(U) = 2(1+U)$ to be the value of $S$ such that $(U, S)$ lies in the nullset of $D_Z$, we have that $S_{DZ(U)} \leq S_+ (U)$, for $U \geq U(P_s)$ with equality only at $U = U(P_s)$. Since the solution will be in $D_Z > 0$ for $R > 1$, we have that $S < S_{DZ}(U)$ from the condition $D_Z > 0$.

\begin{figure}[h]
\centering
\includegraphics[width=0.5\textwidth]{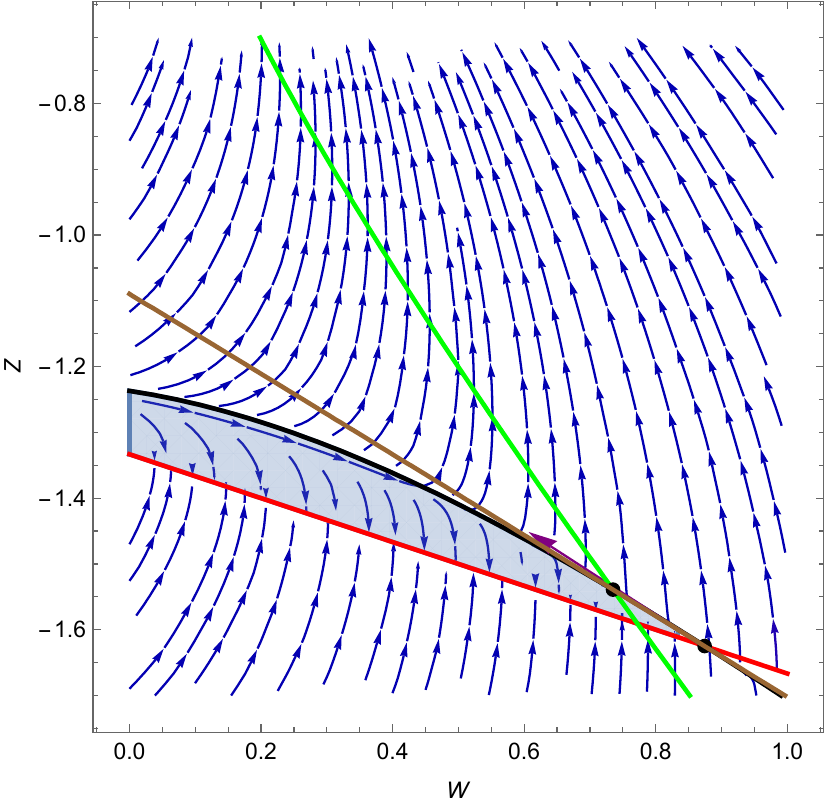}
\caption{Phase portrait in $(W, Z)$ coordinates for $r = 2.01$ (with out fixed $d = 8, p = 3$). The brown line represents $\Xi_2 = 0$, which is also $b(t)$ for $t \geq 0$ when looking at the part that lies to the left of $P_s$. $P_\star$ is the intersection of the green and black curves ($N_W = 0$ and $N_Z = 0$ respectively). We have shaded the region $D_Z > 0, N_Z < 0$. The brown curve enters the shaded region at $P_s$ and exits the region at $P_\star$ (difficult to distinguish the brown and black curves in the picture since they are almost coincident between $P_s$ and $P_\star$).}
\label{fig:partII4}
\end{figure}

Regarding $U \leq U(P_s)$, we denote by $b(t)$ the conic branch of $\Xi_2 = 0$, fixing $b(0) = P_s$, and with the orientation fixed by $D_Z(b(t)) > 0$ for small $t > 0$. We will show:
\begin{itemize}
\item $b(t_\star ) = P_\star$ for some $t_\star > 0$ (recall $P_\star$ is a solution to $N_Z = N_W = 0$ defined at the beginning of the proof).
\item $b(t)$ is in the region $N_Z < 0$, $D_Z > 0$, for $t\in (0, t_\star )$.
\item The solution does not traverse $b(t)$
\end{itemize}
This would conclude the proof, since we can use $N_W = 0$ to connect $P_\star$ with $P_\infty$, and we know that the solution remains in $N_W < 0$ due to Part II.2 (in particular, $\Xi_2 > 0$, since the only intersection of $\Xi_2$ with $N_W = 0$ [under $D_W > 0$, $W>Z$] is given by $P_\star$).

Regarding the first two items, we note that $b(t)$ starts in the region $D_Z > 0$ from its orientation. Since $N_W(P_s) < 0$ (Lemma \ref{lemma:auxiliary}) and $D_W(P_s) > 0$, from the expression of $\Xi_2 = N_Z D_W - N_W D_Z$, $b(t)$ has to lie in the region $N_Z < 0$ for $t \geq 0$ sufficiently small. We denote the region defined by $N_Z < 0, D_Z > 0$ by $E$. $E$ is bounded, since its boundary is described by the segment of $D_Z = 0$ from $P_s$ to $\bar P_s$ and the piece of $N_Z = 0$ also going from $P_s$ to $\bar P_s$.  $E$ is contained on the halfplane $W > Z$ since the hyperbola $N_Z = 0$ only intersects $W = Z$ twice: at $(0, 0)$ and $(-r, -r)$. The intersection at $(0, 0)$ corresponds to the upper branch (which is not the one bounding $E$) and the intersection $(-r, -r)$ is located in $D_Z < 0$ (thus, it is not on the boundary of $E$. Moreover, $b(t)$ has to exit $E$ at some point (if $b(t)$ is unbounded that is obvious, if it is an ellipse, it needs to reach $P_s$ from the side $D_Z < 0$). 

Now, we show that the first point of exit of $b(t)$ from $E$ is actually $P_\star$, thus proving the first two items that we claimed. If $b(t)$ exits through a point with $D_Z > 0$ (and therefore $N_Z = 0$) of $\partial E$, then from $\Xi_2 = 0$, we conclude $N_W D_Z = 0$, so $N_W = 0$. Thus, that point of exit satisfies $N_W = N_Z = 0$, and since $E$ is contained in the semiplane $W > Z$, must be $P_\star$ (the only solution to $N_W = N_Z = 0$ on that semiplane). Therefore, we just need to discard the possibility that $b(t)$ exits $E$ through $D_Z = 0$. In such case, one would have $\Xi_2 = N_Z D_W = 0$ at the point of exit, so $N_Z = 0$, meaning that the exit point is $\bar P_s$ (the only other solution to $N_Z = D_Z = 0$ apart from $P_s$). However, we see from Lemma \ref{lemma:auxiliary} that $N_W (P_s) < 0$. Since $\lim_{r \to r^\ast} N_W(\bar P_s) = \frac{60}{49} \left(13 \sqrt{2}-17\right) = 1.69\ldots$ has a different sign, if $b(t)$ goes from $P_s$ to $\bar P_s$, it would need to cross $N_W = 0$ at some intermediate time. That is impossible while in the region $E$ defined by $D_Z > 0, N_Z < 0$, since it would imply $\Xi_2 = N_Z D_W = 0$. We conclude that $b(t)$ exits $E$ through point $P_\star$ at some time $t_\star$ and remains in $E$ for $t \in (0, t_\star)$.

Lastly, let us show that the solution does not traverse $b(t)$. Given the implicit definition of $b(t)$ as $\Xi_2 = 0$, a normal vector of $\Xi_2$ pointing down-left is given by $\nabla (N_W D_Z - N_Z D_W)$. We have that
\begin{equation*}
\left( -1, -1 \right) \cdot \nabla (N_W D_Z - N_Z D_W) = \frac{1}{2} (W-Z) (10+r+9W+9Z) = S \left( 10 + r + 18U \right).
\end{equation*}
We show that this quantity is positive. Since the parenthesis is increasing with $U$ it suffices to show that it is positive for $U = U(P_\star)$ in Lemma \ref{lemma:auxiliary}. This concludes that in both of our cases, on the arc $\Xi_2 = 0$, the field points always in the lower-left direction.

As a conclusion, we know that the solution cannot traverse $b(t)$ for $t\in (0, t_\star)$ and cannot traverse $N_W = 0$ by Part II.2 (in particular, it cannot traverse the branch from $P_\star$ to $P_\infty$). Therefore, the concatenation of $b(t)$ and $N_W = 0$ is a curve joining $P_s$ and $P_\infty$ that the solution cannot traverse. We clearly have $S \leq S_+(U)$ along that curve ($S = S_+(U)$ along $b(t)$ and then $N_W$ does not intersect $b(t)$ again, so $S < S_+(U)$). Since we have that $S \leq S_+(U)$ at the boundary of our region, we conclude $S < S_+(U)$ in the interior, and therefore the solution satisfies $\Xi_2 > 0$.

\end{proof}

\begin{lemma} \label{lemma:auxiliary} For $(d, \gamma ) = (8, 2)$ and $r$ sufficiently close to $r^\ast$ from below, we have: \begin{itemize}
\item $W_1 + Z_1 < 0$.
\item $N_W(P_s) < 0$. 
\end{itemize}
\end{lemma}
\begin{proof}
\textbf{Item 1} \\
We have that 
\begin{equation*}
\frac{7 r (-29 r+29 \mathcal R_1+16)-448 \mathcal R_1+\sqrt{2} \mathcal R_2-1624}{147 (r-\mathcal R_1-8)}
\end{equation*}
and since $\lim_{r\to r^\ast}  (r-\mathcal R_1-8) = 14(-2+\sqrt{2}) < 0$, we just need to show that the numerator is positive. Since $\sqrt{2} \mathcal R_2$ is trivially positive, we just need to show that it dominates the rest. That is, we need to show $A > 0$ for 
\begin{align*}
A &= (\sqrt{2} \mathcal R_2)^2 - \left( 7 r (-29 r+29 \mathcal R_1+16)-448 \mathcal R_1-1624 \right)^2 \\ 
&= -4704 (r-1) \left(725 r-1070 + 6 r^3-4 r^2+(85-6r^2-128 r ) \mathcal R_1\right)
\end{align*}
Therefore, we just need to show that the right parenthesis is negative. Since $r > 1$, $\mathcal R_1 > 0$, it is obvious that $(85-6r^2-128 r ) \mathcal R_1 < 0$, so we just need to show that this term dominates the other one. That is, we need to show that $B > 0$ for
\begin{equation*}
B = ((85-6r^2-128 r ) \mathcal R_1)^2 - (725 r-1070 + 6 r^3-4 r^2)^2 = -19208 (r-1) (3 r+1) (r^2 + 10r -25).
\end{equation*}
Therefore, the conclusion follows from $r^2 + 10r - 25 < 0$, which holds because $r^\ast = 5(\sqrt{2}-1)$ is the upper root of the polynomial $r^2 + 10r - 25$. \\

\textbf{Item 2} \\ 
We have that 
\begin{equation*}
N_W(P_s) = \frac{2}{49}\left(  \left(2 r^2+59 r-110\right) -  (2 r+5) \mathcal{R}_1 \right)
\end{equation*}
Therefore, in order to show that it is negative, it suffices to show that the term $(2r+5)\mathcal{R}_1$ dominates the parenthesis. Therefore, we just need to show $C > 0$ for 
\begin{equation*}
C = \left( (2 r+5) \mathcal{R}_1 \right)^2 -  \left(2 r^2+59 r-110\right)^2 = -392 (r-1) (r^2 + 10 r-25)
\end{equation*}
Therefore, we see that $C > 0$ using again that $r < r^\ast = 5(\sqrt{2}-1)$ and noticing that $r^\ast$ is the upper root of the polynomial $r^2 + 10 r-25$.

\end{proof}

\begin{lemma} \label{lemma:integrated_repulsivity} We have that the profiles $(\bar U, \bar S)$ satisfy
\begin{equation}
R + \bar U_R - \alpha \bar S > (R-1)\eta
\end{equation}
for $R > 1$.
\end{lemma}
\begin{proof} From \cite{Merle-Raphael-Rodnianski-Szeftel:implosion-i}, \cite{Buckmaster-CaoLabora-GomezSerrano:implosion-compressible} we have that the profiles constructed satisfy $R+\bar U_R - \alpha \bar S = 0$ for $R=1$ (i.e. $D_Z(P_s) = 0$). 
Therefore, using the fundamental theorem of calculus and \eqref{eq:radial_repulsivity}, we have
\begin{align*}
R+\bar U_R - \alpha \bar S = \int_1^R \left( \tilde R + \bar U_R' (\tilde R)  - \alpha \bar S' (\tilde R) \right) d\tilde R > \int_1^R \eta d\tilde R.
\end{align*}
\end{proof}

\bibliographystyle{plain}
\bibliography{references.bib}


\begin{tabular}{l}
  \textbf{Gonzalo Cao-Labora} \\
  {Courant Institute} \\
  {New York University} \\
  {251 Mercer Street, 619} \\
  {New York, NY 10012, USA} \\
  {Email: gc2703@nyu.edu} \\ \\
  \textbf{Javier G\'omez-Serrano}\\
  {Department of Mathematics} \\
  {Brown University} \\
  {314 Kassar House, 151 Thayer St.} \\
  {Providence, RI 02912, USA} \\
  {Email: javier\_gomez\_serrano@brown.edu} \\ \\
  \textbf{Jia Shi} \\
  {Departament of Mathematics} \\
  {Massachusetts Institute of Technology} \\
  {182 Memorial Drive, 2-157} \\
  {Cambridge, MA 02139, USA} \\
  {Email: jiashi@mit.edu} \\ \\
  \textbf{Gigliola Staffilani} \\
  {Departament of Mathematics} \\
  {Massachusetts Institute of Technology} \\
  {182 Memorial Drive, 2-251} \\
  {Cambridge, MA 02139, USA} \\
  {Email: gigliola@math.mit.edu} \\
\end{tabular}

\end{document}